\pdfoutput=1
\documentclass[a4paper, 10pt]{amsart}
\usepackage{
 amsmath,   amscd,    amssymb,   mathrsfs,
 mathtools, fullpage, enumerate, thmtools,
 fancyhdr,  stmaryrd, tikz-cd,   etoolbox,
 chngcntr,
 tabu,
 diagrams,
 rotating,
 pdflscape,
 xspace
}
\usepackage[shortlabels]{enumitem}
\usepackage[l2tabu, orthodox]{nag}

\setlist{leftmargin=2\parindent}

\usepackage[pagebackref, hypertexnames=false, bookmarks=false]{hyperref}
\usepackage[capitalise]{cleveref}

\newenvironment{smatrix}{\left( \begin{smallmatrix} } {\end{smallmatrix} \right) }

\newcommand{\stbt}[4]{\begin{smatrix}#1 & #2 \\ #3 & #4\end{smatrix}}

\theoremstyle{plain}
\newtheorem{theorem}{Theorem}[subsection]
\newtheorem{introtheorem}{Theorem}

\newtheorem{lemma}[theorem]{Lemma}
\newtheorem{proposition}[theorem]{Proposition}
\newtheorem{corollary}[theorem]{Corollary}
\newtheorem{definition}[theorem]{Definition}
\newtheorem{notation}[theorem]{Notation}

\newtheorem{assumption}[theorem]{Assumption}

\theoremstyle{remark}
\declaretheorem[name=Remark, sibling=theorem, qed={\lower-0.3ex\hbox{$\diamond$}}]{remark}
\declaretheorem[name=Note, sibling=theorem, qed={\lower-0.3ex\hbox{$\diamond$}}]{note}

\makeatletter
\newcommand{\colim@}[2]{%
 \vtop{\m@th\ialign{##\cr
   \hfil$#1\operator@font colim$\hfil\cr
   \noalign{\nointerlineskip\kern1.5\ex@}#2\cr
   \noalign{\nointerlineskip\kern-\ex@}\cr}}%
}
\newcommand{\colim}{%
 \mathop{\mathpalette\colim@{\rightarrowfill@\scriptscriptstyle}}\nmlimits@
}
\renewcommand{\varprojlim}{%
 \mathop{\mathpalette\varlim@{\leftarrowfill@\scriptscriptstyle}}\nmlimits@
}
\makeatother

\DeclareMathOperator{\AJ}{AJ}
\DeclareMathOperator{\BGG}{BGG}
\DeclareMathOperator{\DR}{DR}
\DeclareMathOperator{\Fil}{Fil}
\DeclareMathOperator{\GL}{GL}
\DeclareMathOperator{\Gal}{Gal}
\DeclareMathOperator{\Gr}{Gr}

\DeclareMathOperator{\Res}{Res}

\DeclareMathOperator{\ES}{ES}
\DeclareMathOperator{\Iw}{Iw}
\DeclareMathOperator{\Pz}{Pz}
\DeclareMathOperator{\Spec}{Spec}
\DeclareMathOperator{\Sym}{Sym}

\DeclareMathOperator{\coh}{coh}
\DeclareMathOperator{\AF}{AF}
\DeclareMathOperator{\crit}{crit}

\DeclareMathOperator{\fp}{fp}
\DeclareMathOperator{\loc}{loc}
\DeclareMathOperator{\mot}{mot}

\DeclareMathOperator{\pr}{pr}

\DeclareMathOperator{\syn}{syn}
\DeclareMathOperator{\tr}{tr}
\DeclareMathOperator{\cl}{cl}
\DeclareMathOperator{\CG}{CG}

\DeclareMathOperator{\As}{As}

\DeclareMathOperator{\RG}{R\Gamma}

\newcommand{\imp}{\mathrm{imp}}
\newcommand{\f}  {\mathrm{f}}
\newcommand{\m}  {\mathrm{m}}

\newcommand{\sph}{\mathrm{sph}}

\newcommand{\dR} {\mathrm{dR}}
\newcommand{\et} {\text{\textup{\'et}}}

\newcommand{\ord}{\mathrm{ord}}
\renewcommand{\ss} {\mathrm{ss}}
\newcommand{\rig}{\mathrm{rig}}
\newcommand{\Eis}{\mathrm{Eis}}

\newcommand{\fN}{\mathfrak{N}}

\renewcommand{\AA}{\mathbb{A}}
\newcommand{\CC}{\mathbb{C}}
\newcommand{\MM}{\mathbb{M}}
\newcommand{\QQ}{\mathbb{Q}}
\newcommand{\RR}{\mathbb{R}}
\newcommand{\VV}{\mathbb{V}}
\newcommand{\YY}{\mathbb{Y}}
\newcommand{\XX}{\mathbb{X}}
\newcommand{\ZZ}{\mathbb{Z}}

\newcommand{\Af}{\AA_\f}

\newcommand{\Qp}{\QQ_p}
\newcommand{\QQbar}{\overline{\QQ}}
\newcommand{\Zp}{\ZZ_p}

\newcommand{\Pif}{\Pi_\f}

\newcommand{\RGt}{\widetilde{R\Gamma}}

\newcommand{\Dcris}{\mathbb{D}_{\mathrm{cris}}}
\newcommand{\DdR}{\mathbb{D}_{\dR}}
\newcommand{\Nek}{Nekov\'a\v{r}\xspace}
\newcommand{\Niz}{Nizio\l\xspace}

\newcommand{\bfj}{\mathbf{j}}

\newcommand{\cA}{\mathcal{A}}

\newcommand{\cE}{\mathcal{E}}
\newcommand{\cF}{\mathcal{F}}

\newcommand{\cH}{\mathcal{H}}

\newcommand{\cO}{\mathcal{O}}

\newcommand{\cS}{\mathcal{S}}

\newcommand{\cU}{\mathcal{U}}
\newcommand{\cV}{\mathcal{V}}
\newcommand{\cW}{\mathcal{W}}
\newcommand{\cX}{\mathcal{X}}
\newcommand{\cY}{\mathcal{Y}}

\newcommand{\cLPR}{\mathcal{L}_{\mathrm{PR}}}

\newcommand{\sX}{\mathscr{X}}

\newcommand{\fP}{\mathfrak{P}}

\newcommand{\into}{\hookrightarrow}

\newcommand{\sC}{\mathscr{C}}

\newcommand{\uPi}{\underline{\Pi}}

\newcommand{\uet}{\underline{\eta}}

\newcommand{\wH}{\widetilde{H}}
\newcommand{\pp}{\mathfrak{p}}

\numberwithin{equation}{section}
\renewcommand{\le}{\leqslant}
\renewcommand{\leq}{\leqslant}
\renewcommand{\ge}{\geqslant}
\renewcommand{\geq}{\geqslant}

\title{Asai--Flach classes, $p$-adic L-functions, and the Bloch--Kato conjecture for $\operatorname{GO}(4)$}
\author{Giada Grossi}
\address[G.~Grossi]{CNRS, Institut Galilée, Université Sorbonne Paris Nord, 93430 Villetaneuse, France}
\email{grossi@math.univ-paris13.fr}
\urladdr{\href{https://orcid.org/0000-0002-0883-4494}{0000-0002-0883-4494}}

\author{David Loeffler}
\address[D.~Loeffler]{Faculty of Mathematics and Computer Science, UniDistance Suisse, Schinerstrasse 18, 3900 Brig, Switzerland}
\email{david.loeffler@unidistance.ch}
\urladdr{\href{http://orcid.org/0000-0001-9069-1877}{0000-0001-9069-1877}}

\author{Sarah Livia Zerbes}
\address[S.L.~Zerbes]{Department of Mathematics, ETH Z\"urich, R\"amistrasse 101, 8092 Z\"urich, Switzerland}
\email{sarah.zerbes@math.ethz.ch}
\urladdr{\href{http://orcid.org/0000-0001-8650-9622}{0000-0001-8650-9622}}

\thanks{Supported by European Research Council Consolidator Grant No. 101001051 (\textsc{ShimBSD}) (Loeffler) and US National Science Foundation Grant No. DMS-1928930 (all authors)}

\begin{document}
 \maketitle
%\vspace{-10cm}

% dots in the TOC
\makeatletter
\patchcmd{\@tocline}
{\hfil}
{\leaders\hbox{\,.\,}\hfil}
{}{}
\makeatother

\setcounter{tocdepth}{1}
\tableofcontents

\section{Introduction}

 \subsection{Overview}

  The Bloch--Kato conjecture, which relates the dimensions of Selmer groups of geometric $p$-adic Galois representations to the order of vanishing of their $L$-functions, is one of the outstanding open problems of number theory. One of the most successful approaches to this problem so far has been via the method of Euler systems: norm-compatible families of global cohomology classes for Galois representations, which serve to give upper bounds for Selmer groups. These are related to $L$-functions by so-called \emph{explicit reciprocity laws}, relating the localisation of the Euler system classes at $p$ to the values of the $L$-function. This can be used as a criterion for the non-vanishing of Euler systems, a crucial input in applications to the Bloch--Kato conjecture (and related problems such as the Iwasawa Main Conjecture).

  One strategy for proving such explicit reciprocity laws, introduced in the seminal paper \cite{BDP13} in the context of Heegner points, proceeds as follows. Supposing we already have an Euler system for some Galois representation $V$, we aim to carry out the following steps:
  \begin{enumerate}[(1)]
   \item Construct a $p$-adic $L$-function, for some $p$-adic family of Galois representations specialising to $V$, whose values in a suitable interpolation range are critical values of the corresponding complex $L$-functions.
   \item Relate values of this $p$-adic $L$-function at suitable points \emph{outside} the range of interpolation to Bloch--Kato logarithms of Euler system classes.
   \item Via analytic continuation, deduce a relation between Euler systems and $p$-adic $L$-values at points \emph{inside} the interpolation range, where the $p$-adic $L$-function specialises to the complex $L$-value.
  \end{enumerate}

  This strategy was successfully carried out for the Euler system of Beilinson--Flach elements (attached to a Rankin--Selberg convolution of modular forms) in \cite{BDR-BeilinsonFlach, KLZ17}; and more recently in the $\operatorname{GSp}_4$ setting in \cite{LPSZ1, LZ20}. In the present paper we focus on the \emph{Euler system of Asai--Flach classes}, introduced in \cite{leiloefflerzerbes18}, which is an Euler system for the Asai, or twisted tensor, Galois representation associated to a Hilbert modular form for a real quadratic field. (These Galois representations are naturally essentially self-dual, factoring through the orthogonal group $\operatorname{GO}(4)$. Conversely, it follows from the Fontaine--Mazur conjecture that all 4-dimensional, irreducible, essentially self-dual geometric representations of $\Gal(\QQbar / \QQ)$ with distinct Hodge--Tate weights arise either as Rankin--Selberg convolutions, as spin representations of $\operatorname{GSp}(4)$, or as Asai representations of quadratic Hilbert modular forms, so we have Euler system constructions in all cases.)

 In this setting, step (1) of the above strategy was carried out in the prequel paper \cite{grossiloefflerzerbesLfunct}, building on the higher Hida theory for Hilbert modular groups developed in \cite{grossi21}. The goal of the present paper is to carry out Steps (2) and (3). As applications, we prove one inclusion of the cyclotomic Iwasawa Main Conjecture and the  Bloch--Kato conjecture for the Asai $L$-functions of Hilbert modular forms over real quadratic fields (under various technical hypotheses, see below).

 \subsection{Main results}

  The main results of this paper are the following:

  \begin{introtheorem}
   \label{intro:BKIwapplication}
   Let $\Pi$ be an automorphic representation of $\Res_{F / \QQ} \GL_2$ of weight $(k_1 + 2, k_2 + 2)$, for some $k_1 > k_2 \ge 0$ with $k_1 = k_2 \bmod 2$, and level $\fN$, and $p > 2$ a prime. Suppose that
   \begin{itemize}
    \item $p$ is split in $F$,
    \item $\Pi$ is unramified and ordinary at the primes above $p$,
    \item the Galois representation satisfies a ``big image'' condition (see \cref{assPi}),
    \item the local Euler factor $P_p(\Pi, X)$ does not vanish at $p^j$ for any $j \in \ZZ$.
   \end{itemize}

   Then:
   \begin{itemize}
    \item (\cref{thm:BK}) For each $j$ with $k_2 + 1 \le j \le k_1$, the Bloch--Kato Selmer group of $V_p^{\As}(\Pi)^*(-j)$ is zero.
    \item (\cref{thm:IMC}) The characteristic ideal of the dual Selmer group of $\Pi$ over $\QQ(\mu_{p^\infty})$ divides $p^r \cdot L_{p, \As}(\Pi)$ for some $r \ge 0$, where $L_{p, \As}(\Pi)$ is (up to a shift) the $p$-adic Asai $L$-function for $\Pi$ constructed in \cite{grossiloefflerzerbesLfunct}.
   \end{itemize}
  \end{introtheorem}

  For more details of the notations and the definitions of the objects involved, we refer to the body of the paper.

  We briefly survey the main steps towards \cref{intro:BKIwapplication}. The first main ingredient is a \emph{regulator formula}, corresponding to Step (2) in the strategy outlined above. (This is true under somewhat more general hypotheses than \cref*{intro:BKIwapplication}; we refer to the body of the paper for the precise assumptions.)

  \begin{introtheorem}[{\cref{thm:regulator}}]
   \label{intro:regformula}
   For each $j \in \ZZ$ with $0 \le j \le \min(k_1, k_2)$, the Bloch--Kato logarithm of the Asai--Flach class $\AF^{[\Pi, j]} \in H^1(\QQ, V_p^{\As}(\Pi)^*(-j))$ can be expressed in terms of the $p$-adic $L$-value $L_{p, \As}\left(\Pi, 1 + j\right)$.
  \end{introtheorem}

  The proof of this theorem relies on the use of higher Coleman theory (as developed in \cite{boxerpilloni22}, and generalised in \cite{LZ-plectic} to impose finite-slope conditions at only one prime above $p$); and the \emph{Pozn\'an spectral sequence} relating higher Coleman theory to syntomic cohomology, first introduced in \cite{LZ20}. More precisely, we compute the pairing of the Bloch--Kato logarithm of the localisation at $p$ of the class $\AF^{[\Pi, j]}$ with a class $\eta_{\dR}$, which belongs to a certain graded piece of the middle degree de Rham cohomology of the Hilbert modular surface (see $\S$5). To do this, we lift such class to finite-polynomial cohomology in $\S$6. The construction of the $p$-adic $L$-function of \cite{grossiloefflerzerbesLfunct}, as recalled in $\S4$, involves the coherent cohomology incarnation of $\eta_{\dR}$ lifted to a class in coherent cohomology of the $\pp_1$-ordinary locus with compact support (via higher Coleman theory). Hence to prove Theorem \ref{thm:regulator} we lift the finite-polynomial cohomology class to a class over the $\pp_1$-ordinary locus in $\S7$ and, in order to exploit the liftings of both partial Frobenii, restrict it to the full ordinary locus in $\S8$, where the \emph{Pozn\'an spectral sequence} is recalled and used to describe this class in terms of classes in coherent cohomology. Finally in $\S9$, we use the explicit description of the syntomic relisation of the Eisenstein class restricted to the ordinary locus of the modular curve as well as some compatibility of (partial) unit root splittings to compute the pairing in terms of values of the $p$-adic $L$-function.

  Although the statement of \cref*{intro:BKIwapplication} involves only a single automorphic representation $\Pi$, the next steps in its proof require the use of $p$-adic families of automorphic representations (Hida families). More precisely, we consider a two-parameter family $\uPi$ of automorphic representations specialising to $\Pi$, with both components of the weight independently varying, and we prove the following theorem:

  \begin{introtheorem}[{\cref{thm:padicL}}]
   There exists a 3-variable $p$-adic Asai $L$-function for $\uPi$, interpolating all critical values of the $L$-functions for all classical specialisations of $\uPi$.
  \end{introtheorem}

  This strengthens a theorem proved in \cite{grossiloefflerzerbesLfunct} in which we obtained a 2-variable $p$-adic $L$-function, with $k_1$ and $j$ varying but $k_2$ fixed (but without requiring ordinarity at $\pp_2$).

  On the other hand, Sheth \cite{sheth25} has shown that the Euler system construction of \cite{leiloefflerzerbes18} also extends to such families; and \cref{intro:regformula} gives a relation between these objects at a set of points which is Zariski-dense in the parameter space. In order to derive from this an explicit reciprocity law relating the Euler system and $L$-function in the critical range, we need to show that the $p$-adic Eichler--Shimura comparison isomorphism, relating \'etale and de Rham cohomology, ``interpolates in Hida families'' in a suitable sense (analogous to results of Ohta \cite{ohta00} and Andreatta--Iovita--Stevens \cite{andreattaiovitastevens} for $\GL_2 / \QQ$). In our setting, we would hope for an isomorphism between two free rank 1 modules over $\cO(\cU)$, where $\cU$ is a small neighbourhood in weight space over which $\uPi$ is defined: one module $\cS_{w_1}(\uPi)$, arising from higher Hida theory, which interpolates the $\uPi$-eigenspace in coherent $H^1$ of Hilbert modular varieties; and another module $\Dcris(\MM_{\pp_1})$ which is $\Dcris$ of a rank-1 subquotient of the Asai Galois representation of $\uPi$. The Eichler--Shimura map gives canonical isomorphisms between the fibres of these modules at classical points, and we would hope for a single $\cO(\cU)$-module isomorphism having all of these ``pointwise'' isomorphisms as specialisations. What we actually prove is something a little weaker:

  \begin{introtheorem}[{\cref{thm:ERL}}]
   \label{intro:ESandERL}
   There exists an isomorphism between the above modules after base-extending to $\operatorname{Frac} \cO(\cU)$, interpolating the Eichler--Shimura isomorphisms for all but finitely many classical points.  Moreover, the composite of this isomorphism and Perrin-Riou's regulator map sends the Euler system to the $p$-adic $L$-function.
  \end{introtheorem}

  Note that the proofs of both halves of the theorem are intertwined: we \emph{use} the Euler system and regulator formula in order to derive the existence of the meromorphic Eichler--Shimura isomorphism.

  \cref{intro:ESandERL}, together with the results of \cite{leiloefflerzerbes18}, immediately imply \cref{intro:BKIwapplication} as long as $\Pi$ is not one of the finitely many ``bad'' specialisations. We conjecture that such bad specialisations do not actually exist, but we cannot rule it out with the presently available techniques\footnote{This would follow from work in progress by Birkbeck and Williams, aiming to prove an interpolation of $p$-adic Eichler--Shimura maps in higher Hida, or Coleman, families of Hilbert modular forms using perfectoid methods.}. Hence we carry out in $\S13$ a \emph{leading term argument}, adapting methods developed in \cite{LZ20-yoshida} for the standard (rather than Asai) $L$-functions of quadratic Hilbert modular forms, to show that one can recover a modified version of the Euler system which is related to the $L$-function even at ``bad'' specialisations, which suffices for the proof of \cref{intro:BKIwapplication}.

 \subsection*{Acknowledgements}

  Large parts of this paper were written while the authors were in residence at MSRI (now known as Simons Laufer Mathematical Institute) for the 2023 semester program ``Algebraic Cycles, L-Values, and Euler Systems''. We are grateful to MSRI for offering such a fantastic working environment.

%%%%%%%%%%%%%%%%%%%%%%%%%%%%%%%%%%%%%%%%
\section{Conventions}
\label{ss:conventions}
%%%%%%%%%%%%%%%%%%%%%%%%%%%%%%%%%%%%%%%%

 Throughout this paper $F$ denotes a real quadratic field, of discriminant $D$, and we fix an enumeration of the embeddings $F \into \RR$ as $\sigma_1, \sigma_2$.

 \subsection{Algebraic groups}

  As in \cite{leiloefflerzerbes18}, define $G=\Res_{F/\QQ}(\GL_2)$, and $H = \GL_2$, with $\iota: H \into G$ the natural embedding. We write $B_H$ and $B_G$ for the upper-triangular Borel subgroups. The Weyl group of $G$ is isomorphic to $(\ZZ/2\ZZ)\times(\ZZ/2\ZZ)$, with generators $w_1$ and $w_2$ (corresponding to the $\sigma_i$).

  Recall that if $V$ is an algebraic representation of $B_H$, then $V$ gives rise to a vector bundle on the Shimura variety $Y_H$ (for any sufficiently small level group), endowed with an action of Hecke correspondences. There are two possible normalisations for this functor, and we normalise so that the defining 2-dimensional representation of $H$ (restricted to $B_H$) maps to the relative de Rham \emph{cohomology} sheaf of the universal elliptic curve $\cE \to Y_H$ (rather than the relative homology, which is the other convention in use). There is an analogous construction for $G$ as long as we restrict to representations trivial on the norm-one subgroup $\{\stbt{x}{}{}{x} : N_{F/\QQ}(x) = 1\}$ of $Z_G$, cf.~\cite[\S 3.2c]{leiloefflerzerbes18}.

 \subsection{Algebraic representations and the Clebsch--Gordan map}\label{ss:ClebschGordan}

  For the group $G$, we shall always work with representations over fields containing $F$. Given two representations $V, V'$ of $\GL_2$, we write $V \boxtimes V'$ for the tensor product of $V, V'$, endowed with a $G$-action as the tensor product of the actions via  $\sigma_1$ on $V$ and via $\sigma_2$ on $V'$.

  Given integers $k_1, k_2\geq 0$, and $0\leq j\leq \min\{k_1, k_2\}$, let
  \[
   V_G \coloneqq \Sym^{k_1} W_G \boxtimes \Sym^{k_2} W_G\qquad \text{and}
   \qquad V_H \coloneqq \Sym^t W_H, \quad t = k_1 + k_2 - 2j,
  \]
  where $W_H$ and $W_G$ denote the defining 2-dimensional representations of $\GL_2 / \QQ$ and $\GL_2 / F$ respectively.

  The representation $V_H$ has a canonical basis $(v^a w^{t-a})_{0 \le a \le t}$, where $v, w$ are the two standard basis vectors of $W_H$; and is equipped with a decreasing $B_H$-stable weight filtration $\Fil^n V_H = \langle\{ v^a w^{t-a} : a \ge n\}\rangle$. Similarly, $V_G$ is equipped with a bi-filtration (a decreasing filtration indexed by $\ZZ^2$) arising from the weight vector filtrations on each factor.

  There is a non-zero morphism of $H$-representations (the \emph{Clebsch--Gordan map}), unique up to scalars,
  \[ \CG^{[k_1, k_2, j]}: V_H\to V_G\otimes \det{}^{-j}. \]
  See \eqref{eq:CGexplicit} below for explicit formulae. If we equip the one-dimensional representation $\det^{-j}$ with a filtration concentrated in degree $-j$, then this map respects the filtrations, and hence induces a map on the graded pieces.

\section{Setup: Hilbert modular forms}
%%%%%%%%%%%%%%%%%%%%%%%%%%%%%%%%%%%%%%%%

\subsection{Notations}

Let $J \in\{G,H\}$. For a neat open subgroup $K$ of $J(\Af)$, write $Y_{J}(K)$ for the Shimura variety of $J$ of level $K$, $X_J(K)$ for a toroidal compactification, and $D_J = X_J(K) - Y_J(K)$ for the boundary divisor.

A \emph{weight} for $G$ will mean a tuple of integers $(k_1, k_2; w)$ with $w = k_1 = k_2 \bmod 2$. If $\mu$ is a weight, we have a locally-free sheaf $\omega^\mu$ on $X_{G, K} / F$, whose sections are Hilbert modular forms of weight $\mu$.

\begin{note}
 As an abstract sheaf $\omega_\mu$ does not depend on $w$; but the action of Hecke operators on it does. With our conventions, if $\mathfrak{q}$ is a prime of $F$ which is trivial in the narrow ray class group modulo $Z_G \cap K$, the action of $\stbt{\varpi_\mathfrak{q}}{}{}{\varpi_\mathfrak{q}}$ on $R\Gamma\left(X_{G, K} / F, \omega^{\mu}\right)$ is multiplication by $(\# \cO_F / \mathfrak{q})^w$.
\end{note}

%%%%%%%%%%%%%%%%%%%%%%%%%%%%%%%%%%%%%%%%

\subsection{Automorphic representations of $G$}

 \begin{definition}\label{def:Pi}
  Let $\Pi$ be the unitary cuspidal automorphic representation of $G(\AA)$ generated by a holomorphic Hilbert modular newform of weight $(k_1 + 2, k_2 + 2)$, for some integers $k_1, k_2 \ge 0$ with $k_1 = k_2 \bmod 2$. We assume that $\Pi$ has level $\fN \trianglelefteqslant \cO_F$, so if $U_1(\fN)$ denotes the open compact subgroup $\{ g \in G\left(\Af\right) : g = \stbt{\star}{\star}{0}{1} \bmod \fN\}$, then the space of $U_1(\fN)$-invariants of $\Pif$ (the new subspace of $\Pif$) is one-dimensional.
 \end{definition}

 \begin{remark}
  We shall assume throughout this paper that $\fN$ does not divide $6\operatorname{disc}_{F/\QQ}$, so that $U_1(\fN)$ is \emph{sufficiently small} in the sense of \cite[Definition 2.2.1]{leiloefflerzerbes18}. The case of small levels can be dealt with in the usual fashion by introducing an auxiliary full level structure at some prime away from $p$; we leave the details to the interested reader.
 \end{remark}

 We define $a_{\mathfrak{n}}^\circ(\Pi)$, for each ideal $\mathfrak{n} \trianglelefteqslant \cO_F$, to be the Hecke eigenvalues of the new vector of $\Pi$, normalised in the analytic fashion, so that $|a_{\mathfrak{q}}^\circ(\Pi)| \le 2$ for each prime $\mathfrak{q} \nmid \fN$. For each such prime $\mathfrak{q}$ we let $\alpha_{\mathfrak{q}}^\circ$ and $\beta_{\mathfrak{q}}^\circ$ be the Satake parameters of $\Pi_\pp$ (so that $|\alpha_{\mathfrak{q}}^\circ| = |\beta_{\mathfrak{q}}^\circ| = 1$ and $a_{\mathfrak{q}}^\circ(\Pi) = \alpha_{\mathfrak{q}}^\circ + \beta_{\mathfrak{q}}^\circ$).

 By a rationality theorem due to Shimura \cite[Proposition 1.2]{shimura86}, if we choose $w$ such that $(k_1, k_2, w)$ is a weight, then there is a number field $L \subset \CC$ depending only on $\Pi$ which contains the quantities $N(\mathfrak{n})^{(w + 1)/2} a_\mathfrak{n}^\circ(\Pi)$ for all ideals $\mathfrak{n}$. We can (and do) assume that $L$ also contains the images of the embeddings $F \into \RR$, which is automatic if $k_1 \ne k_2$.

 For each integer $n \ge 1$, we write
 \[ a_n(\Pi) = n^{(k_1 + k_2 + 2)/2} a_{n \cO_F}^\circ(\Pi) \in \cO_L. \]

%%%%%%%%%%%%%%%%%%%%%%%%%%%%%%%%%%%%%%%%

\subsection{The Asai $L$-function}

\begin{definition}
 \label{def:AsaiL}
 We define the \emph{imprimitive Asai $L$-series} of $\Pi$ to be the function
 \[ L^{\imp}_{\As}(\Pi, s) = L_{(N)}(\chi, 2s - k_1 - k_2 - 2) \sum_{n \in \ZZ_{\ge 1}} \frac{a_n(\Pi)}{n^{s}}, \]
 where $\chi$ is the restriction to $\AA_{\QQ}^\times$ of the central character of $\Pi$, and $(N)$ denotes omitting the Euler factors at the primes dividing $N = \fN \cap \ZZ$.
\end{definition}

\begin{note}\label{note:critval} \
 \begin{enumerate}[(a)]
  \item The imprimitive $L$-function $L^{\imp}_{\As}(\Pi, s)$ may differ by finitely many Euler factors from the ``true'' Asai $L$-series $L_{\As}(\Pi, s)$ whose local factors are defined using the local Langlands correspondence.
  \item Our normalisations are slightly different here from \cite{grossiloefflerzerbesLfunct}, in which we used the analytic normalisation (with the centre at $s = \tfrac{1}{2}$) which is more convenient for zeta-integral computations; the $L^{\imp}_{\As}(\Pi, s) $ here is $L^{\imp}_{\As}(\Pi, s - \tfrac{k_1 + k_2 + 2}{2})$ in the notation of \emph{op.cit.}. On the other hand, our normalisations \emph{are} compatible with \cite[\S 5.1]{leiloefflerzerbes18}, and are also consistent with the conventional normalisation for  Rankin--Selberg $L$-functions in \cite{KLZ17} and elsewhere (if we interpret Rankin--Selberg $L$-functions as ``Asai $L$-functions for the degenerate quadratic field $\QQ \times \QQ$'').
  \item The critical values of $L_{\As}(\Pi, s)$ are the integers satisfying $k_2 + 2 \leq s \leq k_1 + 1$ if $k_1 > k_2$, and similarly if $k_1 < k_2$. (There are no critical values if $k_1 = k_2$.)\qedhere
 \end{enumerate}
\end{note}

%%%%%%%%%%%%%%%%%%%%%%%%%%%%%%%%%%%%%%%%

\subsection{Hecke eigensystems}

\begin{definition}
 Let $S$ be a finite set of primes containing all those dividing $\operatorname{disc}_{F/\QQ}\cdot \operatorname{Nm}_{F/\QQ}(\fN)$; and write $\mathbb{T}^S$ for the abstract Hecke algebra $L[G(\Af^S) \sslash G(\hat{\ZZ}^S)]= L[\{T_\mathfrak{q}, S_{\mathfrak{q}}, S_{\mathfrak{q}}^{-1} : \mathfrak{q} \notin S_F\}]$ (where $S_F$ is the set of primes of $F$ above $S$).
\end{definition}

\begin{note}
 If we choose $\Pi$ of weight $(k_1, k_2)$ as above, and also a $w$ such that $(k_1, k_2; w)$ is a weight, then $\Pi_\f^S \otimes \|\cdot\|^{-{w/2}}$ is definable over $L$ as a $G(\Af^S)$-module, and hence the action of the Hecke algebra on its $G(\ZZ^S)$-invariants gives a ring homomorphism
 \[ \lambda_\Pi^S:\mathbb{T}^S\to L, \qquad
 T_{\mathfrak{q}} \mapsto \operatorname{Nm}(\mathfrak{q})^{(w + 1)/2} a_{\mathfrak{q}}^{\circ}(\Pi), \qquad
 S_{\mathfrak{q}} \mapsto \operatorname{Nm}(\mathfrak{q})^w \chi_{\Pi}(\varpi_{\mathfrak{q}}),
 \]
 whose kernel we denote $I_\Pi^S$.
\end{note}

\begin{assumption}\label{ass:psplit}
 Let $p \notin S$ be a prime, and assume that $p$ is split in $F$, say $\langle p\rangle=\pp_1 \pp_2$.
\end{assumption}

 Fix a prime $\mathfrak{P}$ of our coefficient field $L$ above $p$ such that $\pp_i$ is the preimage under $\sigma_i$ of the prime $\mathfrak{P}$. Let $\varpi_{\pp_i}$ be a uniformizer of  $\pp_i$.

  Define the following level structures at $p$:

\begin{notation}\
 \begin{enumerate}
  \item For $=1,2$ let
  \[ \Iw(\pp_i)= \{g \in G(\Zp) : g = \stbt * * 0 * \bmod{\pp_i}\}, \]
   and let $U_{\pp_i}$ and $U_{\pp_i}'$ be the $\Iw(\pp_i)$-double cosets of $\stbt{\varpi_{\pp_i}}{}{}{1}$ and  $\stbt{1}{}{}{\varpi_{\pp_i}}$.
  \item Let
   \begin{align*}
    \Iw(p)& =  \Iw(\pp_1)\cap  \Iw(\pp_2)\\
      & = \{g \in G(\Zp) : g = \stbt * * 0 * \bmod p\}
    \end{align*}
    \end{enumerate}
\end{notation}

\begin{definition}
 By a \emph{$p$-stabilisation} of $\Pi$, we mean a choice of one of the Satake parameters (which we denote WLOG by $\alpha_i^\circ$) at each of the primes $\pp_i \mid p$.
\end{definition}

\begin{remark}
 The $p$-stabilisation defines choices of eigenvalues for the operators $U_{\pp_i}$, or their transposes $U'_{\pp_i}$, at level $\Iw(p)$. We normalise these operators in such a way that their eigenvalues are
 \[ \alpha_i \coloneqq p^{(k_i + 1)/2} \alpha_i^{\circ}\qquad
 \text{(\emph{not} } p^{(w + 1)/2}\alpha_i^{\circ}).
 \]
 This normalisation is less canonical (it depends on the choice of ideal $\fP \mid p$ of $L$), but will work better for $p$-adic interpolation; in particular, the valuation of $\alpha_i$ lies in the range $[0, k_i + 1]$.
\end{remark}

\begin{definition}
 For $i=1, 2$, we say that $\Pi$ is \textbf{ordinary at $\pp_i$} if $v_p(\alpha_i)=0$; and $\Pi$ is \textbf{ordinary at $p$} if it is ordinary at both $\pp_1$ and $\pp_2$.
\end{definition}

\begin{definition}\
 \begin{enumerate}
 \item Let  $\mathbb{T}^S_{\pp_1, w_1}$  denote the product of $\mathbb{T}^{S\cup\{ p\}}$ and the subalgebra of  $L[G(\Qp) \sslash \Iw(\pp_1)]$ generated by $U_{\pp_1}'$.
 \item Similarly, let $\mathbb{T}^S_{\Iw, w_1}$denote the product of $\mathbb{T}^{S\cup\{ p\}}$ and the subalgebra of $L[G(\Qp) \sslash \Iw(p)]$
 generated by $U_{\pp_1}'$ and $U_{\pp_2}$ [sic; the roles of the primes are not symmetric].
 \end{enumerate}
\end{definition}

\begin{definition}\
 \begin{enumerate}
     \item The choice of a $\pp_1$-stabilisation defines a character $\lambda^S_{\Pi, \pp_1, w_1}$ of $\mathbb{T}^S_{\pp_1, w_1}$ sending $U_{\pp_1}'$ to $\alpha_1$, and we let $I_{\Pi, \pp_1, w_1}^S$ be its kernel.
     \item Similarly, the choice of a $p$-stabilisation defines a character $\lambda^S_{\Pi, \Iw, w_1}$ of $\mathbb{T}^S_{\Iw, w_1}$ (sending $U_{\pp_1}'$ to $\alpha_1$ and $U_{\pp_2}$ to $\alpha_2$), and we let $I_{\Pi, \Iw, w_1}^S$ be its kernel.
    \end{enumerate}
\end{definition}

%%%%%%%%%%%%%%%%%%%%%%%%%%%%%%%%%%%%%%%%

\subsection{Eigensystems in coherent cohomology}

\begin{notation}
 For $\star\in\{\sph, \pp_1, \Iw\}$, write $Y_{G, *}$ for the Hilbert modular variety $Y_G\left( U_1(\fN)^{(p)} K_{p, \star}\right)$, where
 \[
 K_{p, \star}=
 \begin{cases}
  G(\Zp) &\quad \text{when $\star=\sph$}\\
  \Iw(\pp_1)\subset G(\Zp)&\quad \text{when $\star=\Iw(\pp_1)$}\\
  \Iw(p)\subset G(\Zp)&\quad \text{when $\star=\Iw$}
 \end{cases}
 \]
\end{notation}

 Then $Y_{G, \sph}$, $Y_{G, \pp_1}$ and $Y_{G, \Iw}$ are smooth $\QQ$-varieties. Let $X_{G, \star}$ be a smooth projective toroidal compactification of $Y_{G, \star}$.

\begin{proposition}(c.f. \cite[\S 6.2]{grossiloefflerzerbesLfunct})
 The space\footnote{The label $w_1$ refers to the first generator of the Weyl group $(\ZZ / 2) \times (\ZZ / 2)$; the identity element would correspond to holomorphic Hilbert modular forms.}
 \[
 \cS_{w_1}(\Pi) \coloneqq H^1\left(X_{G, \sph} / L, \omega^{(-k_1,k_2+2; w)}(-D_G)\right)[I_\Pi^S]
 \]
 is $1$-dimensional over $L$, and is independent of the choice of $S$ (and independent of the toroidal boundary data up to a canonical isomorphism).\qed
\end{proposition}

\begin{definition}\label{def:etabasic}
 Let $\eta$ be a basis of the space $\cS_{w_1}(\Pi)$ over $L$.
\end{definition}

\begin{remark}\label{rem:occult}
 Over $\CC$ this space has a canonical basis given by the comparison between coherent cohomology and automorphic forms. The ratio between this basis and $\eta$, uniquely defined as an element of $\CC^\times / L^\times$,  is Harris' \emph{occult period} $\Omega_\infty(\Pi)$ for $\Pi$. See \cite[\S 6.2]{grossiloefflerzerbesLfunct}.
\end{remark}

\begin{definition}
 Similarly, we define
 \begin{align*}
  \cS_{w_1, \Iw(\pp_1)}(\Pi) &\coloneqq H^1\left(X_{G, \Iw(\pp_1)}/L, \omega^{(-k_1,k_2+2; w)}(-D_G)\right)[I_{\Pi, \pp_1, w_1}^S], \\
  \cS_{w_1, \Iw(p)}(\Pi) &\coloneqq H^1\left(X_{G, \Iw}/L, \omega^{(-k_1,k_2+2; w)}(-D_G)\right)[I_{\Pi, \Iw(p), w_1}^S]
 \end{align*}
\end{definition}
%
%\begin{remark}
% The spaces $\cS_{w_1, \pp_1}(\Pi)$ and $\cS_{w_1, \Iw}(\Pi)$ are also $1$-dimensional, and there are isomorphisms
%
% \[ \cS_{w_1}(\Pi) \to \cS_{w_1, \pp_1}(\Pi) \to \cS_{w_1, \Iw}(\Pi)\]
% whose composition is given by
% \begin{equation}
%  \label{eq:pstab}
%  \left(1-\frac{\beta_1}{U'_{\pp_1}}\right)\left(1-\frac{\beta_2}{U_{\pp_2}}\right)\pi_{\Iw}^*
% \end{equation}
% where $\pi_{\Iw} : X_{G, \Iw} \to X_{G, \sph}$ is the natural degeneracy map (c.f. \cite[\S 6.1]{grossiloefflerzerbesLfunct}).
%\end{remark}

%%%%%%%%%%%%%%%%%%%%%%%%%%%%%%%%%%%%%%%%

%%%%%%%%%%%%%%%%%%%%%%%%%%%%%%%%%%%%%%%%

\section{Classes in coherent cohomology}\label{ss:coherent}

\subsection{Coherent cohomology of the 1-ordinary locus}

Let $\YY_{G, \sph}$ denote the natural $\Zp$-model of $Y_{G, \sph}$, and $Y_{G, \sph,0}$ its special fibre, which is a smooth $\mathbb{F}_p$-variety; and similarly for the compactification $X_{G, \sph, 0}$. This special fibre has two commuting endomorphisms, the \textbf{partial Frobenii} $\varphi_1$ and $\varphi_2$ at the primes $\pp_i$, whose composite is the Frobenius $\varphi$; more precisely, $\varphi_i$ corresponds to sending a Hilbert--Blumenthal abelian surface $A$ to the quotient $A / \left( \ker(\varphi_A) \cap A[\pp_i]\right)$. We refer to \cite{tianxiao16} or \cite{nekovar-semisimplicity} for detailed accounts of this construction.

\begin{notation} \
 \begin{itemize}
  \item For $i=1, 2$, denote by $X^{(i-\ss)}_{G, \sph,0} \subset X_{G, \sph,0}$ the $\pp_i$-supersingular locus (the vanishing locus of the partial Hasse invariant, as constructed in \cite[\S 3.2]{tianxiao16}).
  \item Let $X^{(i-\ord)}_{G, \sph,0}$ be the complement of $X^{(i-\ss)}_{G, \sph,0}$, and $X^{\ord}_{G, \sph,0}=X^{(1-\ord)}_{G, \sph,0} \cap X^{(2-\ord)}_{G, \sph,0}$.
 \end{itemize}
 We write $Y^{(i-\ord)}_{G, \sph, 0}$ etc for the intersection of these subvarieties with $Y_{G, \sph,0} \subset X_{G, \sph,0}$.
\end{notation}

The following results on the geometry of the supersingular loci are well-known (see e.g.~\cite{tianxiao16}):

\begin{lemma}
 For $i=1, 2$, $X^{(i-\ss)}_{G, \sph,0}$ is a smooth codimension 1 closed subscheme of $X_{G, \sph, 0}$, disjoint from the toroidal boundary; and $X^{(1-\ss)}_{G, \sph,0} \cap X^{(2-\ss)}_{G, \sph,0}$ is a smooth closed subvariety of codimension $2$ (i.e.~a finite disjoint union of points).
\end{lemma}

\begin{remark}
 The preimage of either $X^{(1-\ord)}_{G, \sph,0}$ or $X^{(2-\ord)}_{G, \sph,0}$ under the finite map $\iota: X_{H, \sph,0} \to X_{G, \sph,0}$ is the ordinary locus $X_{H, \sph,0}^{\ord}$.
\end{remark}

%\begin{proposition}\label{prop:extbyzero}
% The extension-by-0 map
% \[ R\Gamma_{\rig, c}(Y^{(1-\ord)}_{G, \sph, 0}, \cV_G)
% \longrightarrow R\Gamma_{\rig, c}(Y_{G, \sph, 0}, \cV_G) \cong R\Gamma_{\dR, c}(Y_{G, \sph}, \cV_G) \]
% is a quasi-isomorphism on the $\Pi$ generalised eigenspace for the prime-to-$Np$ Hecke operators.
%\end{proposition}
%
%\begin{proof}
% This is a special case of Proposition 4.3 of \cite{LZ-plectic}.
%\end{proof}
%
%
%\begin{notation}\label{not:classfromdR}
% Write $\eta^{(1-\ord)}_{\dR} \in H^2_{\rig, c}(Y^{(1-\ord)}_{G, \sph,0}, \cV_G)$ for the preimage of $\eta_{\dR}$ under the isomorphism in \cref{prop:extbyzero}.
%\end{notation}

%%%%%%%%%%%%%%%%%%%%%%%%%%%%%%%%%%%%%%%%

We write $\cX_{G, \sph}$ for the dagger space associated to $X_{G, \sph} /\QQ_p$; and we denote the tubes in $\cX_{G, \sph}$ of the various subvarieties of $X_{G, \sph,0}$ considered above by the corresponding superscripts on $\cX_{G, \sph}$, so $\cX_{G, \sph}^{(i-\ord)}$ is the tube of $X_{G, \sph,0}^{(i-\ord)}$ in $\cX_{G, \sph}$ etc.

\begin{theorem}
 \label{thm:nu1ord}
 Suppose that $\Pi$ is $\pp_1$-ordinary, or more generally that $v_p(\alpha_1) < k_1 + 1$ and $\alpha_1 \ne \beta_1$. Then there exists a unique class
 \[ \eta^{(1-\ord)} \in H^1_c\left(\cX_{G, \sph}^{(1-\ord)}, \omega^{(-k_1, k_2 + 2)}(-D_G)\right)\]
 satisfying the following properties:
 \begin{enumerate}
  \item it is a $\varphi_1$-eigenvector with eigenvalue $\alpha_1$;
  \item its image in $H^1(\cX_{G, \sph} / E, \omega^{(-k_1, k_2 + 2)}(-D_G))$ is $\eta$.
 \end{enumerate}
\end{theorem}

\begin{proof}
 If we assume $v_p(\alpha_1) < k_1$ and $\alpha_1 \ne \beta_1$, then this is Proposition 5.2 of \cite{LZ-plectic}. However, the restriction $v_p(\alpha_1) < k_1$ in \emph{op.cit.} was imposed in order to apply the classicity theorems in the then-current version of \cite{boxerpilloni22}; and subsequent revisions to \cite{boxerpilloni22} have shown that the classicity theorems hold under the slightly weaker assumption of ``small slope'' rather than ``strictly small slope'' in the terminology of \emph{op.cit.}, which corresponds to $v_p(\alpha_1) < k_1 + 1$ in the present setting.
\end{proof}

%%%%%%%%%%%%%%%%%%%%%%%%%%%%%%%%%%%%%%%%

 \subsection{Comparison with cohomology at $\Iw(\pp_1)$-level}\label{ss:compp1}

  We now compare the class $\eta^{(1-\ord)}$ with algebraic coherent classes at level $\Iw(\pp_1)$.

  \begin{note}
   The special fibre $X_{G, \Iw(\pp_1), 0}$ has a stratification with three strata,
   \[ X_{G, \Iw(\pp_1), 0} = X_{G, \Iw(\pp_1), 0}^{(1-m)} \cup X_{G, \Iw(\pp_1), 0}^{(1-\et)} \cup X_{G, \Iw(\pp_1), 0}^{(1-\alpha)}, \]
   on which the $\pp_1$-level structure is multiplicative, \'etale, or $\alpha_p$ respectively. We obtain a corresponding decomposition of the dagger space $\cX_{G, \Iw(\pp_1)}$, and the natural degeneracy map $\pi_{\pp_1} : \cX_{G, \Iw(\pp_1)} \to \cX_{G, \sph}$ restricts to an isomorphism\footnote{The inverse map is given by the ``$\pp_1$-canonical subgroup'' construction, see \cite[3.11]{tianxiao16}.} of dagger spaces
   \[ \pi_{\pp_1}:\, \cX_{G, \Iw(\pp_1)}^{(1-m)} \xrightarrow{\ \cong\ } \cX_{G, \sph}^{(1-\ord)}.\qedhere\]
  \end{note}

  By \cite[Lemma 5.7]{LZ-plectic}, our Frobenius lift $\varphi_1$ on $H^1_c(\cX_{G, \sph}^{(1-\ord)})$ corresponds to the Hecke operator $U_{\pp_1}'$ acting on $\cX_{G, \Iw(\pp_1)}^{(1-m)}$. From the functoriality of push-forward maps, we have the following compatibility:

  \begin{proposition}\label{prop:compnus}
   Write $\eta^{(1-m)}_{\Iw(\pp_1)}$ for the image of $\eta^{(1-\ord)}$ under the inverse of the isomorphism $ \pi_{\pp_1,*}$. The image $\eta_{\Iw(\pp_1)}$ of $\eta_{\Iw(\pp_1)}^{(1-m)}$ in $H^1(\cX_{G, \Iw(\pp_1)}$ is the unique class which lies in the $\Pif$-eigenspace away from $p$, is a $(U_{\pp_1}' = \alpha_1)$-eigenvector, and maps to $\eta$ under the trace map $\pi_{\pp_1,*}$.\qed
  \end{proposition}

  See also \cref{fig:coh}.

  \begin{figure}[h]
   \caption{Relations between the coherent classes at spherical and $\pp_1$-Iwahori level (all cohomology groups with coefficients in $\omega_G^{(-k_1, k_2 + 2)}(-D_G)$)}
   \label{fig:coh}
   \[
    \begin{tikzcd}[column sep = large, contains/.style = {draw=none,"\in" description,sloped}]
    \eta^{(1-m)}_{\Iw(\pp_1)} \rar[mapsto] \ar[mapsto, ddd, bend right=90, looseness=1.3]
    \dar[contains]
    & \eta_{\Iw(\pp_1)} \dar[contains] \ar[mapsto, ddd, bend left=90, looseness=1.3]
    \\
    H^1_c\left(\cX_{G, \Iw(\pp_1)}^{(1-m)}\right)
    \rar["\text{ext-by-0}"]
    \dar["\pi_{\pp_1, *}" left,  "\cong" right]
    & H^1\left(\cX_{G, \Iw(\pp_1)}\right)
    \dar["\pi_{\pp_1,*}" left] \\[15pt]
    H^1_c\left(\cX_{G, \sph}^{(1-\ord)}\right)
    \rar["\text{ext-by-0}"]
    &
    H^1\left(X_{G, \sph}\right) \\
    \eta^{(1-\ord)} \uar[contains] \rar[mapsto]
    &\eta \uar[contains]
    \end{tikzcd}
   \]
  \end{figure}

%%%%%%%%%%%%%%%%%%%%%%%%%%%%%%%%%%%%%%%%

 \subsection{Push-forward maps in coherent cohomology}\label{ss:pfwd}

  Let $Y_{H, \sph, 0}$ be the special fibre of $Y_{H, \sph} / \Zp$, and write $Y_{H, \sph, 0}^{\ord}$ for its ordinary locus. Similarly, denote by $Y_{H, \Iw(p), 0}$ be the special fibre of $Y_{H, \Iw(p)} / \Zp$, and write $Y_{H, \Iw(p), 0}^m$ for its multiplicative locus. We denote the compactifications of these by $X_{(\dots)}$, and the associated dagger space tubes by calligraphic letters, as above.

  Then the natural degeneracy map of dagger spaces restricts to an isomorphism
  \[ \pi_p:  \cX_{H, \Iw(p)}^m \xrightarrow{\ \cong\ } \cX_{H, \sph}^{\ord}.\]

  \begin{note}\label{note:iota}
    The embedding $\iota:\, H\to G$ gives rise to finite maps
    \[
     \iota:\, \cX_{H, \sph} \to \cX_{G, \sph}
     \quad\text{and}\quad
     \iota^{(p)}:\, \cX_{H, \Iw(p)} \to \cX_{G, \Iw(\pp_1)},
    \]
    and it is easy to check that we have
    \[
     \left(\iota^{(p)}\right)^{-1}\left(\cX_G(\pp_1)^m\right)=\cX_H(p)^m
     \qquad \text{and}\qquad
     \iota^{-1}\left(\cX_{G, \sph}^{(1-\ord)}\right)=\cX_{H, \sph}^{\ord}.
    \]
    The maps $\iota$ and $\iota^{(p)}$ hence give rise to pushforward maps
    \begin{align*}
    \iota^{(p)}_*:\, H^0\left(\cX_{H, \Iw(p)}^m, \omega_H^{k_1-k_2}\right)& \to H^1\left(\cX_{G, \Iw(\pp_1)}^{(1-m)}, \omega_G^{(k_1 + 2, -k_2)}\right), \\
    \iota_*:\, H^0\left(\cX_{H, \sph}^{\ord}, \omega_H^{k_1-k_2}\right)& \to H^1\left(\cX_{G, \sph}^{(1-\ord)}, \omega_G^{(k_1 + 2, -k_2)}\right).\qedhere
   \end{align*}
  \end{note}

 The following result will be useful in the evaluation of the syntomic regulator:

  \begin{lemma}\label{lem:cartesian}
   The following diagram is cartesian:
   \[
    \begin{tikzcd}
     X_{H, \Iw(p)}
     \rar["\iota^{(p)}"] \dar["\pi_p"]&
     X_{G, \Iw(\pp_1)} \dar["\pi_{\mathfrak{p}_1}"] \\
     X_{H, \sph}
     \rar["\iota"] &%\ar{rd} &
     X_{G, \sph}
    \end{tikzcd}
   \]
   so $\pi_{\mathfrak{p}_1}^*\circ \iota_*=\iota^{(p)}_*\circ\pi^*_p$.\qed
 \end{lemma}

%%%%%%%%%%%%%%%%%%%%%%%%%%%%%%%%%%%%%%%%

 \subsection{Construction of the $p$-adic $L$-function}\label{ssect:1parampL}

  We briefly recall the construction of $p$-adic Asai $L$-functions developed in \cite{grossiloefflerzerbesLfunct}. The class $\eta_{\Iw(\pp_1)}^{(1-m)}$ lies in the compactly-supported cohomology of the \emph{overconvergent} dagger space $\cX_{G, \Iw(\pp_1)}^{(1-m)}$; so by duality it defines a linear functional on $H^1\left(\cX_{G, \Iw(\pp_1)}^{(1-m)}, \omega_G^{(k_1 + 2, -k_2)}\right)$.

  However, if we impose the stronger assumption that $\Pi$ is $\pp_1$-ordinary, then we can lift $\eta_{\Iw(\pp_1)}^{(1-m)}$ to the cohomology of the \emph{classical} rigid space $\sX_{G, \Iw(\pp_1)}^{(1-m)}$ underlying $\cX_{G, \Iw(\pp_1)}^{(1-m)}$; so the corresponding linear functional factors through the natural ``forgetting overconvergence'' map
  \[
   H^1\left(\cX_{G, \Iw(\pp_1)}^{(1-m)}, \omega_G^{(k_1 + 2, -k_2)}\right) \to H^1\left(\sX_{G, \Iw(\pp_1)}^{(1-m)}, \omega_G^{(k_1 + 2, -k_2)}\right).
  \]
  This is useful to us since it allows us to pair $\eta_{\Iw(\pp_1)}^{(1-m)}$ with pushforwards of families of $\GL_2$ Eisenstein series, which are not overconvergent.

  We write $\cW$ for the formal scheme $\operatorname{Spf} \cO_E[[\Zp^\times]]$ (``weight space'') parametrising continuous characters of $\Zp^\times$, so elements of $\cO_E[[\Zp^\times]]$ are interpreted as functions on $\cW$.

  \begin{definition}\label{def:AsaipL}
   We define $L_{p, \As}^{\imp}(\Pi) \in \cO_E[[\Zp^\times]][1/p]$ by
   \[
    L_{p, \As}^{\imp}(\Pi, \sigma) \coloneqq
    (1 - \tfrac{\beta_1}{p\alpha_1}) (\sqrt{D})^{-\sigma} (-1)^{\sigma + (k_1 + k_2 + 2)/2}
    \left\langle \eta_{\Iw(\pp_1)}^{(1-m)},
    \iota_{\coh, *}^{{(p)}}\left(\cE_{k_1-k_2}(\sigma - \tfrac{k_1 + k_2+2}{2})\right)\right\rangle, \]
   where $\cE_{k_1-k_2}$ is a family of $p$-depleted Eisenstein series over $\cW$, of constant weight $k_1 - k_2$, defined in \cite[\S 7.2]{grossiloefflerzerbesLfunct}, with the Schwartz function $\Phi^{(p)}$ chosen as in \S 7.3.2 of \emph{op.cit.}.
  \end{definition}

  \begin{remark}\label{rem:competas}
   The factor $ (1 - \tfrac{\beta_1}{p\alpha_1})$ (which is non-zero, since $\alpha_1 / \beta_1$ has complex absolute value 1) is included to correct for a discrepancy between the cohomology class $\eta_{\Iw(\pp_1)}$, which is a preimage of $\eta$ under $\pi_{\pp_1,*}$, and the class $\breve{\nu}_{\Pi, \alpha}$ used in \cite{grossiloefflerzerbesLfunct}, which is the image of $\eta$ under $(1 - \tfrac{\beta_1}{U_{\pp_1}'})\circ\pi_{\pp_1}^*$. An elementary computation shows that these classes are related by
   \[
    \breve{\nu}_{\Pi, \alpha} = p (1 - \tfrac{\beta_1}{p\alpha_1}) \cdot \eta_{\Iw(\pp_1)},
   \]
   so our $p$-adic $L$-function coincides with that of \emph{op.cit.} up to the shift by $\tfrac{k_1 + k_2+2}{2}$. In particular for integers $s$ with $k_2 + 2 \le s \le k_1 + 1$, the $p$-adic $L$-value $L_{p, \As}^{\imp}(\Pi, s)$ is an explicit multiple of the complex $L$-value $L_{\As}^{\imp}(\Pi, s)$ with the conventions of \cref{def:AsaiL}.
  \end{remark}

%%%%%%%%%%%%%%%%%%%%%%%%%%%%%%%%%%%%%%%%

\section{The Asai motive and its realisations}

 \subsection{De Rham and coherent cohomology groups}

  As above $Y_{G, \sph}$ denotes the Shimura variety for $G$ of level $U_1(\fN)$, and $X_{G, \sph}$ a toroidal compactification. For this section we drop the ``$\sph$'' subscript, since only the spherical-level varieties will appear.

  \begin{definition}
   Let $\cV_G$ denote the vector bundle with connection on $Y_G$ corresponding to the algebraic representation\footnote{More precisely, we need to twist by an appropriate character to make the action of norm-one units trivial, but the resulting vector bundle is independent of the choice up to a canonical isomorphism.} $V_G$; and set
   \[ D_L^{\As}(\Pi) = H^2_{\dR, c}\Big(Y_G / L, \cV_G\Big)[\Pif], \]
   which is 4-dimensional over $L$. Here, $[\Pif]$ denotes the $\Pif$-isotypical projection.
  \end{definition}

  We can compute de Rham cohomology using the logarithmic de Rham complex $\cV_G \otimes \Omega^\bullet_{X_G / L}\langle D_G \rangle$, where $D_G = X_G - Y_G$ is the boundary divisor; and we can compute compactly-supported de Rham cohomology using the ``minus-log'' complex $\cV_G \otimes \Omega^\bullet_{X_G}\langle -D_G \rangle$, where $\Omega^\bullet\langle -D_G\rangle \coloneqq \Omega^\bullet\langle D_G \rangle(-D_G)$. We denote the corresponding cohomologies by
  \begin{align*}
   R\Gamma_{\dR}\left(X_G, \cV_G\langle D \rangle\right) &\coloneqq R\Gamma\left(X_G, \cV \otimes \Omega^\bullet_{X_G}\langle D\rangle\right) \cong R\Gamma_{\dR}(Y_G, \cV_G), \\
   R\Gamma_{\dR}(X_G, \cV_G\langle -D_G \rangle) &\coloneqq R\Gamma(X_G, \cV \otimes \Omega^\bullet_{X_G}\langle -D_G \rangle) \cong R\Gamma_{\dR, c}(Y_G, \cV_G).
  \end{align*}

  \begin{definition}
   Define the  \emph{dual BGG complex}
   \[ \BGG^\bullet = \left[ \omega^{(-k_1, -k_2)} \longrightarrow
   \omega^{(-k_1, k_2 + 2)} \oplus \omega^{(k_1 + 2, -k_2)} \longrightarrow
   \omega^{(k_1 + 2, k_2 + 2)} \right], \]
   and write  $\BGG^\bullet(-D_G)$ for its compactly supported analogue.
  \end{definition}

  The following result is well-known:

  \begin{proposition}
   The dual BGG complex (resp. the compactly supported dual BGG complex) is quasi-isomorphic to the logarithmic de Rham complexes $V \otimes \Omega^\bullet\langle D \rangle$ (resp. to $V \otimes \Omega^\bullet \langle -D_G \rangle$).
  \end{proposition}

   The BGG complex is equipped with a natural decreasing $\ZZ^2$-filtration $\Fil^{\bullet\bullet}$, which gives rise to a $\ZZ^2$-filtration on $D_L^{\As}(\Pi)$; this construction is studied in detail in \cite{LZ-plectic}. We can (and do) normalise so that the nontrivial graded pieces are in bidegrees $\{(0, 0), (k_1 + 1, 0), (0, k_2 + 1), (k_1 + 1, k_2 + 1)\}$.
 %, and are given by
 %\begin{gather*}
 % \Gr^{(k_1 + 1, k_2 + 1)}= H^0\left(X_G, \omega^{(k_1 + 1, k_2 + 1)}(-D_G)\right)[\Pif], \\
 % \Gr^{(k_1 + 1, 0)} = H^1\left(X_G, \omega^{(k_1 + 2, -k_2)}(-D_G)\right)[\Pif], \qquad\qquad \Gr^{(0, k_2 + 1)} = H^1\left(X_G, \omega^{(-k_1, k_2 + 2)}(-D_G)\right)[\Pif], \\
 % \Gr^{(0, 0)}= H^2\left(X_G, \omega^{(-k_1, -k_2)}(-D_G)\right)[\Pif].
 %\end{gather*}
   The induced single filtration, with graded pieces in degrees $\{0, k_1 + 1, k_2 + 1, k_1 + k_2 + 2\}$, is the Hodge filtration.

  \begin{notation}
   We write $\Fil_1^+ D_L^{\As}(\Pi) \coloneqq \Fil^{(k_1 + 1, 0)}D_L^{\As}(\Pi)$ and
   $\Fil_2^+ D_L^{\As}(\Pi) \coloneqq \Fil^{(0, k_2 + 1)}D_L^{\As}(\Pi)$. Then each of these subspaces is 2-dimensional; their intersection is $\Fil^{k_1 + k_2 + 2}$, and their sum is $\Fil^1$; and we have a canonical isomorphism
   \[ \frac{\Fil_2^+ D_L^{\As}(\Pi)}{\Fil_1^+ \cap \Fil_2^+} \cong \cS_{w_1}(\Pi).\]
  \end{notation}

%%%%%%%%%%%%%%%%%%%%%%%%%%%%%%%%%%%%%%%%

 \subsection{The Asai Galois representation}

  We now fix (for the remainder of this paper) a prime $p$, a finite extension $E / \Qp$, and an embedding $L \into E$, where $L \subset \CC$ is the finite extension of $\QQ$ generated by the emeddings of $F$ and the normalised Hecke eigenvalues of $\Pi$. The field $E$ will be the coefficient field for our $p$-adic $L$-functions and Galois representations.

  \begin{definition} \
   \begin{enumerate}[(i)]
    \item Let $V_p^{\As}(\Pi)$ denote the four-dimensional Asai Galois representation associated to $\Pi$ as in \cite[Definition 4.4.2]{leiloefflerzerbes18}, defined as the $\Pif$-eigenspace in the $p$-adic \'etale cohomology of the Hilbert modular variety $Y_G \otimes \QQbar$ (with coefficients in the \'etale local system of $E$-vector spaces determined by $(k_1, k_2)$).
    \item Let $D_p^{\As}(\Pi) =  E \otimes_L D_L^{\As}(\Pi)$, so that there is a canonical comparison isomorphism (compatible with the Hodge filtration)
    \[  \DdR\left(\Qp, V_p^{\As}(\Pi) \right) \cong D_p^{\As}(\Pi).\]
   \end{enumerate}
  \end{definition}

  \begin{remark}
   More precisely, the representation $V_p^{\As}(\Pi)$ is $M_{L_{\fP}}(\cF) \otimes_{L_{\fP}} E$ in the notation of \emph{op.cit.}, where $\cF$ is the normalised newform generating $\Pi$ and $\fP \mid p$ the prime of $L$ determined by the embedding into $E$. By results of Brylinski--Labesse and Nekov\'a\v{r} recalled in \emph{op.cit.}, the Galois representation $V_p^{\As}(\Pi)$ can be characterised, up to isomorphism, as the unique semisimple Galois representation whose $L$-series is $L_{\As}\left(\Pi\right)$. However, since we want to consider Euler system classes for $V_p^{\As}(\Pi)$, it is important to fix not only an abstract isomorphism class of Galois representations but a specific realisation of this isomorphism class in \'etale cohomology.
  \end{remark}

  \begin{corollary}
   The representation $V_p^{\As}(\Pi)$ is crystalline at $p$, so $D_p^{\As}(\Pi)$ is naturally a filtered $\varphi$-module; and the eigenvalues of $\varphi$ on this module are the pairwise products
   \( \{ \alpha_1 \alpha_2, \dots, \beta_1 \beta_2\}. \)\qed
  \end{corollary}

% We recall the following result from \cite{taylor89}:
%
%\begin{proposition}
% The restriction of $V_p^{\As}(\Pi)$ to $G_{\QQ_p}$ is equipped with two $2$-dimensional subspaces $\cF^+_i V_p^{\As}(\Pi)$ for $i=1,2$, which induce a $3$-step decreasing filtration
% \begin{multline*}
%  \Fil^0V_p^{\As}(\Pi)=V_p^{\As}(\Pi), \qquad
%  \Fil^1 V_p^{\As}(\Pi)=\cF^+_1V_p^{\As}(\Pi)+\cF^+_2 V_p^{\As}(\Pi), \\ \Fil^2V_p^{\As}(\Pi)=\cF^+_1V_p^{\As}(\Pi)\cap\cF^+_2V_p^{\As}(\Pi).
% \end{multline*}
%\end{proposition}

  \begin{lemma}
   Let $0 \le j \le \min(k_1, k_2)$ be an integer. Then $p^j$ is not an eigenvalue of $\varphi$ on $D_p^{\As}(\Pi)$. Moreover, if $p^{(1+j)}$ is an eigenvalue of $\varphi$, then we must have $k_1 = k_2 = j$.
  \end{lemma}

  \begin{proof}
   This follows from the fact that the Satake parameters $\alpha_{\pp_i}^\circ$, $\beta_{\pp_i}^\circ$ have complex absolute value 1.
  \end{proof}

  This implies that for the Galois representation $V_p^{\As}(\Pi)^*(-j)$, the Bloch--Kato subspaces $H^1_{\mathrm{e}}(\Qp, -)$ and $H^1_\f(\Qp, -)$ agree; and these are also equal to $H^1_{\mathrm{g}}$, except possibly in the boundary case $k_1 = k_2 = j$, in which case $H^1_{\mathrm{g}}$ can be strictly larger. Moreover, the inverse of the Bloch--Kato exponential map for $V_p^{\As}(\Pi)^*(-j)$ is an isomorphism
  \begin{equation}
   \label{eq:expiso}
   \log : H^1_\f\left(\Qp, V_p^{\As}(\Pi)^*(-j)\right) \xrightarrow{\ \cong\ } \left( \Fil^{1+j} D_p^{\As}(\Pi) \right)^* = \left( \Fil^1 D_p^{\As}(\Pi) \right)^*,
  \end{equation}
  with both sides 3-dimensional over $E$.

%%%%%%%%%%%%%%%%%%%%%%%%%%%%%%%%%%%%%%%%
\subsection{Partial Frobenii}
%%%%%%%%%%%%%%%%%%%%%%%%%%%%%%%%%%%%%%%%

We may identify $D_p^{\As}(\Pi)$, as a Frobenius module (forgetting the filtration), with the rigid cohomology of the special fibre of $Y_G$ at $p$.

The operators $\varphi_i$ induce commuting linear operators on $D_p^{\As}(\Pi)$, with $\varphi = \varphi_1 \varphi_2$; and it follows from the ``partial Eichler--Shimura'' comparison result proved in \cite{nekovar-semisimplicity} that for each $i$ we have
\[ (\varphi_i - \alpha_i)(\varphi_i - \beta_i) = 0 \quad \text{on } D_p^{\As}(\Pi).\]
One checks easily that the partial Frobenii satisfy $\lambda(\varphi_i x, \varphi_i y) = p^{k_i + 1} \lambda(x, y)$, where $\lambda$ is the Poincar\'e duality form. This identifies the $\beta_i$-generalised eigenspace with the dual of that for $\alpha_i$. Hence, if $\alpha_i \ne \beta_i$, the $\varphi_i = \alpha_i$ and $\varphi_i = \beta_i$ eigenspaces are both 2-dimensional, and each is isotropic with respect to $\lambda$.

\begin{remark}
 We are using a slightly different normalisation of the partial Frobenii here from \cite{LZ-plectic}: the $\varphi_i$ here is $p^{-t_i} \varphi_i$ in the notation of \emph{op.cit.}, where $(t_1, t_2)$ is an auxiliary choice of integers such that $w = k_i + 2t_i$ is independent of $i$. This reflects the fact that the $\alpha_i$ of \emph{op.cit.} is $p^{(w + 1)/2} \alpha_{\pp_i}^\circ$, while $\alpha_i$ here is $p^{(k_i + 1)/2} \alpha_{\pp_i}^\circ$. The present normalisation is more convenient for comparison with higher Hida theory, since it matches the minimal integral normalisation of the $U_{\pp_i}'$ operators.
\end{remark}

%%%%%%%%%%%%%%%%%%%%%%%%%%%%%%%%%%%%%%%%

 \subsection{Eichler--Shimura isomorphism}

  \begin{proposition}\label{prop:filiso}
   Suppose that $\alpha_1 \ne \beta_1$, and we have the strict inequality $v_p(\alpha_1) < k_1 + 1$. Then the intersection $\Fil_2^+ D_p^{\As}(\Pi) \cap D_p^{\As}(\Pi)^{(\varphi_1 = \alpha_1)}$ is 1-dimensional, and maps isomorphically to
  \[ \frac{\Fil_2^+ D_p^{\As}(\Pi)}{\Fil_1^+ \cap \Fil_2^+} \cong \cS_{w_1}(\Pi) \otimes_L E. \]
  \end{proposition}

  \begin{proof}
   This is a special case of the first main theorem of \cite{LZ-plectic}, which states that (over totally-real fields of any degree) the $\ZZ^d$-indexed filtration of $D_p^{\As}(\Pi)$ has a canonical splitting given by intersecting with partial Frobenius eigenspaces. (As with \cref{thm:nu1ord} above, the result is proved in \cite{LZ-plectic} assuming $v_p(\alpha_1) < k_1$, but this can be extended to $v_p(\alpha_1) < k_1 + 1$.)
  \end{proof}

  If $\Pi$ is ordinary at $p$, then the subspace $D_p^{\As}(\Pi)^{\varphi_1 = \alpha_1}$ is weakly admissible, so arises from a 2-dimensional subrepresentation $\cF_1^+ V \subset V$, and similarly for $\cF_2^+ V$; the intersection of these is the unique one-dimensional unramified subspace. Hence the subspace $\Fil_2^+ D_p^{\As}(\Pi) \cap D_p^{\As}(\Pi)^{(\varphi_1 = \alpha_1)}$ also maps isomorphically to $\Dcris(\cF_1^+ V) / \Dcris(\cF_1^+ V \cap \cF_2^+ V)$, and we obtain an isomorphism
  \begin{equation}
   \label{eq:ES}
   \ES_{\Pi, w_1} : \cS_{w_1}(\Pi) \otimes_L E \cong \Dcris\left( \frac{\cF_1^+ V}{\cF_1^+ V \cap \cF_2^+ V}\right).
  \end{equation}

  Dually, we let $\cF_i^+ V^* \subset V^*$ be the annihilator of $\cF_i^+ V$, so we can regard $\ES_{\Pi, w_1}(\eta)$ as an isomorphism
  \[  \Dcris\left(\frac{\cF_1^+ V^* + \cF_2^+ V^* }{\cF_1^+ V^*}\right) \xrightarrow{\ \cong\ } E.\]

%%%%%%%%%%%%%%%%%%%%%%%%%%%%%%%%%%%%%%%%

\subsection{A class in de Rham cohomology}

The following result is an immediate corollary of Proposition \ref{prop:filiso}:

\begin{proposition}\label{prop:eta}
 Suppose that $\alpha_1 \ne \beta_1$, and $v_p(\alpha_1) < k_1 + 1$. Then there exists a uniquely determined vector $\eta_{\dR} \in D_p^{\As}(\Pi)$ with the following properties:
 \begin{itemize}
  \item $\varphi_1 \eta_{\dR} = \alpha_1 \cdot \eta_{\dR}$;
  \item $\eta_{\dR} \in \Fil_2^+ D_p^{\As}(\Pi)$;
  \item the image of $\eta_{\dR}$ in the graded piece
  \[ \Fil_2^+ D_p^{\As}(\Pi) / (\Fil_1^+ \cap \Fil_2^+) \cong \cS_{w_1}(\Pi) \otimes_L E \]
  coincides with $\eta$.
 \end{itemize}
\end{proposition}

\begin{remark}
 Note that
    \[ \ker\left( (\varphi - \alpha_1 \alpha_2) (\varphi - \alpha_1\beta_2) : D_p^{\As}(\Pi) \to D_p^{\As}(\Pi)\right)\]
    contains $D_p^{\As}(\Pi)^{(\varphi_1 =\alpha_1)}$, but it may be larger; this always occurs if $\Pi$ is a twist of a base-change from $\QQ$ (so that $\alpha_1 / \beta_1 = \alpha_2 / \beta_2$).
\end{remark}

\begin{proposition}\label{prop:extbyzero}
 The extension-by-0 map
 \[ R\Gamma_{\rig, c}\left(Y^{(1-\ord)}_{G, 0}, \cV_G\right)
 \longrightarrow R\Gamma_{\rig, c}(Y_{G, 0}, \cV_G) \cong R\Gamma_{\dR, c}(Y_G, \cV_G) \]
 is a quasi-isomorphism on the $\Pif$ generalised eigenspace for the prime-to-$Np$ Hecke operators.
\end{proposition}

\begin{proof}
 This is a special case of Proposition 4.3 of \cite{LZ-plectic}.
\end{proof}

\begin{notation}\label{not:classfromdR}
 Write $\eta^{(1-\ord)}_{\dR} \in H^2_{\rig, c}(Y^{(1-\ord)}_{G, 0}, \cV_G)$ for the preimage of $\eta_{\dR}$ under the isomorphism in \cref{prop:extbyzero}.
\end{notation}

  \begin{remark}
   Observe that in Proposition \ref{prop:extbyzero} we have chosen a lifting of the de Rham class $\eta_{\dR}$ to the $\pp_1$-ordinary locus characterised by information about the action of Hecke operators \emph{away} from $p$; and, separately, in Theorem \ref{thm:nu1ord} we have lifted the coherent class $\eta$ to the $\pp_1$-ordinary locus using information about the action of the Frobenius $\varphi_1$ at $\pp_1$. So it is not obvious how these liftings are related, and our next task is to find a way to reconcile the two, which we will carry out in \cref{sect:dRss} -- see \cref{prop:nu1ord2} below.
  \end{remark}

%%%%%%%%%%%%%%%%%%%%%%%%%%%%%%%%%%%%%%%%
%%%%%%%%%%%%%%%%%%%%%%%%%%%%%%%%%%%%%%%%

\section{Definition of the $p$-adic regulator}

\subsection{Euler system classes}

Let $0 \le j \le \min(k_1, k_2)$. We refer to \cite[\S 4.4]{leiloefflerzerbes18} for the construction of the \emph{Asai--Flach class}
\[ \AF^{[\Pi, j]}_{\et} \in H^1\Big(\ZZ[1 / S], V_p^{\As}(\Pi)^*(-j)\Big), \]
where $S$ is the set of primes dividing $p N \operatorname{disc}(F)$. (More precisely, this class is defined in \emph{op.cit.} for any normalised eigenform $\cF$, not necessarily new; and we are defining $\AF^{[\Pi, j]}_{\et}$ as $\AF^{[\cF, j]}_{\et}$ where $\cF$ is the unique newform generating $\Pi$, consistently with our definition of the Galois representation $V_p^{\As}(\Pi)$.)
%
%\begin{remark}
% We will see in Section \ref{sect:HidaAF} that the Asai--Flach Euler system varies in a $2$-parameter Hida family.
%\end{remark}

\begin{note}
 We briefly recall the definition of $\AF^{[\Pi, j]}_{\et}$. Letting $\cV_G$ and $\cV_H$ denote the relative Chow motives (over $L$) associated to the representations $V_G$ and $V_H$ defined in \cref{ss:ClebschGordan}, we have a pushforward map
 \[
 \iota^{[j]}_*: H^1_{\mot}\left(Y_H, \cV_H(1 + t)\right) \otimes_{\QQ} L
 \longrightarrow H^3_{\mot}\left(Y_G, \cV_G(2 + k_1 + k_2 - j)\right),
 \]
 where $t = k_1 + k_2 - 2j$ as before. We have an analogous pushforward map in \'etale cohomology with $E$-coefficients, and the two are compatible via the \'etale regulator map $r_{\et}$. Since the $\Pif$-generalised eigenspace in the \'etale cohomology of $Y_G(\QQbar)$ vanishes outside degree 2, the Hochschild--Serre spectral sequence gives a projection map from $H^3_{\et}\left(Y_G, \cV_G(2 + k_1 + k_2 - j)\right)$ to $H^1\left(\QQ, V_p^{\As}(\Pi)^*(-j)\right)$. We can then define $\AF^{[\Pi, j]}_{\et}$ as the image of the weight $t$ Eisenstein class $\Eis^{t}_{\mot, N}$ (c.f. \cite[Theorem 4.1.1]{KLZ20}) under this chain of maps. Since the motivic class clearly lifts to a smooth $\ZZ[1/S]$-model, its image is unramified outside $S$.
\end{note}

%%%%%%%%%%%%%%%%%%%%%%%%%%%%%%%%%%%%%%%%

\subsection{Localisation at $p$ and syntomic cohomology}\label{ss:localisation}

\begin{proposition}
 The localisation of $\AF^{[\Pi, j]}_{\et}$ at $p$ lies in the Bloch--Kato subspace
 \[ H^1_\f\left(\Qp, V_p^{\As}(\Pi)^*(-j)\right) \subseteq H^1(\Qp, V_p^{\As}(\Pi)^*(-j)). \]
 In other words, it takes values in the rank $3$ subrepresentation $\left(\cF^+_1 + \cF^+_2\right)V_p^{\As}(\Pi)^*(-j)$.
\end{proposition}

\begin{proof}
 From the construction of $\AF^{[\Pi, j]}_{\et}$, it is clear that it lies in the image of the natural map from motivic to \'etale cohomology; and this map takes values in the Bloch--Kato $H^1_{\mathrm{g}}$ subspace, by a result of \Nek and \Niz \cite{nekovarniziol16}. For any de Rham Galois representation $W$ we have an inclusion $H^1_\f(\Qp, W) \subseteq H^1_{\mathrm{g}}(\Qp, W)$, and this is an equality unless $p^{-1}$ is an eigenvalue of $\varphi$ on $\Dcris(W)$. In our case, with $W = V_p^{\As}(\Pi)^*(-j)$, the eigenvalues of $\varphi$ have complex absolute value $p^{-1-t}$ with $t \ge 0$; so if $t > 0$ we are done.

 In the case $t = 0$, i.e. $k_1 = k_2 = j$, we must be slightly more circumspect. In this case, after multiplying by a non-zero factor depending on a choice of integer $c > 1$, $\Eis^{t}_{\mot, N}$ reduces to the Siegel unit ${}_c g_{0, 1/N}$ in the notation of \cite{kato04}. Since $(p, N) = 1$, this is a unit on the canonical model of $Y_1(N)$ over $\ZZ_{(p)}$ (not only over $\QQ$); see \cite[\S 1.3]{scholl98}. Hence $\AF^{[\Pi, j]}_{\et}$ lies in the image of the motivic cohomology of $Y_G$ over $\Zp$, implying that its \'etale realisation is in $H^1_\f$ in this case as well. (Compare \cite[Proposition 5.4.1]{KLZ20} in the analogous case of Beilinson--Flach elements.)
\end{proof}

Let $\eta_{\dR}$ be as in Proposition \ref{prop:eta}.  Our first goal will be to compute the pairing
\begin{equation}\label{eq:tocompute}
 \left\langle \eta_{\dR}, \log \left(\loc_p \AF^{[\Pi, j]}_{\et}\right)\right\rangle_{D_p^{\As}(\Pi)},
\end{equation}
where $\log$ is the Bloch--Kato logarithm \eqref{eq:expiso}, and $\langle - , - \rangle_{D_p^{\As}(\Pi)}$ denotes the canonical pairing between $D_p^{\As}(\Pi)$ and its dual.

Let $\YY_G$ denote the canonical $\Zp$-model of $Y_G$. This is a smooth $\Zp$-scheme, and we may choose an arithmetic toroidal compatification $\XX_G$ such that $(\YY_G, \XX_G)$ is a \emph{smooth pair} over $\Zp$ in the sense of \cite{KLZ20} (i.e.~$\XX_G$ is smooth, and the cuspidal divisor $\XX_G - \YY_G$ is a smooth normal-crossing divisor relative to $\Spec \Zp$). Thus Besser's theory of \emph{rigid syntomic} and \emph{finite-polynomial} cohomology applies to $\YY_G$, and rigid-syntomic cohomology has a natural comparison to \'etale cohomology.

\begin{notation}
 Let $P(T) = \left(1-\frac{T}{\alpha_1\alpha_2}\right) \left(1-\frac{T}{\alpha_1\beta_2}\right) \in L[T]$, and let $P_{1+j}(T) = P(p^{1 + j}T)$.
\end{notation}

\begin{proposition}
 \label{prop:defnufp}
 There is a unique lift $\eta_{\fp}$ of $\eta_{\dR}$ to the space
 \[ H^2_{\fp, c}\left(\YY_G, \cV_G; 1+j, P_{1+j}\right)[\Pif]. \]
\end{proposition}

\begin{remark}
 The above group is actually independent of $j$ in the range $0 \le j \le \min(k_1, k_2)$: Besser's cohomology for twist $r$ and polynomial $Q$ is defined using the mapping fibre of $Q(p^{-r}\varphi)$, and we have $P_{1+j}(p^{-1-j}\varphi) = P(\varphi)$ for any $j$. However, different values of $j$ will correspond to the \'etale cohomology of different twists of $V_p^{\As}(\Pi)$.
\end{remark}

\begin{proof}
 Since $\Pi$ is cuspidal, the $\Pif$-generalised eigenspace in de Rham (or, equivalently, rigid) cohomology vanishes outside degree 2. So the natural map
 \[  H^2_{\fp, c}\big(\YY_G, \cV_G; 1+j, P\big) \to \Fil^{(1 + j)} H^2_{\dR, c}\big(\YY_G, \cV_G\big)^{P(\varphi) = 0}\]
 is an isomorphism after localising at the $\Pif$-eigenspace. Since $P(\varphi)$ annihilates $\eta_{\dR}$ the result follows.
\end{proof}

The compatibility of \'etale and syntomic Abel--Jacobi maps for smooth pairs (cf.~Proposition 5.4.1 of \cite{KLZ20}) then implies that
\begin{equation}\label{eq:eqinNNfp}
 \begin{aligned} % one number for the whole equation-group, not per-line
  \left\langle \eta_{\dR}, \,  \log\left(\loc_p \AJ^{[\Pi, j]}_{\et}\right)\right\rangle_{D_p(\Pi)} &=
  \left\langle \eta_{\dR}, \,  \log\circ\pr^{\As}_\Pi\circ\iota^{[j]}_*\left(\Eis^{t}_{\et, N}\right)\right\rangle_{\dR, Y_G} \\
  &=\left\langle \eta_{\fp}, \, \iota^{[j]}_*\left(\Eis^{t}_{\syn, N}\right)\right\rangle_{\fp, \YY_G}\\
  & = \left\langle \iota^{[j], *}(\eta_{\fp}), \, \Eis^{t}_{\syn, N} \right\rangle_{\fp, \YY_H},
 \end{aligned}
\end{equation}
where the last equality follows from the adjunction between pushforward and pullback. Note that
\[  \iota^{[j], *}(\eta_{\fp})\in H^2_{\fp, c}(\YY_H, \cV_H; 1+ j, P_{1+j}), \]
and the coefficient module $\cV_H$ depends on $j$.

%%%%%%%%%%%%%%%%%%%%%%%%%%%%%%%%%%%%%%%%
\section{Computations over the $\pp_1$-ordinary locus}

In this section, we shall establish a relation between $\eta_{\dR}^{(1-\ord)}$ and $\eta^{(1-\ord)}$, and we will show that $\eta_{\fp}$ lifts to a class over the $1$-ordinary locus.

\subsection{The de Rham spectral sequence for $\cX_G^{(1-\ord)}$}
\label{sect:dRss}

Since $(\XX_G, \YY_G)$ is a smooth pair, and our coefficient system $\cV_G$ extends to a vector bundle on $\cX_G$ whose connection has log poles along the boundary divisor $D$, we can compute rigid cohomology of $Y_{G, 0}$ using the analytification of the BGG complex on $\cX_G$ (just as we did for de Rham cohomology above). By taking the mapping fibre of the restriction map we obtain the same result for compactly-supported cohomology of $Y_{G, 0}^{(1-\ord)}$; that is, we have
\[ R\Gamma_{\rig, c}(Y_{G, 0}^{(1-\ord)}, \cV) \cong R\Gamma_{\dR, c}(\cX_G^{(1-\ord)}, \cV\langle-D_G\rangle) \cong R\Gamma_c(\cX_G^{(1-\ord)}, \BGG^\bullet(-D_G)).\]

This gives rise to a first-quadrant spectral sequence converging to $H^*_{\rig, c}(Y^{(1-\ord)}_{G, 0}, \cV_G)$, whose $E_1^{mn}$ terms are $H^n_c\left(\cX_G^{(1-\ord)}, \BGG^m(-D_G)\right)$.

We denote by $\RGt_{\dR, c}(\cX_G^{(1-\ord)}, \cV\langle-D_G\rangle)$ the cohomology of the truncated complex
\[ \tau_{\ge 1} \BGG^\bullet(-D_G) = \left[0 \longrightarrow \omega^{(-k_1, k_2 + 2)} \oplus \omega^{(k_1 + 2, -k_2)} \longrightarrow \omega^{(k_1 + 2, k_2 + 2)}\right](-D),
\]
which is quasi-isomorphic to the filtered de Rham complex $\Fil^{1 + j} \left(\cV_G \otimes \Omega^\bullet_{X_G / L}\langle-D_G\rangle\right)$, for any $j$ in our range. Since $\cX_G$ is connected and non-compact, $H^0_c(\cX_G^{(1-\ord)}, -)$ is zero for all locally-free sheaves, and so we obtain an isomorphism
\[
\alpha_{\dR}^{(1-\ord)} :
H^1_c\left(\cX_G^{(1-\ord)}, \BGG^1(-D_G)\right)^{\nabla = 0}
\cong
\wH^2_{\dR, c}(\cX_G^{(1-\ord)}, \cV\langle-D_G\rangle).
\]
Moreover, the inclusion of the subcomplex $\tau_{\ge 1} \BGG$ into the full BGG complex gives a commutative square of maps
\[
\begin{tikzcd}
 \wH^2_{\dR, c}(\cX_G^{(1-\ord)}, \cV\langle-D_G\rangle)
 \rar \dar &
 \Fil^1 H^2_{\dR}(X_G, \cV\langle-D_G\rangle)_E
 \dar[hook] \\
 H^2_{\dR, c}(\cX_G^{(1-\ord)}, \cV\langle-D_G\rangle) \rar &
 H^2_{\dR}(X_G, \cV\langle-D_G\rangle)_E,
\end{tikzcd}
\]
in which the top horizontal arrow is compatible, via $\alpha_{\dR}^{(1-\ord)}$, with the natural map
\[
H^1_c\left(\cX_G^{(1-\ord)}, \BGG^1(-D_G)\right) \to H^1_c\left(\cX_G, \BGG^1(-D_G)\right)
= \Fil^1 H^2_{\dR} / \Fil^{k_1 + k_2 + 2}.
\]

Since the partial Frobenius $\varphi_1$ lifts to $\cX_G^{(1-\ord)}$, there is an action of $\varphi_1$ on both of the spaces in the left-hand column, compatible with the action on $ H^2_{\dR, c}(\cX_G^{(1-\ord)}, \cV\langle-D_G\rangle, 1+j)$ given by comparison with the rigid cohomology of $X_{G, 0}$.

\begin{proposition}
 \label{prop:nu1ord2}
 If $\eta^{(1-\ord)}$ is as in \cref{thm:nu1ord}, then the class $(\eta^{(1-\ord)}, 0)$ in \[ H^1_c\left(\cX_G^{(1-\ord)}, \BGG^1(-D_G)\right) = H^1_c\left(\cX_G^{(1-\ord)}, \omega^{-k_1, k_2 + 2}(-D_G)\right) \oplus H^1_c\left(\cX_G^{(1-\ord)}, \omega^{k_1 + 2, -k_2}(-D_G)\right)\]
 is in the kernel of $\nabla$, and hence defines a class in $\wH^2_{\dR, c}(\cX_G^{(1-\ord)}, \cV\langle-D_G\rangle)$. The image of this class in $H^2_{\rig, c}(Y_{G, 0}^{(1-\ord)}, \cV_G)$, under the left vertical map of the above diagram, is $\eta^{(1-\ord)}_{\dR}$.
 %   \begin{enumerate}[resume]
  %    \item it lies in the $\Pif$-eigenspace for the Hecke operators away from $\pp_1$ (including $T_{\pp_2}$);
  %    \item it is in the kernel of $\nabla$;
  %    \item its image in $H^2_{\rig, c}(Y_{G, 0}^{(1-\ord)}, \cV_G)$ is $\eta^{(1-\ord)}_{\dR}$;
  %    \item its image in $H^2_{\dR, c}(\cY_G, \cV_G)$ is $\eta_{\dR}$.
  %   \end{enumerate}
\end{proposition}

\begin{proof}
 We first show that $\eta^{(1-\ord)}$ is in the kernel of $\nabla$. This follows from the fact that it has strictly small slope for $\varphi_1$: the slopes of $\varphi_1$ on $\omega^{(k_1 + 2, k_2 + 2)}$ are all at least $k_1 + 1$, and the operator $\nabla$ commutes with the Frobenius, so it must be zero on all Frobenius eigenspaces of slope smaller than $k_1 + 1$.

 This shows that $\eta^{(1-\ord)}$ has a well-defined image in $H^2_{\rig, c}(Y_{G, 0}^{(1-\ord)}, \cV_G)$. Let us temporarily write $\hat{\eta}^{(1-\ord)}_{\dR}$ for this image; our goal is to show that it coincides with $\eta_{\dR}^{(1-\ord)}$. Since the latter is characterised as the unique lifting of $\eta_{\dR}$ compatible with Hecke operators away from $p\fN$, it suffices to show that $\hat{\eta}^{(1-\ord)}_{\dR}$ lies in the correct Hecke eigenspace, and that it maps to $\eta_{\dR}$ in $H^2_{\dR, c}(Y_G, \cV_G)$.

 It follows readily from the construction of $\eta^{(1-\ord)}$ that it lies in the $\Pi$-eigenspace for the Hecke operators away from $\pp_1$ (including $T_{\pp_2}$), since these operators commute with $\varphi_1$. So $\hat{\eta}^{(1-\ord)}_{\dR}$ has the correct Hecke action. Moreover, its image in $H^2_{\dR, c}(\cY_G, \cV_G)$ is in the $\varphi_1 = \alpha_1$ eigenspace (because $\eta^{(1-\ord)}$ is); and it lies in $\Fil^{(0, 1+k_2)}$, and maps to $\eta$ in $\Gr^{(0, 1 + k_2)}$, so it must be equal to $\eta_{\dR}$.
\end{proof}

%%%%%%%%%%%%%%%%%%%%%%%%%%%%%%%%%%%%%%%%

\subsection{FP-cohomology of the $\pp_1$-ordinary locus}

We now consider a modified form of Besser's finite-polynomial cohomology, namely \emph{Gros fp-cohomology}, for the $\pp_1$-ordinary locus (with compact supports):

\begin{definition}
 Define the Gros fp-cohomology\footnote{As before, this is in fact independent of $j$ in the stated range, despite the notations.}
 \[ \RGt_{\fp, c}(\cX_G^{(1-\ord)}, \cV\langle-D_G\rangle; 1+j, P_{1+j})\]
 to be the mapping fibre of the map
\[ \RGt_{\dR, c}(\cX_G^{(1-\ord)}, \cV\langle-D_G\rangle)
\xrightarrow{\ P_{1+j}(p^{-1-j}\varphi)\ } R\Gamma_{\rig, c}(X_{G, 0}^{(1-\ord)}, \cV\langle-D_G\rangle).\]
\end{definition}

\begin{note}
 Although $\varphi_1$ lifts to $\cX_G^{(1-\ord)}$, the full Frobenius $\varphi$ does not; so although $R\Gamma_{\rig, c}(X_{G, 0}^{(1-\ord)}, \cV\langle-D_G\rangle)$ is isomorphic to de Rham cohomology of $\cX_G^{(1-\ord)}$, the action of Frobenius (given by the functoriality of rigid cohomology) cannot be `seen' via this description.
\end{note}

For proper schemes such as $\XX_G$, there is no difference between Gros fp-cohomology and the usual fp-cohomology, so there is an extension-by-zero map
\[ \RGt_{\fp, c}(\cX_G^{(1-\ord)}, \cV\langle-D_G\rangle; 1+j, P_{1+j}) \to
R\Gamma_{\fp, c}(\XX_G, \cV_G\langle-D_G\rangle; 1+j, P_{1+j}).\]

\begin{proposition}\label{prop:lasteta}
 There exists a class
 \[ \tilde{\eta}_{\fp}^{(1-\ord)} \in \wH^2_{\fp, c}(\cX_G^{(1-\ord)}, \cV\langle-D_G\rangle; 1+j, P_{1+j}) \]
 with the following properties:
 \begin{itemize}
  \item Its image in $H^2_{\fp, c}(\XX_G, \cV_G\langle-D_G\rangle; 1+j, P_{1+j})$ is the $\eta_{\fp}$ of \cref{prop:defnufp}.
  \item Its image in $\wH^2_{\dR, c}(\cX_G^{(1-\ord)}, \cV\langle-D_G\rangle)$ is the class $(\eta^{(1-\ord)}, 0)$ of \cref{prop:nu1ord2}.
 \end{itemize}
\end{proposition}

\begin{remark}
 The reader may be relieved to hear that the class $\tilde{\eta}_{\fp}^{(1-\ord)}$ is ``the ultimate among liftings of $\eta$'': all other variants of $\eta$ will be images of this one.
\end{remark}

\begin{proof}
 From the mapping-fibre definition of Gros fp-cohomology we have a long exact sequence
 \[
 \dots \to H^{1}_{\rig, c}(X_{G, 0}^{(1-\ord)}, \cV\langle-D_G\rangle) \to
 \wH^2_{\fp, c}(\cX_G^{(1-\ord)}, \cV\langle-D_G\rangle; 1+j, P_{1+j}) \to \wH^2_{\dR, c}(\cX_G^{(1-\ord)}, \cV\langle-D_G\rangle) \to \dots, \]
 in which the boundary map is $P(\varphi) \circ \iota$. Moreover, this is compatible under extension-by-zero with the corresponding sequence for fp-cohomology of $\XX_G$.

 We have seen that the image of $(\eta^{(1-\ord)}, 0)$ under $\iota$ is the class $\eta_{\dR}^{(1-\ord)}$, which is annihilated by $P(\varphi)$. Hence it lifts to $\wH^2_{\fp, c}$. This lift is unique up to the image of an element of $H^{1}_{\rig, c}(X_{G, 0}^{(1-\ord)}, \cV\langle-D_G\rangle)$; but from \cref{prop:extbyzero} it follows that this group has trivial $\Pif$-eigenspace for the prime-to-$p$ Hecke operators, so there is a unique Hecke-equivariant lifting of $(\eta^{(1-\ord)}, 0)$ to Gros-fp cohomology. The image of this class under extension-by-0 is therefore a Hecke-equivariant lifting of $\eta_{\dR}$ to fp-cohomology of $\XX_G$, so it must be $\eta_{\fp}$.
\end{proof}

Now as shown in Note \ref{note:iota} we have $\iota^{-1}(\cX_G^{\ord}) = \cX_H^{\ord}$, so we deduce the following:

\begin{corollary}
 The pairing \eqref{eq:tocompute} is equal to $\left\langle \iota^{[j],*}\left(\tilde{\eta}_{\fp}^{(1-\ord)}\right), \widetilde{\Eis}^{t, \ord}_{\syn} \right\rangle$, where $\widetilde{\Eis}^{t, \ord}_{\syn}$ is the image of $\Eis^t_{\syn}$ in the Gros fp-cohomology of $\cX_H^{\ord}$.\qed
\end{corollary}

Here we define $\iota^{[j],*}$ for classes in Gros fp-cohomology using the quasi-isomorphism from the BGG complex to the full de Rham complex. We shall give explicit formulae in \cref{sect:compareBGG} below, but first we need to give an explicit form for $\tilde{\eta}_{\fp}^{(1-\ord)}$, which can only be done after restricting to $\cX_G^{\ord} \subset \cX_G^{(1-\ord)}$.

%%%%%%%%%%%%%%%%%%%%%%%%%%%%%%%%%%%%%%%%

\section{Restricting to the fully-ordinary locus}

\subsection{Cohomology with partial compact support}

We now consider the cohomology of the fully ordinary locus $X_{G, 0}^{\ord}$. Since the complement $X_{G, 0} - X_{G, 0}^{\ord}$ is the disjoint union of a closed subvariety $X_{G, 0}^{(1-\ss)}$ and the open subvariety $X_{G, 0}^{(2-\ss)} \cap X_{G, 0}^{(1-\ord)}$, we can apply the formalism of \cite[\S 13]{LZ20} to define ``cohomology of $\cX_G^{\ord}$ with compact support towards $\cX_G^{(1-\ss)}$'' (with coefficients in any abelian sheaf on $\cX_G$). We write this as $R\Gamma_{c1}(\cX_G, -)$. By construction, this comes equipped with a restriction map
\[ R\Gamma_c\left(\cX_G^{(1-\ord)}, -\right) \to R\Gamma_{c1}\left(\cX_G^{\ord}, -\right), \]
which fits into a triangle whose third term is the compactly-supported cohomology of $\cX_G^{(1-\ord)} \cap \cX_G^{(2-\ss)}$. In particular, by the same argument as \cref{prop:extbyzero}, we have isomorphisms
\[
R\Gamma_{\rig, c1}\left(\cX_G^{\ord}, \cV\langle-D_G\rangle\right)[\Pif] \longleftarrow
R\Gamma_{\rig, c}\left(\cX_G^{(1-\ord)},  \cV\langle-D_G\rangle\right)[\Pif] \longrightarrow
R\Gamma_{\rig, c}\left(\cX_G,  \cV\langle-D_G\rangle\right)[\Pif].
\]
The advantage of working with $\cX_G^{\ord}$ is that both $\varphi_1$ and $\varphi_2$ have liftings.

\begin{notation}\label{not:nuord}
 Write $\eta^{\ord} \in H^1_{c1}\Big(\cX_G^{\ord}, \omega^{(-k_1, k_2 + 2)}(-D_G)\Big)$ for the image of $\eta^{(1-\ord)}$ under the above restriction map.
\end{notation}

\begin{note}\label{note:Tpdecomp}
 Over the ordinary locus, we have commuting liftings of both $\varphi_1$ and $\varphi_2$, which both act on $R\Gamma_{c1}\left(\cX_G^{\ord}, \omega^{(\dots)}\right)$; and the operator $T_{\pp_2}$ decomposes as
 \(T_{\pp_2}=U_{\pp_2}+\varphi_2, \)
 with $U_{\pp_2} \circ \varphi_2 = p^{k_2 + 1} \langle \pp_2\rangle$.
\end{note}

\begin{corollary}\label{cor:inkerU2}
 The class $P(\varphi) \cdot \eta^{\ord}$ lies in the kernel of the Hecke operator $U_{\pp_2}$.
\end{corollary}

\begin{proof}
 We have $\varphi=\varphi_1\varphi_2$.
 The result follows easily from \cref{thm:nu1ord} and Note \ref{note:Tpdecomp}, using the fact that $\eta^{(1-\ord)}$ is a $T_{\pp_2}$-eigenvector.
\end{proof}

%%%%%%%%%%%%%%%%%%%%%%%%%%%%%%%%%%%%%%%%

\subsection{The Poznan spectral sequence}

We now recall a spectral sequence (introduced in \cite{LZ20}) relating Gros fp-cohomology to coherent cohomology. Here Gros fp-cohomology is defined in the same way as for the $\pp_1$-ordinary locus above, but now with $c1$-support.

\begin{definition}
 \label{def:cij}
 We define groups $\sC^{m, n}_{\fp, c1}(\cX_G^{\ord}, \cV\langle-D_G\rangle; 1+j, P_{1+j})$, for $m, n \ge 0$, by
 \[ \sC^{m, n}_{\fp, c1}(\dots) =
 H^n_{c1}\left(\cX_G^{\ord}, (\tau_{\ge 1} \BGG)^m(-D_G)\right) \oplus H^n_{c1}\left(\cX_G^{\ord}, \BGG^{m-1}(-D_G)\right);
 \]
 and we define differentials $\sC^{m, n}_{\fp, c1}(\dots) \to \sC^{m+1, n}_{\fp, c1}(\dots)$ by
 \[ (x, y) \mapsto \left(\nabla x, P(\varphi / p^n) \iota(x) - \nabla y\right), \]
 where $\iota$ is the inclusion of $\tau_{\ge 1} \BGG^\bullet$ into $\BGG^\bullet$, and $\nabla$ the differential of the BGG complex.
\end{definition}

Note that $\sC^{m, n}_{\fp, c1}$ is zero for $m \le 0$ (this is obvious for $m \le -1$, and holds for $m = 0$ since the truncated BGG complex vanishes in degree 0). It is also zero for $n \le 0$, since $H^0_{c1}$ vanishes for locally-free sheaves.

\begin{proposition}\label{prop:Poznan}
 There is a first-quadrant spectral sequence, the \emph{Pozna\'n spectral sequence}, with
 \[ {}^{\Pz}E_1^{mn} = \sC^{m, n}_{\fp, c1}(\cX_G^{\ord}, \cV\langle-D_G\rangle; 1+j, P_{1+j}), \]
 and the differentials on the $E_1$ page given by the formula above. This spectral sequence abuts to the Gros fp-cohomology $\wH^{m+n}_{\fp, c1}(\cX_G^{\ord}, \cV\langle-D_G\rangle; 1+j, P_{1+j})$.
\end{proposition}
\begin{proof}
 Choose double complexes computing $ \RGt_{\dR, c1}(\cX_G^{\ord}, \cV\langle-D_G\rangle)$ and $R\Gamma_{\rig, c1}(X_{G, 0}^{\ord}, \cV\langle-D_G\rangle)$, respectivly, in such a way that $P_{1+j}(\varphi)$ extens to a map of double complexes. Then $\wH^{m+n}_{\fp, c1}(\cX_G^{\ord}, \cV\langle-D_G\rangle; 1+j, P_{1+j})$ is computed by the total complex of the associated mapping fibre. The Pozna\'n spectral sequence is one of the spectral sequences associated to this triple complex.
\end{proof}

\begin{definition}
 We define a \emph{coherent fp-pair} (of degree $(m, n)$, twist $1+j$ and $c1$-support) to be an element of the kernel of the differential $\sC^{m, n}_{\fp, c1} \to \sC^{m+1, n}_{\fp, c1}$; we write the group of these as $\mathscr{Z}^{m, n}_{\fp, c1}(\cX_G^{\ord}, \cV\langle-D_G\rangle; 1+j, P_{1+j})$.
\end{definition}

Thus an fp-pair is a pair of elements
\[
x \in H^n_{c1}(\cX^{\ord}_G, \tau_{\ge 1}\BGG^m(-D_G)),
\qquad y \in H^n_{c1}(\cX^{\ord}_G, \BGG^{m-1}(-D_G))
\]
which satisfy
\begin{equation}
 \label{eq:fppair}
 \nabla(x)=0\qquad \text{and}\qquad \nabla(y)=P(p^{-1-j}\varphi)\iota(x).
\end{equation}

\begin{note}\label{note:kernabla}
 Given $x$, the equation \eqref{eq:fppair} does not determine the element $y$ uniquely: it is determined up to an element of $H^n_{c1}(\cX^{\ord}_G, \BGG^{m-1}(-D_G))^{\nabla=0}$.
\end{note}

\begin{proposition}
 For $0\leq j\leq \min\{k_1, k_2\}$, the spectral sequence gives rise to an isomorphism
 \[
 \alpha_{G, \fp}:  \mathscr{Z}^{1, 1}_{\fp, c1}(\cX_G^{\ord}, \cV\langle-D_G\rangle; 1+j, P_{1+j})
 \xrightarrow{\ \cong\ }
 \wH^2_{\fp, c1}(\cX^{\ord}_G, \cV\langle -D_G\rangle; 1+j, P_{1+j}).
 \]
\end{proposition}

\begin{proof}
 Since the $E_1$ page of the spectral sequence is zero for $m \le 0$, the term ${}^{\Pz} E_2^{1, 1}$ is the kernel of the differential on ${}^{\Pz} E_1^{1, 1}$, which is the group of fp-pairs. Since this is the only nonzero term with $m + n = 2$, and the incoming and outgoing differentials at ${}^{\Pz} E_r^{1, 1}$ are trivially zero for all $r \ge 1$, we conclude that ${}^{\Pz} E_2^{1, 1}$ coincides with $\wH^2$ of the abutment, as required.
\end{proof}

\begin{corollary}\label{cor:existencefppair}
 Every cohomology class in $\wH^2_{\fp, c1}(\cX^{\ord}_G, \cV\langle -D_G\rangle; 1+j, P_{1+j})$ can be uniquely represented by a coherent fp-pair of degree $(1, 1)$.
\end{corollary}

\begin{proof}
 Since the $E_1$ terms of the spectral sequence are supported in the region $m, n \ge 1$, there are no other terms of total degree 2 except $(m, n) = (1, 1)$; and clearly $E_2^{(1, 1)} = E_\infty^{(1, 1)}$ since the differentials on the $E_2$ page and beyond land outside this region.
\end{proof}

An exactly analogous argument shows that for the truncated de Rham cohomology groups $\wH^i_{\dR, c1}$ (the hypercohomology of $\tau_{\ge 1} \BGG^\bullet(-D_G)$), we have an isomorphism
\[H^1_{c1}(\cX^{\ord}_G, \BGG^1(-D_G)) \xrightarrow{\ \alpha_{G, \rig}\ }  \wH_{\dR, c1}^2(\cX^{\ord}_G, \cV\langle -D_G\rangle, 1+j).\]

\begin{lemma}\label{lem:fpdiagram}
 Let $0\leq j\leq\min\{k_1, k_2\}$. We have a commutative diagram
 \[
 \begin{tikzcd}[row sep=large, column sep=large]
  \mathscr{Z}^{1, 1}_{\fp, c1}(\cX^{\ord}_G, \cV\langle-D_G\rangle; 1+j, P_{1+j})
  \rar["\alpha_{G, \fp}"] \dar &
  \wH^2_{\fp, c1}(\cX^{\ord}_G, \cV\langle -D_G\rangle; 1+j, P_{1+j})
  \dar \\
  H^1_{c1}(\cX^{\ord}_G, \BGG^1(-D_G)) \rar["\alpha_{G, \rig}"] & \wH_{\dR, c1}^2(\cX^{\ord}_G, \cV\langle -D_G\rangle, 1+j)
 \end{tikzcd}
 \]
 where the vertical arrows are the natural projection maps.
\end{lemma}
\begin{proof}
 Clear from the constructions.
\end{proof}

%%%%%%%%%%%%%%%%%%%%%%%%%%%%%%%%%%%%%%%%

\subsection{Construction of a coherent fp-pair}

\begin{notation}\label{not:nuresord}
 Define
 \[\tilde\eta^{\ord}_{\fp}\in \wH^2_{\fp, c1}(\cX^{\ord}_G, \cV\langle -D_G\rangle; 1+j, P_{1+j}) \]
 to be the restriction of the class $\tilde\eta^{(1-\ord)}_{\fp}$ constructed above.
\end{notation}

\begin{corollary}
 There exists a uniquely determined class
 \[ \xi \in H^1_{c1}(\cX^{\ord}_G, \omega^{(-k_1, -k_2)}(-D_G)), \]
 which is independent of $j$, such that $(\eta^{\ord}, \xi)$ forms an fp-pair representing the class $\tilde\eta^{\ord}_{\fp}$, and such that $\xi$ lies in the $\Pi$-eigenspace for the Hecke operators away from $pN$.
\end{corollary}

\begin{proof}
 The existence of $\xi$ is immediate from Corollary \ref{cor:existencefppair} and Lemma \ref{lem:fpdiagram}; the independence from $j$ is clear by construction.

 Now, if $\xi'$ is another element such that $(\eta^{\ord}, \xi')$ also represents $\tilde\eta^{\ord}_{\fp}$, then
 \[ \xi-\xi'\in H^1_{c1}(\cX^{\ord}_G, \BGG^0(-D_G))^{\nabla=0} \cong H^1_{\rig, c1}(X_{G, 0}, \cV\langle -D_G\rangle).\]
 As we have seen, the $\Pif$-eigenspace in this cohomology is zero, so there is a unique choice of $\xi$ which is Hecke-equivariant.
\end{proof}

\begin{lemma}
 The element $\xi$ has the following properties:
 \begin{equation}\label{lem:propxi}
  \varphi_1 \cdot \xi=\alpha_1\, \xi\qquad \text{and}\qquad U_{\pp_2} \cdot \xi=0.
 \end{equation}
\end{lemma}
\begin{proof}
 We first show the latter statement. We deduce from Corollary \ref{cor:inkerU2} that
 \[ U_{\pp_2} \cdot \xi\in H^1_{c1}(\cX^{\ord}_G, \BGG^0(-D_G))^{\nabla=0}[\Pif].\]
 But as observed above, this space is zero. The former statement follows analogously, by considering the element $(\varphi_1-\alpha_1) \cdot \xi$.
\end{proof}

We now lift these classes from the BGG complex to the full de Rham complex:

\begin{definition}\label{def:nudot}\
 \begin{enumerate}
  \item Write $\dot\xi$ for the image of $\xi$ in $H^1_{c1}(\cX^{\ord}_G, \cV(-D_G))$.
  \item Similarly, for $0 \leq j \leq \min\{k_1, k_2\}$ write $\dot\eta^{\ord}_j$ for the image of $\eta^{\ord}$ in
  \[ H^1_{c1}(\cX^{\ord}_G, \Fil^j \cV \otimes \Omega^1\langle -D_G\rangle).\]
 \end{enumerate}
\end{definition}

\begin{lemma}
 The element $\dot\xi$, like $\xi$, also satisfies
 \begin{equation}\label{lem:propxibreve}
  \varphi_1 (\dot\xi)=\alpha_1\, \dot\xi\qquad
  \text{and}\qquad U_{\pp_2}(\dot\xi)=0.
 \end{equation}
\end{lemma}

The following result will be very important for the regulator evaluation:

\begin{proposition}\label{prop:xiinkerUp}
 We have
 \[ U_p\circ \iota^*(\dot\xi)=0.\]
\end{proposition}
\begin{proof}
 As maps $H^\bullet(\cX^{\ord}_G)\to H^\bullet(\cX^{\ord}_H)$, we have the following identity:
 \begin{equation}\label{eq:Heckeop}
  U_p\circ\iota^*\circ\varphi_{\pp_1}=\iota^*\circ (\langle \pp_1\rangle U_{\pp_2}).
 \end{equation}
 Since $\dot\xi$ is an eigenvector of $\varphi_{\pp_1}$ with non-zero eigenvalue, and it is in the kernel of $U_{\pp_2}$, the Proposition follows.
\end{proof}

%%%%%%%%%%%%%%%%%%%%%%%%%%%%%%%%%%%%%%%%
\section{Expression via coherent cohomology}

\subsection{Relating the BGG and de Rham complexes}
\label{sect:compareBGG}

We need to recall some formulae relating the de Rham and BGG complexes for $\GL_2$. Let $k \in \ZZ_{\ge 0}$. The BGG complex for weight $k \ge 0$ is given by $\left[\omega^{-k} \xrightarrow{\Theta} \omega^{k+2}\right]$, where $\Theta$ is a differential operator given in terms of $q$-expansions by
\[ \Theta = \tfrac{(-1)^k}{k!} \theta^{k+1}, \qquad \theta = q \tfrac{\mathrm{d}}{\mathrm{d}q}. \]
(Cf.~\cite[Remark 2.17]{tianxiao16} for example.) The map $\BGG^\bullet \to \DR^\bullet$ is given as follows.
\begin{itemize}
 \item In degree 1, it is given by the tensor product of the natural inclusion $\omega^k = \Fil^k (\Sym^k \cW_H) \into \cW_H$ and the Kodaira--Spencer isomorphism $\omega^2 \cong \Omega^1(D)$.
 \item In degree 0, a section $s$ of $\omega^{-k}$ is mapped to the unique section of $\Sym^k \cW_H$ whose image in $\Sym^k \cW_H / \Fil^0 \cong \omega^{-k}$ is $s$, and whose image under the differential $\nabla$ lands in $\omega^k \otimes \Omega^1(D)$.
\end{itemize}

We now recall (and somewhat reformulate) some results from \S 4 of \cite{KLZ20} giving a completely explicit description of these maps. We write $\sX^{\ord}_H$ for the ordinary locus \emph{as a classical rigid space} (not a dagger space), i.e.~neglecting overconvergence. Passing to the Igusa tower $\mathscr{I\!G}_H$ (the canonical $\Zp^\times$-covering of $\sX^{\ord}_H$ parametrising ordinary elliptic curves with a trivialization of their formal group), we obtain a canonical section $v$ of $\omega$, corresponding to the invariant differential form $\tfrac{dT}{T}$ on the Tate curve $\mathbb{G}_m / q^\ZZ$; note that $\varphi^*(v) = pv$. If we let $w$ be the unique section such that $\nabla(v) = u \otimes \xi$, where $\xi$ the local basis of $\Omega^1_{\sX^{\ord}_H}(D)$ corresponding to $\tfrac{\mathrm{d}q}{q}$, then $v$ and $w$ are a basis of sections of $\cW_H$ over the Igusa tower, with $w$ spanning the unit-root subspace (we have $\varphi^*(w) = w$) and $\nabla w = 0$. Hence we obtain a basis of sections $(v^a w^{k - a})_{0 \le a \le k}$ of $\Sym^k \cW_H$ in which the actions of $\nabla$ and $\varphi$ are completely explicit.

In these coordinates, the map from $\omega^{-k}$ to $\Sym^k \cW_H$ sends a section $g$ of $\omega^{-k}$ to the section
\[ \sum_{i = 0}^k \tfrac{(-1)^i}{i!} \theta^i(g) \cdot v^i w^{k-i}. \]
One verifies easily that the image of this under $\nabla$ is $\theta^{k+1}(g) v^k \otimes \xi$, as expected. (Note that it is not \emph{a priori} obvious that such a sum is overconvergent if $g$ is, since neither $\theta$ nor the local bases $v^i w^{k-i}$ have any meaning outside the ordinary locus.)

\subsection{The Eisenstein class as a coherent fp-pair}

Recall that the $\GL_2$ Eisenstein class $\Eis^t_{\syn, N}$ lies in $H^1_{\syn}(\YY_H, \cV_H(1 + t))$, where $\YY_H$ is the $\Zp$-model of the modular curve $Y_1(N)$, and $\cV_H$ is the sheaf corresponding to the $t$-th symmetric power of the standard representation.

If we restrict to the open subscheme $\YY_H^{\ord}$ given by removing the supersingular points of the special fibre (so $\YY_H^{\ord}$ has the same generic fibre as $\YY_H$), and work with Gros syntomic cohomology, then the Poznan spectral sequence applied to $\cX_1(N)^{\ord}$ gives the following explicit description:

\begin{proposition}\label{prop:synpairH}
    The Poznan spectral sequence gives rise to an isomorphism
    \[
    \alpha_{H, \syn}:  \mathscr{Z}^{1, 0}_{\syn}(\cX_H^{\ord}, \cV_H(t))
    \xrightarrow{\ \cong\ }
    \wH^1_{\syn}(\cX^{\ord}_H, \cV_H(t)).\]
\end{proposition}

Explicitly, a class $x \in \wH^1_{\syn}(\cX_1(N)^{\ord}, \cV_H, 1+t)$ is given by a coherent syntomic pair $(x_0, x_1)$, where
\[
x_0 \in H^0(\cX_1(N)^{\ord}, \cV), \qquad x_1 \in H^0(\cX_1(N)^{\ord}, \Fil^t \cV \otimes \Omega^1\langle D\rangle), \qquad \nabla x_0 = (1 - p^{-1-t}\varphi) x_1.
\]

\begin{remark}
    Applying the Poznan speactral sequence to Gros fp-cohomology of $\cX_1(N)^{\ord}$ with compact support, we obtain an isomorphism
    \[  \alpha_{H, \fp}: \mathscr{Z}^{1,1}_{\fp,c}(\cX^{\ord}_H, \cV\langle -D_H\rangle, 2; P_{1 + j}) \xrightarrow{\ \cong\ } \wH^2_{\fp,c}(\cX^{\ord}_H, \cV\langle -D_H\rangle, 2; P_{1 + j}). \qedhere\]
\end{remark}

The sheaf $\Fil^t \cV \otimes \Omega^1\langle D\rangle$ is simply $\omega^{t+2}$, so $x_1$ is an overconvergent $p$-adic modular form of weight $t + 2$; via this description, $\varphi$ acts on overconvergent forms as $p^{t + 1} \langle p \rangle V_p$, where $\langle p \rangle$ is the diamond operator for $p \bmod N$.

We let $\widetilde\Eis^{t, \ord}_{\syn, N}$ be the image of $\Eis^t_{\syn, N}$ in this group; we shall now write down an explicit representing pair. The following is a reformulation of Theorem 4.5.7 of \cite{KLZ20}; the formulations differ because we are using symmetric powers here rather than symmetric tensors (the basis vector $v^{[r, s]}$ of \emph{op.cit.} corresponds to $\tfrac{v^s w^r}{r! s!}$ in our present notation) and because we use a slightly different notation for Eisenstein series following \cite{kato04}.

\begin{proposition}\label{prop:Eisexplicit}
 The Eisenstein class $\widetilde\Eis^{t, \ord}_{\syn, N}$ is represented by the coherent syntomic pair $\left(\epsilon_0^t, \epsilon_1^{t}\right)$ whose restrictions to $\sX_H^{\ord}$ are
 \begin{align*}
  \epsilon_0^t & = -N^k \sum_{u=0}^t\frac{(-1)^{t-u}}{u!}\theta^{u}\left(E^{-t, \ord}_{0, 1/N}\right)\cdot v^{t-u}w^u, \\
  \epsilon_1^t & = -\tfrac{N^t}{t!} F^{t+2}_{0, 1/N}\cdot v^t \otimes \xi.
 \end{align*}
 Here $F^{t+2}_{0, 1/N}$ is the algebraic Eisenstein series with $q$-expansion
 \[ \zeta(-1-t) + \sum_{n > 0} q^n \sum_{d|n}(\tfrac{n}{d})^{t + 1}(\zeta_N^d + (-1)^t \zeta_N^{-d}), \]
 and $E^{-t, \ord}_{0, 1/N}$ is an ordinary $p$-adic Eisenstein series of weight $-t$, satisfying $\theta^{t+1} (E^{-t, \ord}_{0, 1/N}) = (1 - \langle p \rangle V_p)F^{t+2}_{0, 1/N}$.\qed
\end{proposition}

\begin{remark}
 Note that the individual terms $\theta^u(E^{-t, \ord}_{0, 1/N})$ are $p$-adic modular forms, but they are not overconvergent (unless $u =  0$). Nonetheless, $\epsilon_0$ is an overconvergent section of $\cV_H$; the non-overconvergence arises because the sections $v, w$ we are using to trivialize $\cV_H$ over the ordinary locus are not themselves overconvergent.
\end{remark}

%%%%%%%%%%%%%%%%%%%%%%%%%%%%%%%%%%%%%%%%
\subsection{Explicit formulae for the Clebsch--Gordan map}

The map $\operatorname{CG}^{[j]}$ is given explicitly by the following formula (c.f. \cite[Prop. 5.1.2]{KLZ20})\footnote{The factorials appear slightly different from \emph{op.cit.} since we are here working with symmetric powers $v^m$ rather than symmetric tensors $v^{[m]}$. This is also the reason for the presence of $t!$ in the formula for the Eisenstein series.}: for $0\leq s\leq t \coloneqq k_1 + k_2 = 2j$, we have
\begin{multline}
 \label{eq:CGexplicit}
 \iota^{[j]}(v^sw^{t - s})= \\
 \sum_{\substack{0 \le r \le k_1-j \\ 0 \le r' \le k_2-j \\ r+r'=s}}\sum_{i=0}^j(-1)^i \frac{s!(t-s)!}{r!(r')!(k_1-r-j)!(k_2-r'-j)!i!(j-i)!}v^{r+i}w^{k_1-r-i}\boxtimes v^{r'+j-i}w^{k_2-r'-j+i}\otimes e_{-j}.
\end{multline}
\begin{lemma}
 For given values of $k_1, k_2, j$, the image of $\iota^{[j]}(v^sw^{t - s})$ in the line spanned by the basis vector $v^{k_1} \boxtimes w^{k_2} \otimes e_{-j}$ is zero for all $s$ except $s = k_1-j$, in which case it is equal to $\tfrac{(-1)^j}{j!} \cdot v^{k_1} \boxtimes w^{k_2} \otimes e_{-j}$.
\end{lemma}

\begin{proof}
 This is a superficially modified version of Proposition 5.1.2 of \cite{KLZ20}.
\end{proof}

\begin{remark}
 In particular, this shows that the Clebsch--Gordan map is not in general defined integrally (i.e.~does not respect the lattice given by the $\ZZ$-span of the basis vectors). However, the coefficient $j!$ is the worst possible -- one can check that $j! \iota^{[j]}$ is integral.
\end{remark}

%%%%%%%%%%%%%%%%%%%%%%%%%%%%%%%%%%%%%%%%

\subsection{Reduction to a pairing in coherent cohomology}

Recall that we want to evaluate the pairing
\begin{equation}\label{eq:toevaluate}
 \left\langle (\iota^{[j]})^*(\eta_{\fp}), \, \Eis^{t}_{\syn, N}\right\rangle.
\end{equation}
We first observe that
\begin{align}
 \text{\eqref{eq:toevaluate}} & = \left\langle (\iota^{[j]})^*\left(\eta^{(1-\ord)}_{\fp}\right), \, \Eis^{t, \ord}_{\syn, N}\right\rangle. \notag\\
 & =  \left\langle (\iota^{[j]})^*\left(\tilde\eta^{(1-\ord)}_{\fp}\right), \, \widetilde\Eis^{t, \ord}_{\syn, N}\right\rangle\notag \\
 & =  \left\langle (\iota^{[j]})^*\left(\tilde\eta^{\ord}_{\fp}\right), \, \widetilde\Eis^{t, \ord}_{\syn, N}\right\rangle\label{eq:cohpairing}.
% & =
\end{align}
Here,
\eqref{eq:cohpairing} takes values in the group
\[ \wH^3_{\fp, c}(\cX^{\ord}_H, \Qp\langle -D_H\rangle, 2; P_{1 + j})\cong^{\tr} \Qp.\]
Now using the isomorphisms $\alpha_{G, \fp}$ and $\alpha_{H, \fp}$ we can express
\[ (\iota^{[j]})^*\left(\tilde\eta^{\ord}_{\fp}\right) \in \wH^2_{\fp,c}(\cX^{\ord}_H, \cV\langle -D_H\rangle, 1+j; P_{1 + j}) \]
as the coherent fp-pair
\[\left( \iota^{[j]})^*( \dot\eta^{\ord}_j), \,  \iota^{[j]})^*( \dot\xi)\right)\in  \mathscr{Z}^{1,1}_{\fp,c}(\cX^{\ord}_H, \cV\langle -D_H\rangle, 1+j; P_{1 + j}).\]
We recall the following result from  \cite[Lemma 18.2.1]{LZ20}:

\begin{lemma}
 Using Besser's formalism we can define a pairing
 \[ \mathscr{Z}^{1,1}_{\fp,c}(\cX^{\ord}_H, \cV\langle -D_H\rangle, 1+j; P_{1 + j})\, \times\,  \mathscr{Z}^{1, 0}_{\syn}(\cX_H^{\ord}, \cV_H(t))\to \QQ_p\]
 which is compatible
 \[ \wH^2_{\fp,c}(\cX^{\ord}_H, \cV\langle -D_H\rangle, 1+j; P_{1 + j})\, \times\,  \wH^1_{\syn}(\cX_1(N)^{\ord}, \cV_H, t)\to \QQ_p\]
 under the isomorphisms $\alpha_{H, \fp}$ and $\alpha_{H, \syn}$.
\end{lemma}

 As a direct consequence, we obtain

\begin{proposition}
    We have
    \begin{equation}\label{eq:cohpairing1}
    \text{\eqref{eq:cohpairing}}=  \left\langle \left( (\iota^{[j]})^*( \dot\eta^{\ord}_j), \,  \iota^{[j]})^*( \dot\xi)\right), \, (\epsilon_1, \epsilon_0)\right\rangle.
    \end{equation}
\end{proposition}

%
%  \begin{note}
 %   The isomorphism
 %   \[ \tr: \wH^3_{\fp, c}(\cX^{\ord}_H\langle -\cD_H\rangle, \Qp, 2;P_j)\cong \Qp\]
 %   is given by $P_j(p^{-1})\cdot\tr_{\rig}$, where $\tr_{\rig}:H^2_{\rig, c}(\cX^{\ord}_H\langle -\cD_H\rangle, \Qp)\cong\Qp$ is the trace map on rigid cohomology (c.f. \cite[\S 2.5]{KLZ20}). Explicitly, we have
 %   \begin{equation}\label{eq:tracefactor}    P_j(p^{-1})=\left(1-\frac{p^j}{\alpha_1\alpha_2}\right)\left(1-\frac{p^j}{\alpha_1\beta_2}\right).\qedhere
  %   \end{equation}
 %  \end{note}

We evaluate \eqref{eq:cohpairing1} by making the formalism of cup products in fp-cohomology explicit (cf.~\cite{besser00a}). We have
\[ P_{1 + j}(xy)= a(x, y) P_{1 + j}(x)+b(x, y)(1-y)\]
with $a(x, y) = y$ and
\[ b(x, y)=\frac{P_{1+j}(xy)-y P_{1 + j}(x)}{1-y}=1-b\cdot x^2 y, \qquad b = \frac{p^{2j+2}}{\alpha_1^2\alpha_2\beta_2} = \frac{\beta_1}{\alpha_1} \cdot \frac{1}{p^{t} \chi_{\Pi}(p)}. \]
Hence
\begin{align*}
 & P_{1 + j}(p^{-1})\times \left\langle (\iota^{[j]})^*\left( \dot\eta^{\ord}_j, \dot\xi\right), \, (\epsilon_1^{t}, \epsilon_0^{t})\right\rangle\\
 & = (\iota^{[j]})^*(\dot\xi)\cup \varphi_H^*(\epsilon_1^{t}) +
 \left(1 - \frac{\beta_1 \cdot (\varphi_H^{*, 2}\otimes \varphi_H^*)}{\alpha_1 \cdot p^{t} \chi_{\Pi}(p)} \right)
 \left((\iota^{[j]})^*(\dot\eta^{\ord}_j)\cup \epsilon_0^{t}\right).
\end{align*}

\begin{lemma}\label{lem:zeroterm}
 We have $(\iota^{[j]})^*(\dot\xi)\cup \varphi_H^*(\epsilon_0^{t})=0$.
\end{lemma}
\begin{proof}
 Observe that $U_p$ acts on the top-degree rigid cohomology as multiplication by a power of $p$. But we also have
 \[
 U_p\left((\iota^{[j]})^*(\dot\xi)\cup \varphi_H^*(\epsilon_0^{t})\right)=
 U_p\left((\iota^{[j]})^*(\dot\xi)\right)\cup \epsilon_0^{t},
 \]
 which is equal to zero by Proposition \ref{prop:xiinkerUp}. The Proposition follows.
\end{proof}

\begin{proposition}\label{prop:valuepairing}
 Equation \eqref{eq:cohpairing} is equal to
 \[ \frac{\left(1-\tfrac{\beta_1}{p\alpha_1}\right)}{\left(1-\frac{p^j}{\alpha_1\alpha_2}\right)\left(1-\frac{p^j}{\alpha_1\beta_2}\right)}\left\langle(\iota^{[j]})^*(\dot\eta^{\ord}_j), \epsilon_0^{t}\right\rangle.\]
\end{proposition}

\begin{proof}
 Since any $p$-depleted form will pair to 0 with a form that's in a direct sum of finite-slope $\varphi$-eigenspaces, we only care about the $\epsilon_0$ term modulo exact forms. We also only care about its projection to the eigenspace where the diamond operators act by $\chi_{\Pi}^{-1} |_{\QQ}$, because we are pairing it with a class in the $\Pi$-eigenspace.

 On the space $H^2_{\rig, c}(\cX_H^{\ord}\langle -D_H\rangle, \Qp(2))$, the Frobenius acts as $p^{-1}$, and $\langle p \rangle$ acts trivially; so we can replace $\frac{\beta_1 \cdot (\varphi_H^{*, 2}\otimes \varphi_H^*)}{\alpha_1 p^{t} \chi_{\Pi}(p)}$ with $\frac{\beta_1 \cdot (1 \otimes \langle p \rangle \varphi_H^{-1})}{\alpha_1 p^{t+2}}$.

 The operator $\varphi^{-1}$ makes sense modulo $p$-depleted forms, and in this quotient we have
 \[ \varphi^{-1} \cdot \epsilon_0^{t} = p^{t+1} \langle p \rangle^{-1} \epsilon_0^{t}. \]
 So we are done.
\end{proof}

\begin{remark}
 Note that $\dot\eta^{\ord}_j$ naturally extends to a class defined over the $\pp_1$-ordinary locus $\cX_G^{(1-\ord)}$; the antiderivative $\dot\xi$ is only defined over the fully ordinary locus $\cX_G^{\ord}$, but this term has disappeared from our formula. So we can also interpret the pairing of \cref{prop:valuepairing} as a cup-product in the cohomology of $\cX_G^{(1-\ord)}$, which will allow us to compare with the construction of $p$-adic $L$-functions from \cite{grossiloefflerzerbesLfunct}.
\end{remark}

%%%%%%%%%%%%%%%%%%%%%%%%%%%%%%%%%%%%%%%%

 \subsection{A partial unit root splitting}

  We saw above that over $\sX_H^{\ord}$ (the ordinary locus as a classical rigid space) the Hodge filtration of $\cV_H$ has a canonical splitting given by the unit-root subspace. For $X_G$ we have the more refined structure of a $\ZZ^2$-filtration, and we can ask for splittings of either factor. We state below the case of interest.

  \begin{proposition}
   Over $\sX_G^{(1-\ord)}$, the natural inclusion map $\omega^{(k_1, -k_2)} \into \cV_G / \Fil^{k_1 + 1}$ admits a canonical splitting.
  \end{proposition}

  \begin{proof}
   We have a projection map from $\cV_G$ onto $\cV_G / \Fil^{(0, 1)} \cV_G \cong \Sym^{k_1} \cW_G \boxtimes \omega^{-k_2}$. So it suffices to show that the filtration on $\Sym^{k_1} \cW_G$ is splittable over the 1-ordinary locus; but this is clear, since a complement is provided by the unit-root subspace for the operator $\varphi_1$.
  \end{proof}

  We now use the unit-root splitting to relate the pushforward maps $\iota_*^{[j]}$ on the cohomology of the de Rham sheaves to the pushforward maps on coherent cohomology defined in \cref{note:iota} above (which we write here as $\iota_{\coh, *}$, to distinguish them from the de Rham versions)

  \begin{notation}
   We use the notation of Section \ref{ss:ClebschGordan}, and we write
   \begin{itemize}
    \item $\cV_G=\Sym^{[k_1, k_2]}\cH(\cA)_{\Qp}$;
    \item $\cV_H=\Sym^{t}\cH(\cE)_{\Qp}$ for $0\leq j\leq \min\{k_1, k_2\}$.
   \end{itemize}
   for for vector bundles attached to $V_H$ and $V_G$, respectively.
  \end{notation}

  \begin{proposition}\label{prop:commdiagsheaves}
   Pushforward along $\iota^{[j]}$ induces a commutative diagram
   \[
   \begin{tikzcd}
    H^0(\sX_H^{\ord}, V_H/\Fil^{k_1-j+1}V_H)
    \rar["\iota^{[j]}_*"] &
    H^1(\sX^{(1-\ord)}_G, V_G/\Fil^{k_1+1}V_G\otimes\Omega^1\langle D \rangle) \\
    H^0(\sX_H^{\ord}, \Gr^{k_1-j}V_H)
    \rar \uar \ar[rd, "\tfrac{1}{j!} \iota_{\coh, *}" below] &
    H^1(\sX^{(1-\ord)}_G, \Gr^{k_1}V_G\otimes\Omega^1\langle D \rangle)
    \dar \uar \\
    & H^1(\sX^{(1-\ord)}_G, \omega^{k_1 + 2, -k_2}).
   \end{tikzcd}
   \]
  \end{proposition}

  \begin{proof}
   For the commutativity of the top square, we note that the Clebsch--Gordan map is compatible with filtrations and hence induces a commutative diagram of $B_H$-representations
   \[
   \begin{tikzcd}
    V_H/\Fil^{k_1-j}V_H
    \rar &
    V_G/\Fil^{k_1}V_G \\
    \Gr^{k_1-j}V_H
    \rar \uar &
    \Gr^{k_1}V_G.
    \uar
   \end{tikzcd}
   \]
   The identification of the diagonal map as $\tfrac{1}{j!} \iota_{\coh, *}$ arises similarly, using the explicit formula \eqref{eq:CGexplicit}.
  \end{proof}

  \begin{proposition}
   \label{prop:urdiag}
   The diagram in Proposition \ref{prop:commdiagsheaves} is compatible with the unit root splitting
   \[
    u_H : H^0(\sX_H^{\ord}, V_H/\Fil^{k_1-j+1}V_H)
    \longrightarrow H^0(\sX_H^{\ord}, \Gr^{k_1-j}V_H)
   \]
   and the partial unit root splitting
   \[
    u_G : H^1(\sX^{(1-\ord)}_G, V_G/\Fil^{k_1+1}V_G\otimes\, \Omega^1)
    \longrightarrow H^1(\sX^{(1-\ord)}_G, \omega^{k_1 + 2, -k_2}):
   \]
   we have a commutative diagram
   \[
   \begin{tikzcd}[column sep=large]
    H^0(\sX_H^{\ord}, V_H/\Fil^{k_1-j+1}V_H)
    \rar["\iota_*^{[j]}"]  \arrow[dashed, bend right=90, "u_H" left]{d}&[-1.5em]
    H^1(\sX^{(1-\ord)}_G, V_G/\Fil^{k_1+1}\otimes\, \Omega^1)
    \arrow[dashed, bend left=90, "u_G" right]{dd} \\
    H^0(\sX_H^{\ord}, \Gr^{k_1-j}V_H)
    \rar \uar \ar[rd, "\tfrac{1}{j!} \iota_{\coh, *}" below] &
    H^1(\sX^{(1-\ord)}_G, \Gr^{k_1}V_G\otimes\Omega^1)
    \uar \dar \\
    &
    H^1(\sX^{(1-\ord)}_G, \omega^{k_1 + 2, -k_2})
   \end{tikzcd}
   \]
  \end{proposition}

  \begin{proof}
   Over $\sX_H^{\ord}$, we have a canonical splitting of the Hodge filtration of $\cW_H$, as above. To prove the proposition, it is sufficient to check that the unit root splittings induce a commutative diagram
   \[
   \begin{tikzcd}
    H^0(\sX_H^{\ord}, V_H/\Fil^{k_1-j+1}V_H)
    \rar  \arrow[dashed, bend right=90]{d}&[-1.5em]
    H^0\left(\sX^{\ord}_H, \iota^*(V_G/\Fil^{k_1+1}\otimes\, \Omega^1)\right)
    \arrow[dashed, bend left=90]{dd}
    \\
    H^0(\sX_H^{\ord}, \Gr^{k_1-j}V_H)
    \rar \uar \arrow{rd} &
    H^0\left(\sX^{\ord}_H, \iota^*(\Gr^{k_1}V_G\otimes\Omega^1)\right)
    \uar \dar \\
    &
    H^0\left(\sX^{\ord}_H, \iota^*(\omega^{(k_1 + 2, -k_2)}\right)
   \end{tikzcd}
   \]
   which boils down to an explicit computation with the Clebsch--Gordan map using equation \eqref{eq:CGexplicit}.
  \end{proof}

  It follows from the classicity theorem of higher Hida theory proved in \cite{grossiloefflerzerbesLfunct} that the natural restriction map (forgetting overconvergence)
  \[
   H^1(\cX^{(1-\ord)}_G, \omega^{k_1 + 2, -k_2})
   \to H^1(\sX^{(1-\ord)}_G, \omega^{k_1 + 2, -k_2})
  \]
  is an isomorphism on the ordinary (slope 0) eigenspace for $U_{\pp_1}$. If $\Pi$ is ordinary at $\pp_1$ (and $\alpha_1$ is, necessarily, the unit root) then pairing with $\eta^{(1-\ord)}$ factors through this eigenspace, so we obtain the following:

  \begin{corollary}\label{cor:cohexplicit}
   Assume $\Pi$ is ordinary at $\pp_1$. Then the linear functional on $H^0(\sX_H^{\ord}, V_H/\Fil^{k_1-j+1}V_H)$ given by pairing with $(\iota^{[j], *})(\dot{\eta}^{\ord})$ factors through the unit root splitting into $:H^0(\sX_H^{\ord}, \Gr^{k_1-j}V_H)$, and it is given by
   \[
   \frac{k_1!\, k_2!}{j!}\cdot\, \left\langle\eta^{(1-\ord)}, \iota_{\coh, *}(-)\right\rangle_{\sX_G^{(1-\ord)}}. \qed
   \]
  \end{corollary}

  \begin{remark}
   The factors $k_i!$ arise from the pairing of the basis vectors $v^{k_i}$ and $w^{k_i}$ of $\Sym^{k_i} W_G$. %% Is this correct? -- DL
  \end{remark}

%%%%%%%%%%%%%%%%%%%%%%%%%%%%%%%%%%%%%%%%

\subsection{Relation to $p$-adic $L$-functions}

We use Corollary \ref{cor:cohexplicit} in order to relate the formula in Proposition \ref{prop:valuepairing} to values of $p$-adic $L$-functions. We assume henceforth that $\Pi$ is ordinary at $\pp_1$.

\begin{proposition}\label{prop:coheval}
 The pairing \eqref{eq:toevaluate} is given by
 \begin{multline*}
  \left\langle (\iota^{[j]})^*(\eta_{\fp}), \, \Eis^{t}_{\syn, N}\right\rangle=\\
  \frac{N^{(k_1 + k_2 - 2j)} (-1)^{k_2 - j + 1} \left(1-\tfrac{\beta_1}{p\alpha_1}\right)}
  {\left(1-\frac{p^j}{\alpha_1\alpha_2}\right)\left(1-\frac{p^j}{\alpha_1\beta_2}\right)}
  \binom{k_1}{j}\, k_2!\, \cdot
  \left(\eta^{\ord}\cup \iota_{\coh, *}(\theta^{(k_1-j)}E_{0, 1/N}^{-t, \ord})\right).
 \end{multline*}
\end{proposition}

Compare Corollary 6.5.7 in \cite{KLZ20}.

\begin{proof}
 We deduce from Proposition \ref{prop:Eisexplicit} that the image of $\epsilon_0^t$ under projection to $H^0(\cX^{\ord}_H, \Gr^{k_1-j}\cV_H)$ is given by $-N^k (-1)^{k_2-j}\frac{(t)!}{(k_1-j)!} \theta^{k_1 - j}(E^{-t, \ord}_{0, 1/N})$. Combining this with \cref{prop:valuepairing} and \cref{cor:cohexplicit} gives the result.
\end{proof}
%
%\begin{note}
% By adjunction, we can write the formula in Proposition \ref{prop:coheval} as
% \begin{multline*}
%  \left\langle (\iota^{[j]})^*(\eta_{\fp}), \, \Eis^{t}_{\syn, N}\right\rangle=\\
%  \frac{N^{(k_1 + k_2 - 2j)} (-1)^{k_2 - j + 1} \left(1-\tfrac{\beta_1}{p\alpha_1}\right)}
%  {\left(1-\frac{p^j}{\alpha_1\alpha_2}\right)\left(1-\frac{p^j}{\alpha_1\beta_2}\right)}
%  \binom{k_1}{j}\, k_2!\, \cdot
%  \left(\iota_{\coh}^*(\eta^{\ord})\cup \theta^{(k_1-j)}E_{0, 1/N}^{-t, \ord}\right).\qedhere
% \end{multline*}
%\end{note}

In order to relate this formula to a (non-critical) value of a $p$-adic $L$-function, we need to replace
$E_{0, 1/N}^{-t, \ord}$
by its $p$-depletion
\[ E^{-t,[p]}_{0,1/N}=(1-\varphi \circ U_p)E^{-t, \ord}_{0,1/N}.\]
We adapt the argument from \cite[\S 6.5]{KLZ20}. Let $V_{\pp_2}=p^{-1-k_2}\langle \pp_2\rangle^{-1}\varphi_2$, so $V_{\pp_2}$ is a right inverse of $U_{\pp_2}$. Let $\lambda, \mu$ be constants such that
\[ U_{\pp_2}(\eta^{\ord})=\lambda\eta^{\ord}-\mu V_{\pp_2}(\eta^{\ord})\]
(explicitly, we have $\lambda=a_{\pp_2}(\cF)$ and $\mu=p^{1+k_2}$), and let $\gamma=p^{k_1-j}$.
Using an analogue of Lemma 6.5.8 from \emph{op.cit.}, we deduce the following result:

\begin{lemma}
 We have
 \[
 \iota^*_{\coh}(\eta^{\ord})\cup \theta^{k_1-j}E^{-t,[p]}_{0,1/N}=\left(1-\lambda\gamma p^{-k_1}\beta_{\pp_1}\cdot V_p+\mu\gamma^2(p^{-k_1}\beta_{\pp_1})^2V_p^2\right) \left( \iota^*_{\coh}(\eta^{\ord})\cup \theta^{k_1-j}E^{-t, \ord}_{0,1/N}\right).
 \]
\end{lemma}

\begin{proof}
 Arguing as in \S 6.5 in \emph{op.cit.}, we see that
 \begin{multline*}
  \iota^*_{\coh}(\eta^{\ord})\cup \theta^{k_1-j}E^{-t,[p]}_{0,1/N}=\left(1-\lambda\gamma  p^{-k_1}\beta_{\pp_1}\cdot V_p+\mu\gamma^2(p^{-k_1}\beta_{\pp_1})^2V_p^2\right)\left( \iota^*_{\coh}(\eta^{\ord})\cup \theta^{k_1-j}E^{-t, \ord}_{0,1/N}\right)\\
  +(1-V_pU_p)\left(\mu \varphi^2\iota^*_{\coh}(\eta^{\ord})\cup \theta^{k_1-j}E^{-t, \ord}_{0,1/N}-\iota^*_{\coh}(\eta^{\ord})\cup\varphi \theta^{k_1-j}E^{-t, \ord}_{0,1/N}\right).
 \end{multline*}
 But the operator $1-V_pU_p$ acts as the zero map, which proves the result.
\end{proof}

We deduce the following formula. We suppose that $p^{1 + j} \notin \{ \beta_1 \alpha_2, \beta_1 \beta_2\}$ (which is immediate unless $k_1 = k_2 = j$). Then

\begin{proposition}\label{prop:exp}
 If $p^{1 + j} \notin \{ \beta_1 \alpha_2, \beta_1 \beta_2\}$, then we have
 \begin{multline*}
  \left\langle (\iota^{[j]})^*(\eta_{\fp}), \, \Eis^{t}_{\syn, N}\right\rangle=\\
  \frac{N^{(k_1 + k_2 - 2j)} (-1)^{k_2 - j + 1} \binom{k_1}{j}\, k_2! \left(1-\tfrac{\beta_1}{p\alpha_1}\right)}
  {\left(1-\frac{p^j}{\alpha_1\alpha_2}\right)\left(1-\frac{p^j}{\alpha_1\beta_2}\right) \left(1-\tfrac{\beta_1\alpha_2}{p^{j+1}}\right) \left(1-\tfrac{\beta_1\beta_2}{p^{j+1}}\right)}
  \left(\iota_{\coh}^*(\eta^{\ord})\cup \theta^{k_1-j}E_{0, 1/N}^{-t, [p]}\right).
 \end{multline*}
\end{proposition}

We can also express this in terms of the class $\eta_{\Iw(\pp_1)}^{(1-m)}$ on the higher-level variety $\cX_{G, \Iw(\pp_1)}$, using the results of Sections \ref{ss:pfwd} and  \ref{ssect:1parampL}.
    We have
    \begin{align*}
    \iota_{\coh}^*(\eta^{\ord})\cup \theta^{k_1-j}E_{0, 1/N}^{-t, [p]}&= \eta^{(1-\ord)}\cup \iota_{\coh,*}\left(\theta^{k_1-j}E_{0, 1/N}^{-t, [p]}\right)\\
    &= \pi_{\pp_1,*}(\eta_{\Iw(\pp_1)}^{(1-m)})\cup \iota_{\coh,*}\left(\theta^{k_1-j}E_{0, 1/N}^{-t, [p]}\right)\\
    &=\eta_{\Iw(\pp_1)}^{(1-m)}\cup ( \pi_{\pp_1}^*\circ \iota_{\coh,*})\left(\theta^{k_1-j}E_{0, 1/N}^{-t, [p]}\right)\\
    &=\eta_{\Iw(\pp_1)}^{(1-m)}\cup  ( \iota^{(p)}_{\coh,*}\circ \pi_p^*\left(\theta^{k_1-j}E_{0, 1/N}^{-t, [p]}\right)\\
    &=\left\langle \eta_{\Iw(\pp_1)}^{(1-m)}, \,  (\iota^{(p)}_{\coh,*}\circ \pi_p^*)\, \theta^{k_1-j}E_{0, 1/N}^{-t, [p]}\right\rangle.
   \end{align*}
 So the formula in Proposition \ref{prop:exp} is equivalent to
\begin{multline*}
 \left\langle (\iota^{[j]})^*(\eta_{\fp}), \, \Eis^{t}_{\syn, N}\right\rangle=N^{(k_1 + k_2 - 2j)} (-1)^{k_2 - j + 1}\\
 \frac{\left(1-\tfrac{\beta_1}{p\alpha_1}\right)}
 {\left(1-\frac{p^j}{\alpha_1\alpha_2}\right)\left(1-\frac{p^j}{\alpha_1\beta_2}\right) \left(1-\tfrac{\beta_1\alpha_2}{p^{j+1}}\right) \left(1-\tfrac{\beta_1\beta_2}{p^{j+1}}\right)}
 \binom{k_1}{j}\, k_2!\, \cdot
 \left\langle\eta_{\Iw(\pp_1)}^{(1-m)}, \,  (\iota^{(p)}_{\coh,*}\circ \pi_p^*)\, \theta^{k_1-j}E_{0, 1/N}^{-t, [p]}\right\rangle.
\end{multline*}

Comparing the previous proposition with Definition \ref{def:AsaipL}, we hence deduce the first main theorem of this paper, which relates the Euler system classes for the geometric twists to non-critical values of the $p$-adic Asai $L$-function:

\begin{theorem} \label{thm:regulator}
 Let $0\leq j\leq \min\{k_1,k_2\}$, and let $t=k_1+k_2-2j \ge 0$. If $p^{1 + j} \notin \{ \alpha_1 \alpha_2, \alpha_1 \beta_2\}$, then we have
 \[
 \left\langle \eta_{\dR}, \, \log\left(\loc_p \AF_{\et}^{[\Pi, j]}\right)\right\rangle =
 \frac{(\sqrt{D})^{j + 1} (-1)^{(k_1 - k_2)/2 + 1} N^t \binom{k_1}{j}\, k_2!}{\left(1-\frac{p^j}{\alpha_1\alpha_2}\right) \left(1-\frac{p^j}{\alpha_1\beta_2}\right) \left(1-\tfrac{\beta_1\alpha_2}{p^{j+1}}\right)
  \left(1-\tfrac{\beta_1\beta_2}{p^{j+1}}\right)}
 \, \cdot L_{p, \As}^{\imp}\left(\Pi, 1 + j\right).\]
\end{theorem}

 Note that the $p$-adic $L$-values in the above formula are always outside the critical range. In the next section we shall establish a relation between critical twists of the Euler system and critical values of $L_{p, \As}^{\imp}$, using variation in Hida families.

 \begin{remark}
  In the awkward case $(k_1, k_2, j) = (k, k, k)$ for some $k \ge 0$, the factor $\left(1 - \tfrac{\beta_1 \alpha_2}{p^{j+1}}\middle)\middle(1 - \tfrac{\beta_1\beta_2}{p^{j+1}}\right)$ can indeed vanish. (In fact this \emph{always} occurs if $\Pi$ is a base-change of an elliptic modular form of trivial nebentype at $p$.)

  However, in this case we can recognise the right-hand side of \cref{prop:coheval} as a $p$-adic $L$-value in a somewhat different sense, using the ``improved $p$-adic $L$-function'' $L_{p, \As}^{\sharp}$ defined in \cite{LZ-improved}. This gives the formula
  \[
   \left\langle \eta_{\dR}, \, \log\left(\loc_p \AF_{\et}^{[\Pi, k]}\right)\right\rangle =
   \frac{-(\sqrt{D})^{k + 1} k!}{\left(1-\frac{p^k}{\alpha_1\alpha_2}\right) \left(1-\frac{p^k}{\alpha_1\beta_2}\right)}
   \, \cdot L_{p, \As}^{\sharp}\left(\Pi\right), \]
  where $L_{p, \As}^{\sharp}\left(\Pi\right)$ is related to $L_{p, \As}^{\imp}\left(\Pi, 0\right)$ by the formula
  \[ L_{p, \As}^{\imp}(\Pi, 0) = \left(1 - \tfrac{\beta_1 \alpha_2}{p^{k+1}}\middle)\middle(1 - \tfrac{\beta_1\beta_2}{p^{k+1}}\right) \cdot L_{p, \As}^{\sharp}(\Pi).\qedhere\]
 \end{remark}

%%%%%%%%%%%%%%%%%%%%%%%%%%%%%%%%%%%%%%%%
\section{Hida families}
%%%%%%%%%%%%%%%%%%%%%%%%%%%%%%%%%%%%%%%%

\subsection{Families of eigensystems}

Let us choose an ``initial'' weight $(k_1, k_2; w)$, with $k_i \ge 0$. We shall consider families of eigensystems around a neighbourhood of this weight. Let $E$ be a finite extension of $\Qp$.

\begin{notation}
 Let $\mathbb{T}$ denote the product of the prime-to-$\mathfrak{N}p$ Hecke algebra $\mathbb{T}^{(\mathfrak{N}p)}$ (with $E$-coefficients) and the polynomial ring in formal variables $U_{\pp_1}$, $U_{\pp_2}$ and their inverses.
\end{notation}

\begin{definition}
 For $i=1,2$ let $\cU_i \subset \cW =(\operatorname{Spf} \Lambda)^{\mathrm{rig}}_E$ be an open affinoid containing 0; and let
 \[\kappa_{\cU_i} : \Zp^\times \to \cO(\cU_i)\]
 be the universal weight. A \emph{$p$-adic family of eigensystems} over $\cU = \cU_1 \times \cU_2$ of weight $\left(k_1 + 2\kappa_{\cU_1}, k_2 + 2\kappa_{\cU_2}; w\right)$ and level $\mathfrak{N}$ is a homomorphism $\mathbb{T} \to \cO(\cU_1\times \cU_2)$ with the following property:
 \begin{itemize}
  \item For every $P = (a_1, a_2) \in \cU \cap \ZZ^2$ with $k_i + 2a_i \ge 0\ \forall i$, the composite homomorphism
  \[ \mathbb{T} \to \mathcal{O}(\cU) \to E\]
  given by evaluation at $P$ is the Hecke eigenvalue system associated to a $p$-stabilisation of $\Pi[P]$, for some cuspidal automorphic representation $\Pi[P]$ of $\GL_2 / F$ of level $\mathfrak{N}$ and weight $(k_1 + 2a, k_2+2a_2; w)$ and some embedding of its coefficient field into $E$.
 \end{itemize}
 We denote such a family by $\uPi$. We say that $\uPi$ is \emph{$p$-ordinary} if the Hecke eigenvalues of $U'_{\pp_1}$, $U_{\pp_2}$ and their inverses acting on $\uPi$ are power-bounded.
\end{definition}

By standard results due to Hida, we see that for any given $\Pi$ of weight $(k_1, k_2; w)$ and conductor $\mathfrak{N}$, with $k_1, k_2 \ge 0$, ordinary at $p$, we may find a $p$-ordinary family $\uPi$ over some sufficiently small $\cU_1\times \cU_2$ passing through $\Pi$ (i.e.~with $\Pi[0] = \Pi$), and this is unique up to further shrinking $\cU_1\times \cU_2$.

\subsection{Interpolation sets}

Let $\uPi$ be a family as above. We now define certain special sets of points in $\cU$ and in the product $\cU \times \cW$.

\begin{notation}
 We write $\Sigma_{\mathrm{cl}}$ for the set of $P = (a_1,a_2) \in \cU \cap \ZZ^2$ with $k_i + 2a_i \ge 0\, \forall i$, and we call these points \emph{classical points}.
\end{notation}

We shall also need to consider various loci in the 3-dimensional space $\cU \times \cW$:

\begin{notation}
 Let $(P, Q) \in \cU \times \cW$, with $P = (a_1, a_2) \in \Sigma_{\mathrm{cl}}$.
 \begin{itemize}
  \item We say $(P, Q)$ is \emph{geometric} if $Q = j$ is an integer with $0 \le j \le \min(k_1 + 2a_1, k_2 + 2a_2)$.
  \item We say $(P, Q)$ is \emph{1-dominant critical} if $Q = j + \chi$, with $\chi$ a finite-order character, and $j$ an integer such that $k_2 + 2a_2 < j \le k_1 + 2a_1$.
 \end{itemize}
 We write $\widetilde{\Sigma}_{\mathrm{geom}}$, respectively $\widetilde{\Sigma}_{\mathrm{crit}}$, for the sets of points $(P, Q)$ satisfying these conditions.
\end{notation}

\begin{remark}
 Note that $\widetilde{\Sigma}_{\mathrm{crit}}$ will be the locus where our $p$-adic $L$-function interpolates critical complex $L$-values, and $\widetilde{\Sigma}_{\mathrm{geom}}$ the locus where it interpolates regulators of Euler-system classes. Crucially, both loci are Zariski-dense in $\cU \times \cW$.

 However, while the intersection of $\widetilde{\Sigma}_{\mathrm{crit}}$ with $\cU_1 \times \{0\} \times \cW$ is Zariski-dense in this smaller space, this is \textbf{not} true for $\widetilde{\Sigma}_{\mathrm{geom}}$. This is why we need to introduce the second weight variable in order to prove our reciprocity law, rather than using the two-variable $p$-adic $L$-function (with $k_2$ fixed) defined in \cite{grossiloefflerzerbesLfunct}.
\end{remark}

%%%%%%%%%%%%%%%%%%%%%%%%%%%%%%%%%%%%%%%%
\section{A three-variable $p$-adic Asai $L$-function}
%%%%%%%%%%%%%%%%%%%%%%%%%%%%%%%%%%%%%%%%

We now prove a strengthening of the main result of \cite{grossiloefflerzerbesLfunct} (c.f. Definition \ref{def:AsaipL}) by constructing a 3-variable $p$-adic Asai $L$-function over $\cU_1 \times \cU_2 \times \cW$, for a Hida family over $\cU_1 \times \cU_2$.

\subsection{Geometry of the Iwahori-level Hilbert modular surface}

Let $\YY_{G, \Iw(p)}$ denote the canonical $\Zp$-model of $Y_{G, \Iw(p)}$, and write $Y_{G, \Iw(p),0}$ denote its special fibre. This has a stratification as a disjoint union of 9 smooth strata $Y_{G, \Iw(p),0}^{\diamondsuit, \heartsuit}$, where $\diamondsuit, \heartsuit\in\{m, \alpha, \et\}$. We use the notation
\[ Y_{G, \Iw(p),0}^{ \bullet,m} \coloneqq Y_{G, \Iw(p),0}^{m, m} \cup Y_{G, \Iw(p),0}^{m, \alpha} \cup Y_{G, \Iw(p),0}^{m, \et}\]
which is open in $Y_{G, \Iw(p), 0}$. This stratification extends naturally to $X_{G, \Iw(p), 0}$.

\begin{remark}
 There exists a finite covering of $\YY_{G, \Iw(p)}$ which is a moduli space for Hilbert--Blumenthal abelian surfaces $A$ over $\Zp$-algebras, equipped with level groups $C_1 \subset A[\pp_1]$ and $C_2 \subset A[\pp_2]$ (along with a prime-to-$p$ polarisation and level structure). The strata $Y_{G, \Iw,0}^{\diamondsuit, \heartsuit}$ correspond to the regions where $C_1$ is (\'etale-locally) of type $\diamondsuit$ and $C_2$ of type $\heartsuit$.
\end{remark}

\begin{note}
 \label{note:iotas}
 As in \cref{note:iota} above, there is a natural map $\iota_{\Iw(p)}:Y_{H, \Iw(p)} \to Y_{G, \Iw(p)}$ giving rise to pushforward maps in coherent cohomology.
\end{note}

\subsection{Summary of higher Hida theory}

 Let $\sX_{G, \Iw(p)}$ denote the generic fibre of $\XX_{G, \Iw(p)}$ as classical rigid analytic space (not a dagger space). We now state the classicality isomorphism of higher Hida theory at Iwahori level generalizing \cite[Corollary 2.3.6-Corollary 2.4.1]{grossiloefflerzerbesLfunct}.

 Recall that in \emph{op. cit.} we developed the theory at $\mathfrak{p}_1$-level: we constructed Hecke operators $U_{\pp_1}$ and $U_{\pp_1}'$ and certain integral models of the sheaves $\omega^{(k_1,k_2,w)}(-D_G), \omega^{(k_1,k_2,w)}$. After tensoring with $\mathbb{Q}_p$, Corollary 2.3.6 (respectively Corollary 2.4.1) in \emph{op. cit.} state that the $U_{\pp_1}$-ordinary part (the $U_{\pp_1}'$-ordinary part respectively) of the cohomology of the $\pp_1$-multiplicative locus (of the $\pp_1$-multiplicative locus with compact support towards the $\pp_1$-supersingular locus respectively) is isomorphic to the corresponding ordinary part of the cohomology of $\sX_{G, \Iw(\pp_1)}$ of the sheaves $\omega^{(k_1,k_2,w)}(-D_G), \omega^{(k_1,k_2,w)}$ when $k_1\geq 2$ (when $k_1\leq 0$ respectively). The theory applies \emph{verbatim} for $\pp_2$ and we can define similarly the Hecke operators $U_{\pp_2}$ and $U_{\pp_2}'$. We denote by $e_{\pp_i}$ and $e_{\pp_i}'$ the ordinary projectors for $U_{\pp_i}$ and $U_{\pp_i}'$ respectively.

 For $i=1,2$  let $D^{(i-\et)}$ be the Cartier divisor such that
 \[ X_{G, \Iw(p),0}-D^{(1-\et)} = X_{G, \Iw(p),0}^{\m, \bullet}\quad \text{ and}\quad  X_{G, \Iw(p),0}-D^{(2-\et)} = X_{G, \Iw(p),0}^{\bullet, \m}.\]

 \begin{definition}
  We define the following complex
  \begin{align*}
   M^\bullet_{\m, \m}(k_1,k_2)&= \RG_{w1}(\mathfrak{X}_{G, \Iw(p)}^{\m, \m}, \omega^{(-k_1,k_2+2,w)}(-D_G))\\
   &=\varprojlim_m \colim_{n_2} \varprojlim_{n_1} \RG(\XX_{G, \Iw(p),R/\wp^m}, \omega^{(-k_1,k_2+2,w)}(-D_G - n_1 D ^{(1-\et)}  +n_2 D^{(2-\et)}) ),
  \end{align*}
  and let
  $$H^1_{w1}(\mathfrak{X}_{G, \Iw(p)}^{\m, \m}, \omega^{(-k_1, k_2+2, w)}(-D_G))= H^1(M^\bullet_{\m, \m}(k_1,k_2)).$$
 \end{definition}

 We then obtain the following result.

 \begin{theorem}
  If $k_1,k_2\geq 0$, then we have an isomorphism
  \[
  e_{\pp_1}'e_{\pp_2}H^1_{w1}(\mathfrak{X}_{G, \Iw(p)}^{\m, \m}, \omega^{(-k_1,k_2+2,w)}(-D_G)) \cong
  e_{\pp_1}'e_{\pp_2}H^1(X_{G, \Iw(p)}, \tilde{\omega}^{-k_1,k_2+2}(-D_G)),
  \]
  for some coherent sheaf $\tilde{\omega}^{(-k_1,k_2+2)}(-D_G)$ isomorphic to $\omega^{(-k_1,k_2+2)}(-D_G)$ on the generic fibre of $X_{G, \Iw(p)}$.
 \end{theorem}
 \begin{proof}
  The proof follows as in \cite{grossiloefflerzerbesLfunct}. In particular, one shows that the Hecke operators factor as follows:
  \[
  \begin{tikzcd}
   \RG(X, \omega_{n_1}(n_2t_2D^{(2-\et)}))\arrow{r}{}\arrow{d}{U_{\pp_2}} &\RG(X, \omega_{n_1}((n_2+1)t_2D^{(2-\et)}))\arrow{d}{U_{\pp_2}}\arrow[dashed]{dl}{U_{\pp_2}}\\
   \RG(X, \omega_{n_1}(n_2t_2D^{(2-\et)}))\arrow{r}{} &\RG(X, \omega_{n_1}((n_2+1)t_2D^{(2-\et)})),
  \end{tikzcd}
  \]
  for some $t_2\in \ZZ$ (see \cite[Lemma 2.3.3 and Definition 2.3.4]{grossiloefflerzerbesLfunct}) and where $\omega_{n_1}=\omega^{(-k_1,k_2+2,w)}(-D_G-n_1D^{(1-\et)})$. And
  \[
  \begin{tikzcd}
   \RG(X, \omega_{n_2}(-(n_1+1)t_1D^{(1-\et)}))\arrow{r}{}\arrow{d}{U'_{\pp_1}} &\RG(X, \omega_{n_2}(-n_1t_1D^{(1-\et)}))\arrow{d}{U'_{\pp_1}}\arrow[dashed]{dl}{U'_{\pp_1}}\\
   \RG(X, \omega_{n_2}(-(n_1+1)t_1D^{(1-\et)}))\arrow{r}{} &\RG(X, \omega_{n_2}(-n_1t_1D^{(1-\et)})),
  \end{tikzcd}
  \]
  for some $t_1\in \ZZ$ and where $\omega_{n_2}=\omega^{(-k_1,k_2+2,w)}(-D_G+n_2D^{(2-\et)})$. Using the base change formula as in \emph{op. cit.}, we hence deduce that the maps in the direct and inverse limits defining $M^\bullet_{\m, \m}(k_1,k_2)$ are quasi-isomorphisms, from which the theorem follows (after fixing a certain choice of twist of the sheaf depending on $t_1,t_2$).
  Morevoer note that the proof also yields the fact that the cohomology of the complex $M^\bullet_{\m, \m}(k_1,k_2)$ is independent on the order for which we take the inverse and direct limit for $n_1,n_2$.
 \end{proof}

 Let $\Lambda=\Lambda(\ZZ_p^\times)$, and denote by $\kappa_1, \kappa_2:\ZZ_p^\times\to (\Lambda^\times)^2$ the universal characters associated to the two factors of $(\ZZ_p^\times)^2$.
 We consider the following sheaf
 \begin{equation}
  \tilde{\Omega}^{(\kappa_1, \kappa_2)} := \left(\left(\pi_{\star}\cO_{\mathfrak{IG}} \otimes (\pi_{\mathfrak{IG}\times \mathfrak{IG}^{\vee}})_{\star}\cO_{\mathfrak{IG}\times \mathfrak{IG}^{\vee}}\right) \hat{\otimes} \Lambda^2\right)^{(\ZZ_p^\times)^6},
 \end{equation}
 where $\mathfrak{IG}$ denotes the Igusa tower over the (formal completion of the) multiplicative locus and each copy $(\ZZ_p^{\times})^2$ acts on $\pi_{\star}\cO_{\mathfrak{IG}}$ via the action on $\mathfrak{IG}=\mathfrak{IG}(\pp_1^{\infty})\times \mathfrak{IG}(\pp_2^{\infty})$ as in \cite[Definition 4.2.2]{grossi21}. Moreover the action of
 $((x_{\pp_i})_{i},(z_{\pp_i})_{i},(t_{\pp_i})_{i})\in(\ZZ_p^\times)^6$ is given on $\Lambda^2$ by
 $$\kappa_1(x_{\pp_1}^{-2}t_{\pp_1})\kappa_2(x_{\pp_2}^2t_{\pp_2}^{-1}).$$
 We remark that here we do not vary the \emph{norm variable} $w$ as in \cite{grossi21}. In particular, the specialisation $\tilde{\Omega}^{(\kappa_1, \kappa_2)}[P]$ of $\tilde{\Omega}^{(\kappa_1, \kappa_2)}$ at an integer point $P=(a_1,a_2)$ is isomorphic to $\omega^{(- 2a_1, 2a_2 ; 0)}\big|_{\mathfrak{X}_{G, \Iw}^{\m, \m}}.$

 \begin{remark}
  Note that this sheaf of $\Lambda$-modules is defined over the moduli space, but it can be descended to a sheaf on the (multiplicative loci of) the Shimura variety.
  We have an action of $x\in \cO_{F,+}^{\times}$ on $\tilde{\Omega}^{(\kappa_1, \kappa_2)}$ via the tautological isomorphism (since the construction of $\tilde{\Omega}^{(\kappa_1, \kappa_2)}$ does not use the polarisation)
  \[
  x^*\tilde{\Omega}^{(\kappa_1, \kappa_2)}\simeq \tilde{\Omega}^{(\kappa_1, \kappa_2)}.
  \]
  This action factors through the quotient of $\cO_{F,+}^{\times}$ by $(K\cap\cO_{F}^\times)^2$ (c.f. \cite[4.2.2]{grossi21}).
 \end{remark}

 Finally, for our fixed choice of weight $(-k_1,k_2+2,w)$ as above, we let
 \begin{equation}
  {\Omega}^{(\kappa_1, \kappa_2)} :=\tilde{\Omega}^{(\kappa_1, \kappa_2)}\otimes_{\mathcal{O}_{\mathfrak{X}_{G, \Iw}^{\m, \m}}} \omega^{(-k_1,k_2+2,w)}.
 \end{equation}

 The sheaf $\Omega^{(\kappa_1, \kappa_2)}$ has the following interpolation property:

 \begin{proposition}
  The sheaf $\Omega^{(\kappa_1, \kappa_2)}$ is an invertible sheaf of $\cO(\mathfrak{X}_{G, \Iw}^{\m, \m})\otimes\Lambda^2$-modules, and its specialisation at an integer point $P=(a_1, a_2)$ is given by
  \[ \Omega^{(\kappa_1, \kappa_2)}[P] = \omega^{(-k_1 - 2a_1, k_2 + 2a_2 + 2; w)}\big|_{\mathfrak{X}_{G, \Iw}^{\m, \m}}. \]
 \end{proposition}

 We can now define Hecke operators $U_{\pp_1}'^{(\kappa_1, \kappa_2)}$ and $U_{\pp_2}^{(\kappa_1, \kappa_2)}$ acting on the cohomology of the sheaf $\Omega^{(\kappa_1, \kappa_2)}$ and specialising  at  an integer point $P=(a_1, a_2)$ to $U'_{\pp_1}$ and $U_{\pp_2}$ acting on the cohomology of $\omega^{(-k_1 - 2a_1, k_2 + 2a_2 + 2; w)}$ (under the isomorphism of the above proposition).
 The operator $U_{\pp_2}^{(\kappa_1, \kappa_2)}$ is defined as in \cite[Definition 3.2.1]{grossiloefflerzerbesLfunct}, namely considering the pullback induced by the universal étale $\pp_2$-isogeny (as in (3.4) of \emph{op. cit.}) and tensoring it with the trace map for $p_1$ and multiplying by the normalisation factor $p^{-1-\tfrac{w-k_2-2}{2}}$.
 The operator $U_{\pp_1}'^{(\kappa_1, \kappa_2)}$ is defined similarly but considering the pullback by the dual $\pp_1$-isogeny (the pullback on the Igusa tower of the map considered in \cite[$\S$ 2.4]{grossiloefflerzerbesLfunct}), tensoring it with the trace map for $p_2$ and multiplying by the normalisation factor $p^{-\tfrac{w-k_1}{2}}$. The fact that they specialise to the correct Hecke operators follows as in \cite[Proposition 3.2.2]{grossiloefflerzerbesLfunct}.

 By abuse of notation, we will use $e_{\pp_1}',e_{\pp_2}$ for the ordinary part with respect to $U_{\pp_1}'^{(\kappa_1, \kappa_2)}$ and $U_{\pp_2}^{(\kappa_1, \kappa_2)}$.

 \begin{definition}
  Define
  \[
  M^{\bullet}_{\m, \m}(\kappa_1, \kappa_2) \coloneqq e_{\pp_1}'e_{\pp_2} R\Gamma_{w_1}(\mathfrak{X}_{G, \Iw(p)}^{\m, \m}, \Omega^{(\kappa_1, \kappa_2)}(-D_G)).
  \]
 \end{definition}

 The complex $M^{\bullet}_{\m, \m}(\kappa_1, \kappa_2) $ is a perfect complex of $\Lambda^2$-modules, and it has the following interpolation property (c.f. \cite[Theorem 3.2.3]{grossiloefflerzerbesLfunct}, \cite[Theorem 4.2.13]{grossi21}):

 \begin{proposition}
  For $a_1,a_2\in\ZZ$ with $k_1+2a_1,k_2+2a_2\geq 0$, we have
  \[
  M^{\bullet}_{\m, \m}(\kappa_1, \kappa_2)\otimes^{\mathbb{L}}_{\Lambda^2,a_1,a_2}R\simeq e_{\pp_1}'e_{\pp_2}R\Gamma(X_{G, \Iw(p)}, \tilde{\omega}^{-k_1-2a_1,k_2+2a_2+2}(-D_G)).
  \]
 \end{proposition}

 We have a similar picture over $\mathfrak{X}^{m}_{H, \Iw(p)}$: write $\lambda:\ZZ_p^\times\to \Lambda^\times$ for the universal character, and let  $k$ be any weight. Then there exists a sheaf $\Omega_H=\underline{\omega}_H(\lambda)$ of $\cO(\mathfrak{X}^{(m)}_{H, \Iw(p)})\otimes\Lambda$-modules whose specialisation at an integer point $a$ is given by
 \[ \Omega_H(a)=\omega^{k+2a}.\]

 Then choosing $\lambda=\kappa_2-\kappa_1$, pullback along $\iota_{\Iw}$ induces a morphism
 \begin{equation}
  \label{eq:pullback}
  \iota^*_{\Iw}:\,R\Gamma_{w_1}\left(\mathfrak{X}^{(m,m)}_{G, \Iw(p)}, \Omega_G(\kappa_1, \kappa_2)\,(-D_G)\right)\to R\Gamma_{c}\left(\mathfrak{X}^{m}_{H, \Iw}, \Omega_H(\kappa_2-\kappa_1)\otimes\Omega_H^1\,(-D_H)\right)\otimes \|\cdot\|^{-w}.
 \end{equation}

 %%%%%%%%%%%%%%%%%%%%%%%%%%%%%%%%%%%%%%%%

 \subsection{P-adic families}

 \begin{proposition}\label{thm:families}
  Let $\uPi$ be a $p$-ordinary family over $\cU_1\times \cU_2$, and define
  \[ M^\bullet_{\m, \m}(\kappa_{\cU_1}, \kappa_{\cU_2}) \coloneqq M^\bullet_{\m, \m}(\kappa_1, \kappa_2)  \otimes_{\Lambda^2} \cO(\cU_1\times \cU_2). \]
  Then $H^1$ of this complex contains a free rank 1 $\cO(\cU_1\times \cU_2)$-direct summand $H^1_{\m, \m}(\kappa_{\cU_1}, \kappa_{\cU_2})[\uPi]$ whose specialisation at every $(a_1,a_2) \in (\cU_1\times \cU_2) \cap \ZZ^2$ with $k_i + 2a_i\geq 0$ is the eigenspace in
  \[ e_{\pp_1}'e_{\pp_2} H^1\left(X_{G, \Iw(p)}, \tilde{\omega}^{(-k_1-2a_1, k_2+2a_2+2)}(-D_G)\right) \otimes_{\Zp} E \]
  on which $\mathbb{T}$ acts via the ordinary $p$-stabilisation of $\Pi[a]$.
 \end{proposition}

 \begin{definition}\label{def:ueta}
  Write
  \[ \cS_{w_1}(\uPi) = H^1_{\m, \m}(\kappa_{\cU_1}, \kappa_{\cU_2})[\uPi], \]
  and let $\underline\eta$ be a basis of $\cS_{w_1}(\uPi)$ as a $\cO(\cU_1\times \cU_2)$-module.
 \end{definition}

 \begin{note}
  The specialisation of $\underline\eta$ at $P=(a_1,a_2)$ is a basis of the space $\cS_{w_1, \Iw(p)}(\Pi[P]) \otimes_{L_P} E$ (where $L_P \subset \QQbar$ is the coefficient field of $\Pi[P]$), which can be identified with $\cS_{w_1}(\Pi[P]) \otimes_{L_P} E$ via degeneracy maps as before. However, it may not descend to a basis of this space over $L_P$. If we choose an algebraic basis $\eta_{\Pi[P]}$ defined over $L_P$ (see Remark \ref{rem:occult}), then this defines (as in \S 7.4 in \cite{grossiloefflerzerbesLfunct}) a pair of periods
  \[
  \Omega_p(\Pi[P])\in E^\times\qquad \text{and}\qquad \Omega_\infty(\Pi[P])\in\CC^\times.
  \qedhere
  \]
 \end{note}

 %%%%%%%%%%%%%%%%%%%%%%%%%%%%%%%%%%%%%%%%

 \subsection{The $p$-adic $L$-function}

 Consider the dual interpolation sheaf $\Omega_H^\vee(\lambda)$, whose specialisation at an integer $a$ specializes to $\omega_H^{-k-2a}$ with $k=k_2-k_1$. We then have a cup product pairing
 \[ \langle\quad, \quad\rangle:\, H^0(\mathfrak{X}^{m}_{H, \Iw(p)}, \Omega^\vee_H(\kappa_{\cU_1}-\kappa_{\cU_2}))\times H_{c}^1(\mathfrak{X}^{m}_{H, \Iw(p)}, \Omega_H(\kappa_{\cU_2}-\kappa_{\cU_1})\otimes\Omega_H^1\,(-D_H))\to \cO(\cU).\]
 Write $\bfj:\ZZ_p^\times\to \cO(\cW)^\times$ for the universal character.
 We consider Katz's Eisenstein measure
 \[ \cE\left(\frac{k_1-k_2}{2}+\kappa_{\cU_1}-\kappa_{\cU_2}-\bfj,  \frac{k_1-k_2}{2}+\kappa_{\cU_1}-\kappa_{\cU_2}+\bfj\right)\in  H^0(\mathfrak{X}^{m}_{H, \Iw(p)}, \underline\omega^\vee_H(\kappa_{\cU_2}-\kappa_{\cU_1}))\hat\otimes  \cO(\cW);\]
 to lighten the notation, we denote it by $ \cE_{\kappa_{\cU_1}-\kappa_{\cU_2}}(\bfj)$.

 \begin{definition}
  We define the $3$-parameter  $p$-adic Asai $L$-function by the pairing
  \[ L_{p, \As}^{\imp}(\uPi, \uet)(\kappa_1, \kappa_2, \bfj)=\frac{1}{p}(\sqrt{D})^{-1-(\kappa_1+\kappa_2)/2-\bfj}(-1)^{\bfj}\cdot \left\langle \cE_{\kappa_{\cU_1}-\kappa_{\cU_2}}(\bfj), \, \iota^*_{\Iw}(\uet)\right\rangle\in \cO(\cU\times\cW), \]
  where $\cU=\cU_1\times \cU_2$.
 \end{definition}

 \subsection{Interpolation of critical values}

 This function $L_{p, \As}^{\imp}(\uPi, \uet)$ satisfies the following interpolation property at $\pp_1$-dominant critical points:

 \begin{theorem}\label{thm:padicL}
  Let $(P, Q)\in \widetilde{\Sigma}_{\crit}$, with $Q = j \in \ZZ$. Then we have
  \[
   \frac{L_{p, \As}^{\imp}(\uPi, \uet)(P, j+1)}{\Omega_p(\Pi[P])} =
  \cE_p\left(\Pi[P], 1+j\right)\cdot \frac{\Gamma(j-1) \Gamma\left(j-k_2-2a_2\right)}{2^{k_1+2a_1-2}i^{1-k_2-2a_2}(-2\pi i)^{2j-k_2-2a_2-1}}
  \cdot \frac{L^{\imp}_{\As}(\Pi[P],j+1)}{\Omega_\infty(\Pi[P])},
  \]
  where
  \[
  \cE_p(\Pi[P], 1 + j)=
  \left(1 - \frac{p^j}{\alpha_{P, 1} \alpha_{P, 2}}\right)
  \left(1 - \frac{p^j}{\alpha_{P, 1} \beta_{P, 2}}\right)
  \left(1 - \frac{\beta_{P, 1} \alpha_{P,2}}{p^{j+1}}\right)
  \left(1 - \frac{\beta_{P, 1} \beta_{P,2}}{p^{j+1}}\right)
  \]
  (the same Euler factor as in \cref{thm:regulator}).
 \end{theorem}

 \begin{proof}
  Since the cup-product defining $L_{p, \As}^{\imp}(\uPi, \uet)$ commutes with specialisation, and the specialisation of the Eisenstein series at this point is a classical modular form, the value at $(P, j+1)$ can be written as a coherent-cohomology cup-product and hence as a zeta-integral. This is a very similar computation as in the proof of \cite[Theorem C]{grossiloefflerzerbesLfunct}; the only difference is that we use a slightly different test vector in the local zeta-integral computation at $p$, since our local Whittaker function at $\pp_2$ is now the normalised $U_{\pp_2}$-eigenfunction, rather than the spherical Whittaker function. However, this makes no difference to the zeta-integral computation (which only sees the value of the Whittaker function at 1).
 \end{proof}

 Comparing with the main theorem of \cite{grossiloefflerzerbesLfunct}, which gives a 2-variable $p$-adic $L$-function varying analytically over the two-dimensional ``slice'' $\cU_1 \times \{a_2\} \times \cW$, for any given $a_2$ with $k_2 + 2a_2 \ge 0$, and noting that the intersection of $\widetilde{\Sigma}_{\crit}$ with this set is Zariski-dense, we conclude that the restriction of $L_{p, \As}^{\imp}(\uPi, \uet)$ to this slice must coincide with the 2-variable $p$-adic $L$-function of \emph{op.cit.}. In particular, its restriction to $\{a_1\} \times \{a_2\} \times \cW$, for any $P = (a_1, a_2) \in \Sigma_{\cl}$ with $k_1 + 2a_1 > k_2 + 2a_2 \ge 0$,  coincides with the 1-parameter $p$-adic $L$-function $L_{p, \As}^{\imp}(\Pi[P])$ from Section \ref{ssect:1parampL} for the specialisation $\Pi[P]$.

 %%%%%%%%%%%%%%%%%%%%%%%%%%%%%%%%%%%%%%%%
 \section{An explicit reciprocity law}
 %%%%%%%%%%%%%%%%%%%%%%%%%%%%%%%%%%%%%%%%

 \subsection{\'Etale cohomology and Euler systems in families}\label{sect:HidaAF}

 We recall some results due to Sheth \cite{sheth25} on families of Galois representations associated to $\uPi$, and families of Galois cohomology classes for these.

 \begin{theorem}[{Sheth, \cite{sheth24}}]
  For small enough $\cU$, there exists a free rank 4 module $\VV=\VV(\uPi)$ over $\cO(\cU_1 \times \cU_2)$, whose specialisation at each classical point $P \in \Sigma_{\cl}$ is canonically isomorphic to the Asai representation $V_{p, \As}(\Pi[P])$.
 \end{theorem}

 Now let $c > 1$ be an integer coprime to $6p N$, and $S$ the set of primes dividing $p D N$. We write $\VV^*(-\bfj)$ for the ``universal twist'' $\VV^* \mathop{\hat\otimes} \cO_E[[\Zp^\times]](-\bfj)$, where $\bfj$ denotes that Galois acts through the composite of the cyclotomic character and the \emph{inverse} of the canonical character from $\Zp^\times$ into its own Iwasawa algebra.

 \begin{note}
  Note that for any open $\cO_E$-lattice $\mathbb{T}^* \subseteq \VV^*$ we have
  \[ H^1(\ZZ[1/S], \VV^*(-\bfj)) = E \otimes_{\cO_E} \varprojlim_r H^1(\QQ(\mu_{p^r}), \mathbb{T}^*),
  \]
  so the cohomology of $\VV^*(-\bfj)$ is the classical Iwasawa cohomology; this is the reason for the choice of sign.
 \end{note}

 \begin{theorem}[{Sheth, \cite{sheth25}}]
  \label{thm:sheth2}
  There exists a cohomology class
  \[ {}_c z_{p^\infty}(\uPi) \in H^1(\ZZ[1/S], \VV^*(-\bfj))\]
  which has the following interpolation property at each point $(P, j) \in \widetilde{\Sigma}_{\mathrm{geom}}$:
  \[
  {}_c z_{p^\infty}(\uPi)(P,Q)=
  \frac{\left(1-\frac{p^j}{\alpha_{P, 1}\alpha_{P,2}}\right)
   \left(1-\frac{\alpha_{P, 1}\beta_{P,2}}{p^{1+j}}\right)
   \left(1-\frac{\beta_{P, 1}\alpha_{P,2}}{p^{1+j}}\right)
   \left(1-\frac{\beta_{P, 1}\beta_{P,2}}{p^{1+j}}\right)}
   {(-1)^jj!\tbinom{k_1+2a_1}{j}\tbinom{k_2+2a_2}{j}}
  \times (c^2 - c^{-t} \chi_{\Pi}^{-1}(c)) \AF_{\et}^{[\Pi[P], j]}.
  \]
 \end{theorem}

 Here $t = (k_1 + 2a_1) + (k_2 + 2a_2) - 2j$ if $P = (a_1, a_2)$. More generally, Sheth's result gives classes ${}_c z_{m p^\infty}(\uPi)$ for any $m$ coprime to $p$, interpolating the Asai--Flach classes over $\QQ(\mu_m)$.

 \subsection{Localisation at $p$}

 By comparison with the \emph{standard} representation in families, we see that (after possibly shrinking $\cU$) there are free rank 2 subrepresentations $\cF_i^+ \VV$ whose specialisations at all points $P \in \Sigma_{\cl}$ of \emph{non-parallel weight} interpolate the subrepresentations $\cF_i^+ V_P$ described above.\footnote{More generally, this applies to any specialisation with $\alpha_{P, 1} \beta_{P, 2} \ne \beta_{P, 1} \alpha_{P, 2}$, which is automatic in non-parallel weights. If these two quantities are equal, then the subrepresentations $\cF_1^+V_P$ and $\cF_2^+V_P$ are isomorphic to each other as $G_{\Qp}$-representations, and cannot be distinguished without using the partial Frobenii; so it is not immediately clear which of these subrepresentations is the one given by specialising the (uniquely defined) ``family'' subrepresentations $\cF_i^+ \VV$.}

 \begin{proposition}
  The localisation of ${}_c z_{p^\infty}(\uPi)$ at $p$ takes values in the rank 3 subrepresentation $\cF_1^+ \VV^* + \cF_2^+ \VV^*$.
 \end{proposition}

 \begin{proof} The analogous result for the specialisations at classical points is Corollary 9.2.3 of \cite{leiloefflerzerbes18}, and the result for the family $\uPi$ follows from this by interpolation. \end{proof}

 We now consider the image of ${}_c z_{p^\infty}(\uPi)$ in the rank 1 quotient $\MM_{\pp_1} = (\cF_1^+ \VV^* + \cF_2^+ \VV^*) / \cF_1^+  \VV^*$. Note that $\MM_{\pp_1}(-k_2-2\kappa_2-1)$ is unramified, so we can define a Perrin-Riou ``big logarithm'' for $\MM_{\pp_1}$ by pulling back the big logarithm for its unramified twist (compare \cite[Theorem 8.2.8]{KLZ17}), giving a map of $\cO(\cU \times \cW)$-modules
 \[
 \cLPR:\, H^1(\QQ_p, \MM_{\pp_1}(-\bfj))\to \Dcris(\MM_{\pp_1}(-\bfj)).
 \]
 Comparing the interpolating property of Perrin-Riou's map with the preceding theorem (as in the Rankin--Selberg case, see \cite[proof of Theorem 10.2.2]{KLZ17}), we obtain:

 \begin{proposition}
  For any $(P, j) \in \widetilde{\Sigma}_{\mathrm{geom}}$ such that $\alpha_{P, 1} \beta_{P, 2} \ne p^{1 + j}$, we have
  \[
   \cLPR({}_c z_{p^\infty} (\uPi))\big|_{(P, j)}
   =
   \frac{(-1)^{k_2-j}}{(k_2+2a_2-j)!}\cdot \frac{\left( 1-\frac{p^j}{\alpha_{P, 1}\beta_{P,2}}\right)}{\left(1-\frac{\alpha_{P, 1}\beta_{P,2}}{p^{j+1}}\right)}\cdot \log\left({}_c z_{p^\infty}(\uPi)\big|_{(P, j)}\right).\]
 \end{proposition}

 We can now reformulate Theorem \ref{thm:regulator} in terms of the Eichler-Shimura isomorphism $\ES_{\Pi[P], w_1}$ (see \eqref{eq:ES}):

 \begin{theorem}\label{thm:relation}
  For each $(P,Q)\in \widetilde{\Sigma}_{\mathrm{geom}}$, say $P=(a_1,a_2)$ and $Q=j$, such that $\alpha_{1, P} \beta_{2, P} \ne \beta_{1, P} \alpha_{2, P}$ and $\alpha_{P, 1} \beta_{P, 2} \ne p^{1 + j}$, we have
  \begin{multline*}
   \langle \cLPR( {}_c z_{p^\infty}(\Pi[P],j)), \, \ES_{\Pi[P], w_1}(\eta_P)\rangle = \\
   (c^2-c^{-t}\chi_{\Pi}(c)^{-1}) \cdot \sqrt{D}^{j+1}(-1)^{(k_1-k_2)/2+j+1}N^t \cdot
   L_{p, \As}^{\imp}(\uPi, \uet)(P,j+1).
  \end{multline*}
  Here, $t=k_1+k_2+2(a_1+a_2)-2j$, and $\chi$ is the restriction to $\QQ$ of the nebentype character of $\uPi$.
 \end{theorem}

 \subsection{The base-change case} We briefly summarize what can happen in an important special case when the assumption $\alpha_{1, P} \beta_{2, P} \ne \beta_{1, P} \alpha_{2, P}$ is \emph{not} satisfied. We consider a family $\uPi$ whose specialisation at $P = (0, 0)$ is the base-change of a classical modular form $\pi$ of weight $k + 2$ and trivial character, which is good ordinary at $p$. In this case there is an additional involution $\rho$ acting on $V_{p, \As}(\Pi)$ arising from the action of $\Gal(F / \QQ)$, with the $\rho = +1$ eigenspace one-dimensional and the $\rho = -1$ eigenspace three-dimensional. Moreover, $V_{p, \As}(\Pi)$ also carries a canonical symmetric bilinear form arising from Poincar\'e duality.

 Since the specialisation at $P$ of the ``family'' subrepresentation $\cF_1^+ \VV$ is a 2-dimensional subrepresentation which is isotropic, and contains the unique unramified subrepresentation, it must be either $\cF_1^+ V$ or $\cF_2^+ V$. So either $\ES_{\Pi, w_1}(\eta)$, or its image under $\rho$, gives a trivialisation of $\Dcris(\MM_{\pp_1})|_{P = 0}$. As the class ${}_c z_{p^\infty}(\Pi[P],j)$ lies in the $\rho = (-1)^j$ eigenspace, the regulator formula above continues to hold up to a possible sign $(-1)^j$, for all $j < k$.

 (The case $j = k$ is still more awkward, because then we automatically have $\alpha_{1} \beta_{2} = p^{1 + j}$. In this case, the right-hand side of \cref{thm:relation} is zero, because of the results on improved $p$-adic $L$-functions in \cite{LZ-improved}. However, it is not \emph{a priori} clear whether this holds for the left-hand side as well; the image of ${}_c z_{p^\infty}(\uPi)$ in $H^1(\QQ, V_{p, \As}(\Pi)^*(-k))$ is zero, because of the form of the Euler factors in \cref{thm:sheth2}; but this does not imply that the Perrin-Riou regulator vanishes at $(P, Q)$, since the regulator depends on the classes over the whole $p$-cyclotomic tower.)

%%%%%%%%%%%%%%%%%%%%%%%%%%%%%%%%%%%%%%%%

 \subsection{Explicit reciprocity laws and meromorphic Eichler--Shimura}

 \begin{definition}
  Denote by $\Sigma_{\mathrm{cl}}^{\sharp} \subset \Sigma_{\mathrm{cl}}$ the set of $P$ with the following property: there exists a $j$ such that $(P, j) \in \widetilde{\Sigma}_{\mathrm{geom}}$ and $L_p^{\imp}(\Pi[P], \eta_P; j + 1) \ne 0$. (We call these \emph{good classical points}.)
 \end{definition}

 \begin{lemma}\label{lem:dense}
  The set $\Sigma_{\mathrm{cl}}^{\sharp}$ is Zariski-dense in $\cU$.
 \end{lemma}

 \begin{proof}
  Consider a function $f$ vanishing at all these points. Then $f \cdot L^{\imp}_p(\uPi, \uet)$ must vanish on $\widetilde{\Sigma}_{\mathrm{geom}}$. As $\widetilde{\Sigma}_{\mathrm{geom}}$ is Zariski-dense, it is identically 0 on $\cU \times \cW$. But $L_p^{\imp}(\uPi, \uet)$ is not a zero-divisor, since every component of $\cU \times \cW$ contains a point where $L_p(\uPi, \uet)$ interpolates a non-vanishing complex $L$-value. Thus $f = 0$.
 \end{proof}

 Write $\mathcal{Q}(\cU)$ for the fraction field of $\cO(\cU)$; and let us write ${}_c L_{p, \As}(\uPi)$ for the $p$-adic $L$-function multiplied by the $c$-factor $(c^2 - \dots)$ (which clearly interpolates to an element of $\cO(\cU \times \cW)$.

 \begin{theorem}\label{thm:ERL}\
  \begin{enumerate}
   \item There exists an isomorphism of $\mathcal{Q}(\cU)$-modules
   \[\ES_{w_1, \uPi}:\, \cS_{w_1}(\uPi)\otimes_{\cO(\cU)}\mathcal{Q}(\cU)\cong  \Dcris(\MM_{\pp_1})\otimes_{\cO(\cU)}\mathcal{Q}(\cU), \]
   depending only on $\uPi$, which satisfies the following property: for all $P \in \Sigma_{\mathrm{cl}}^{\sharp}$, the specialisation of the morphism $\ES_{w_1, \uPi}$ at $P$ is well-defined and coincides with the Eichler--Shimura isomorphism $\ES_{w_1, \Pi[P]}$.
   \item Extending $\ES_{w_1, \uPi}$ to an isomorphism of $\cO(U\times\cW)$-modules, we have an explicit reciprocity law
   \[
    \left\langle \cLPR( {}_c z_{p^\infty}(\uPi)), \,
    \ES_{w_1, \uPi}(\underline\eta)\right\rangle =
    {}_c L_{p, \As}^{\imp}(\uPi, \uet)(1 + \bfj).
   \]
  \end{enumerate}
 \end{theorem}

 \begin{proof}
  Let $A : \cS_{w_1}(\uPi) \to \MM_{\pp_1}$ be some isomorphism. Since $L_{p, \As}^{\imp}(\uPi, \underline\eta)$ is not identically zero, we can consider the ratio
  \[
   R=
   \frac{1}{{}_c L_{p, \As}^{\imp}(\uPi, \uet)(1 + \bfj)}\cdot
    \left\langle \cLPR( z_{p^\infty}(\uPi)), \,
    A(\underline\eta)\right\rangle \in \cO(\cU \times \cW).
  \]
  Let $(P,Q)$ be an geometric point such that $\alpha_{P, 1} \beta_{P, 2} \ne \beta_{P, 1} \alpha_{P, 2}$ and $L_{p, \As}^{\imp}(\uPi, \uet)(P,Q)\neq 0$. Then it follows from Theorem \ref{thm:relation} that $R$ is regular at $(P,Q)$, and its value is equal to $A_P/\ES_{w_1, \Pi[P]}$. In particular, this ratio depends only on $P$, and is independent of $Q$. From the Zariski-density of the set of such $(P, Q)$, it follows that the meromorphic function $R$ on $\cU \times \cW$ is actually independent of the $\cW$ variable.

  Hence $R \in \mathcal{Q}(\cU)^\times$, regular at all points $P \in \Sigma_{\cl}^\sharp$, and coinciding at each such point with $A_P / \ES_{w_1, \Pi[P]}$. It follows that if we define $\ES_{w_1, \uPi} = R^{-1} A$, then $\ES_{w_1, \uPi}$ is regular at all points in $\Sigma_{\cl}^\sharp$ and coincides at such points with $\ES_{w_1, \Pi[P]}$. By Lemma \ref{lem:dense}, this interpolating property uniquely determines $\ES_{w_1, \uPi}$, and the reciprocity law holds by construction.
 \end{proof}

 It follows, in particular, that if $(0, 0) \in \Sigma^\sharp_{\mathrm{cl}}$ (i.e.~$L_{p, \As}^{\imp}(\Pi, \eta)$ is non-zero for some $j$ in the geometric range), and $\alpha_{1} \beta_{2} \ne \beta_{1} \alpha_{2}$, then we have an equality in $\Lambda(\Gamma)$
 \[ \left\langle \cLPR( {}_c z_{p^\infty}(\Pi)), \, \ES_{w_1, \Pi}(\eta) \right\rangle={}_c L_{p, \As}^{\imp}(\Pi, \eta) \]
 for any basis $\eta$ of $\cS_{w_1}(\Pi)$. However, if $(0, 0) \notin \Sigma^\sharp_{\mathrm{cl}}$ then we cannot rule out the possibility that $\cLPR( {}_c z_{p^\infty}(\Pi))$ is identically 0, while ${}_c L_{p, \As}^{\imp}(\Pi, \eta)$ is generically non-zero but vanishes at the finitely many characters where the regulator formula applies.

%%%%%%%%%%%%%%%%%%%%%%%%%%%%%%%%%%%%%%%%
\section{Leading-term arguments and applications}
%%%%%%%%%%%%%%%%%%%%%%%%%%%%%%%%%%%%%%%%

\subsection{Assumptions}

Let $\Pi$ be an automorphic representation of $G(\AA)$ as in Definition \ref{def:Pi}. We impose the following assumptions on $\Pi$:

\begin{assumption}\label{assPi}\
 \begin{enumerate}
  \item $\Pi$ is ordinary at $p$;
        \item $\Pi$ is not of parallel weight;
  \item the local Euler factor $P_p(\Pi, X)$ doe not vanish at $p^{-j}$ for any $j\in\ZZ$;
  \item $\Pi$ satisfies the ``big image'' condition (BI) of \cite[\S 9.4]{leiloefflerzerbes18}, as do all its twists $\Pi \otimes \psi$ for Dirichlet characters $\psi$ of $p$-power order and conductor coprime to some finite set of primes $S$.
 \end{enumerate}
\end{assumption}

As noted in \emph{op.cit.}, if $\Pi$ is non-CM, and there is at least one prime which ramifies in $F$ and is coprime to $\fN$, then $\Pi$ satisfies condition (BI) for a positive-density set of primes $\mathfrak{P}$ of $L$ (and this can be upgraded to ``all but finitely many primes'' if $L = \QQ$). The argument easily generalises to show that the twists $\Pi \otimes \psi$ have big image as well, after discarding a finite set of possible primes from the conductor of $\psi$ (the primes which ramify in $F' / \QQ$, where $F'$ is a finite abelian extension of $F$ determined by the inner twists of $\Pi$, see the proof of Theorem 9.4.6 of \emph{op.cit.}).

\subsection{Leading terms}

Let $\cU' = \cU_1 \cap \cU_2$, embedded diagonally in $\cU_1 \times \cU_2$. By restricting $\uPi$, we obtain a one-parameter family of Hilbert modular forms with weight $(k_1 + 2a, k_2 + 2a; w)$ for varying $a \in \cU'$. Note that the restriction of $L_{p, \As}(\uPi, \uet)$ to $\cU'$ is well-defined, and non-vanishing at 0.

Let $T$ be a uniformiser at $0 \in \cU'$. From the meromorphic reciprocity law and the non-vanishing of the $p$-adic $L$-function modulo $T$, we deduce that the $T$-adic valuation of $\cLPR( z_{p^\infty}(\uPi, \bfj))|_{\cU' \times \cW}$ is equal to the order of the pole of $\ES_{w_1, \Pi}(\uet)$. Let $c \in \ZZ_{\ge 0}$ be this common value.

Repeating the construction with the Eisenstein series replaced by a Dirichlet-character twist (but the same $\uPi$ and $\uet$), we conclude that in fact $\cLPR( z_{mp^\infty}(\uPi))|_{\cU' \times \cW}$ is divisible by $T^c$, for all $m$. So not only the $m = 1$ class of the Euler system, but in fact all of the classes, have vanishing projection to the 1-dimensional local subquotient at $p$ used to define the regulator map. So specialising at $T = 0$ gives an Euler system whose local condition is ``too strong'': it has core rank 0 in the sense of \cite{mazurrubin04}. Arguing as in \cite{LZ20-yoshida}, we conclude that in fact the specialisation at $T = 0$ must vanish identically, for all $m$; and we can therefore divide out factors of $T$ repeatedly, and compensate by removing powers of $T$ from the denominator of $\ES_{w_1, \uPi}$, until we have $c = 0$.

\begin{corollary}
 There exists an Euler system $(\hat{\mathbf{z}}_m)_{m \ge 1}$ for $V_{p, \As}(\Pi)^*$ such that the image of $\hat{\mathbf{z}}_1$ under the Perrin-Riou regulator is a scalar multiple of ${}_c L_{p, \As}(\Pi, \eta)(1 + \bfj)$.
\end{corollary}

\begin{proof}
 The above construction gives an Euler system mapping to a multiple of the imprimitive $L$-series ${}_c L_{p, \As}^{\imp}(\Pi, \eta)(1 + \bfj)$. However, both the Euler system and the $p$-adic $L$-function depend on a choice of test data at the bad primes for $\Pi$, and both have the same equivariance property when the test data changes. Since the space of linear forms with this equivariance property is 1-dimensional, we obtain the reciprocity law for any choice of test data, including the (inexplicit) test data which give the optimal $p$-adic $L$-function.
\end{proof}

\subsection{Arithmetic applications}

Let $T\subset V_{p, \As}(\Pi)$ be a $G_{\QQ}$-invariant lattice, and define $A=T\otimes \Qp/\Zp$. Let $\cF_1^+A$ be the corank 2 $R$-submodule of $A$ dual to the kernel $\cF_1^+ V_{p, \As}(\Pi)^*\to M_{\pp_1}$.

We can now harvest some applications to the Iwasawa Main Conjecture, following \cite[\S 9]{leiloefflerzerbes18}.
\begin{definition}
 Define the Selmer group
 \[
 \mathrm{Sel}^{(\pp_1)}(\QQ(\mu_{p^\infty}),A)=\left\{\begin{array}{l}
  x\in H^1(\QQ(\mu_{p^\infty}),A):\, \loc_\ell(x)=0\, \, \, \forall \, \ell\neq p\\ \loc_p(x)\in\mathrm{image}\, H^1(\Qp(\mu_{p^\infty}), \cF_1^+ A)\end{array}\right\},
 \]
 and write $X^{(\pp_1)}(\QQ(\mu_{p^\infty}),A)$ for its Pontryagin dual.
\end{definition}

Let $\Gamma=\Gal(\QQ(\mu_{p^\infty}/\QQ)$, and let $\Lambda=R\otimes_{\ZZ_p}\ZZ_p[[\Gamma]]$,  where $R$ is the ring of integers of $L_{\mathfrak{P}}$.

\begin{theorem}\label{thm:IMC}
 Assume that Assumptions \ref{assPi} are satisfied. There exists an integer $r$ such that
 \[ \mathrm{char}_{\Lambda}\left(X^{(\pp_1)}(\QQ(\mu_{p^\infty}),A)\middle)\, \right|\, p^r \cdot L_{p, \As}(\Pi, 1 + \bfj).\]
\end{theorem}

\begin{proof}
 It is immediate from Theorem 9.5.3 in \emph{op.cit.} and Theorem \ref{thm:ERL} that the characteristic ideal of the Selmer group divides $p^r {}_cL_{p, \As}(\Pi, \eta)(1 + \bfj)$, for some $r$. It remains to get rid of the factor $(c^2 - c^{2\bfj - k_1 - k_2}\chi(c)^{-1})$ relating the $p$-adic $L$-functions with and without the $c$-subscript.

 If the Dirichlet character $\chi$ is non-trivial, then we can choose $c$ such that the $c$-factor is a unit in the Iwasawa algebra. If $\chi = 1$, then we use the compatibility of the Iwasawa main conjecture with functional equations. The functional equation interchanges $\bfj$ with $k_1 + k_2 + 1 - \bfj$, so the $c$-factor $c^2 - c^{2\bfj - k_1 - k_2}$ is coprime to its image under this involution.
\end{proof}

\begin{remark}
 The arbitrary power of $p$ is inevitable since $L_{p, \As}(\Pi)$ is itself only well-defined up to an arbitrary scalar factor. It should be possible to nail down a canonical normalisation up to $p$-adic units using integrality of equivariant $p$-adic $L$-functions as in \cite{LZ20-yoshida}, but we have not attempted this here for reasons of space.
\end{remark}

By descent from $\QQ(\mu_{p^\infty})$ to $\QQ$, we also obtain cases of the Bloch--Kato conjecture for critical values of the $L$-series.

\begin{theorem}\label{thm:BK}
 Assume that the Assumptions \ref{assPi} are satisfied. Let $j$ be a critical value. Then the Bloch--Kato Selmer group
 \[ H^1_\f(\QQ,V_{p, \As}(\Pi)^*(-j))\]
 is zero.
\end{theorem}

\begin{proof}
 This follows from the previous theorem via the same arguments as in \S 11.7 of \cite{KLZ17} in the Rankin--Selberg case, using the fact that the critical $L$-values are automatically non-zero (as the central value is not an integer).
\end{proof}
\newpage

\section{Appendix: variants of $\eta$}
\label{sect:variantsofeta}

To prove Theorem \ref{thm:regulator} (Theorem \ref{intro:regformula}) we make use of different variants of $\eta$, which we summarize below, together with their relations and a sketch of their roles in the proof, for the reader's convenience.

\begin{center}
    \tabulinesep=1mm
    $\begin{tabu}{c|c|c}
    \text{Notation} & \text{Cohomology group} & \text{Definition} \\
    \tabucline \\
    \eta & H^1\left(X_{G, \sph} / L, \omega^{(-k_1,k_2+2; w)}(-D_G)\right) & \text{Def. \ref{def:etabasic}}\\
    \eta_{\Iw(\pp_1)}&H^1\left(\cX_{G, \Iw(\pp_1)}, \omega^{(-k_1, k_2 + 2)}(-D_G)\right)& \text{Proposition \ref{prop:compnus}}\\
    %\breve\eta_{\Pi, \alpha}&H^1(\cX_{G, \Iw(\pp_1)}, \omega_G^{(-k_1, k_2 + 2)}(-D_G))& \text{\S 3.4 and \cite[\S 7.3]{grossiloefflerzerbesLfunct}}\\
    \eta^{(1-\ord)} &H^1_c\left(\cX_{G, \sph}^{(1-\ord)}, \omega^{(-k_1, k_2 + 2)}(-D_G)\right)& \text{Theorem \ref{thm:nu1ord}}\\
    \eta^{(1-m)}_{\Iw(\pp_1)} &H^1_c\left(\cX_{G, \Iw(\pp_1)}^{(1-m)}, \omega^{(-k_1, k_2 + 2)}(-D_G)\right)& \text{Proposition \ref{prop:compnus}}\\
    \eta^{\ord} &H^1_{c1}\Big(\cX_G^{\ord}, \omega^{(-k_1, k_2 + 2)}(-D_G)\Big)& \text{Notation \ref{not:nuord}}\\
    \dot\eta_j^{\ord}&H^1_{c1}(\cX^{\ord}_G, \Fil^j\cV\otimes \Omega^1\langle -D_G\rangle)& \text{Def. \ref{def:nudot}}\\
    \eta_{\dR}&D_p^{\As}(\Pi)&\text{Proposition \ref{prop:eta}}\\
    \eta_{\dR}^{(1-\ord)} &H^2_{\rig, c}(\cX^{(1-\ord)}_{G, 0}, \cV_G\langle -D_G\rangle)& \text{Notation \ref{not:classfromdR}}\\
    \eta_{\fp}&H^2_{\fp, c}\left(\YY_G, \cV_G; 1+j, P_{1+j}\right)& \text{Proposition \ref{prop:defnufp}}\\
    \tilde\eta^{(1-\ord)}_{\fp}&\wH^2_{\fp, c}(\cX_G^{(1-\ord)}, \cV\langle-D_G\rangle; 1+j, P_{1+j}) & \text{Proposition \ref{prop:lasteta}}\\
    \tilde\eta^{\ord}_{\fp}& \wH^2_{\fp, c1}(\cX^{\ord}_G, \cV\langle -D_G\rangle; 1+j, P_{1+j})& \text{Notation \ref{not:nuresord}}\\
    \underline\eta&H^1_{\m, \m}(\kappa_{\cU_1}, \kappa_{\cU_2})[\uPi]&\text{Def. \ref{def:ueta}}\\
    \end{tabu}$
\end{center}

\vspace{2ex}

The Asai $p$-adic $L$-function of \cite{grossiloefflerzerbesLfunct} is constructed as a pairing of the class $\eta^{(1-m)}_{\Iw(\pp_1)}$ with the pushforward of a $p$-adic family of Eisenstein series (specialising to $p$-adic modular forms with Iwahori level at $p$). The regulator computations take place at spherical level (as this is where the Euler system classes appear), therefore after recalling the compatibilities between spherical and $\pp_1$-Iwahori level classes in $\S4$ (both the \emph{classical} ones and the partial ordinary loci ones), the rest of the proof of the theorem deals with classes at spherical level. The class $\eta$ gives rise to a class in de Rham cohomology $\eta_\dR$. Proposition \ref{prop:extbyzero} lifts such class to a class $\eta_{\dR}^{(1-\ord)}$ over the $\pp_1$-ordinary locus compatibly to the lift $\eta^{(1-\ord)}$ of $\eta$ by Proposition \ref{prop:nu1ord2}.

The first step to compute the regulator pairing is to reduce it to a pairing in finite-polynomial cohomology; this is done in $\S6$, where the class $\eta_{\fp}$ is introduced. This class also has a lift over the $\pp_1$-ordinary locus $ \tilde\eta^{(1-\ord)}_{\fp}$, which translates the pairing in finite-polynomial cohomology over $\YY_H$ into a pairing in Gross fp-cohomology of the ordinary locus of the modular curve. This is good because the image of the syntomic Eisenstein class on the ordinary locus can be described explicitly as in \cite{KLZ20} (see Proposition \ref{prop:Eisexplicit}).

So all it is left to do in order to compute the pairing with such class is describe $ \tilde\eta^{(1-\ord)}_{\fp}$. In order to have liftings of both partial Frobenii, we restrict the class to the full ordinary locus to obtain $\tilde\eta^{\ord}_{\fp}$. This is where the \emph{Pozn\'an spectral sequence} comes into play, providing a link back to coherent cohomology: our class is \emph{represented by a coherent fp pair of degree $(1,1)$} given by $(\eta^\ord,\xi)$ where $\eta^\ord$ is the restriction of $ \eta^{(1-\ord)}$ to the full ordinary locus and $\xi$ is another class in coherent cohomology which is an $\alpha_1$-eigenvector for the partial Frobenius $\varphi_1$ and is annihilated by $U_{\pp_2}$ (see $\S8.3$). The dotted version of these classes $(\dot{\eta}^\ord_j,\dot{\xi})$ are the lifts from the BGG complex to the full de Rham complex.

The final step is to consider certain unit root splittings for the coefficients of the cohomology where our final classes belong: Corollary \ref{cor:cohexplicit} states that the map obtained by pairing with the pullback of $\dot{\eta}^\ord_j$ factors through the unit root splitting and (up to some factors) is given by pairing with the pullback of $ \eta^{(1-\ord)}$. The properties of $\xi$ and the coherent description of the Eisenstein class come together in the final computations, where the compatibility results of $\S4$ allow us to bring back the Iwahori level class $\eta^{(1-m)}_{\Iw(\pp_1)}$ and provide the link with the $p$-adic $L$-function.

To get the version of Theorem \ref{intro:regformula} \emph{in families}, Theorem \ref{thm:ERL}, (which then allows to specialise at critical points and obtain Theorem \ref{intro:BKIwapplication}), we use the final class $\underline{\eta}$ appearing in the table, which specialises (up to periods) to $\eta_{\Iw(\pp_1)}$ at certain classical points.

\newpage
\begin{landscape}
\vspace{15cm}

\begin{center}
    Relations between the classes at spherical level
\end{center}
\vspace{2ex}

\[
\begin{diagram}
\tilde\eta^{\ord}_{\fp}&&\tilde\eta^{(1-\ord)}_{\fp}&&  \eta_{\fp}&&\\
\text{\rotatebox[origin=c]{270}{$\in$}}  &&\text{\rotatebox[origin=c]{270}{$\in$}} &&  \text{\rotatebox[origin=c]{270}{$\in$}}&&\\
 \wH^2_{\fp, c1}(\cX^{\ord}_{G, \sph}, \cV_G\langle -D_G\rangle; 1+j, P_{1+j})&\lTo^{\quad\mathrm{res}\quad }&
 \wH^2_{\fp, c}(\cX_{G, \sph}^{(1-\ord)}, \cV_G\langle-D_G\rangle; 1+j, P_{1+j}) &
 \rTo^{\quad \text{ext-by-0}\quad} &
 H^2_{\fp, c}\left(\YY_{G, \sph}, \cV_G; 1+j, P_{1+j}\right) &&\\
   &&\dTo^{\mathrm{proj}} && \dTo^{\mathrm{proj}} & &\\
   &&H^2_{\rig, c}(\cX^{(1-\ord)}_{G, \sph, 0}, \cV_G\langle -D_G\rangle) &  \rTo^{\quad \text{ext-by-0}\quad} & H^2_{\dR}(X_{G, \sph}, \cV_G\langle-D_G\rangle) & \supset & D_p^{\As}(\Pi)\\
 &&\text{\rotatebox[origin=c]{90}{$\in$}} && \twoheaddownarrow &&\text{\rotatebox[origin=c]{90}{$\in$}}\\
  &&\eta_{\dR}^{(1-\ord)}&&\Gr^{(0,k_2+1)}&& \eta_{\dR}\\
 &&&& \text{\rotatebox[origin=c]{90}{$\cong$}} &&\\
 H^1_{c1}\Big(\cX_{G, \sph}^{\ord}, \omega^{(-k_1, k_2 + 2)}(-D_G)\Big)&\lTo^{\quad\mathrm{res}\quad }&H^1_c\left(\cX_{G, \sph}^{(1-\ord)}, \omega^{(-k_1, k_2 + 2)}(-D_G)\right) &  \rTo^{\quad \text{ext-by-0}\quad} &H^1\left(X_{G, \sph} / L, \omega^{(-k_1,k_2+2; w)}(-D_G)\right)  &&\\
\text{\rotatebox[origin=c]{90}{$\in$}}  &&\text{\rotatebox[origin=c]{90}{$\in$}} &&  \text{\rotatebox[origin=c]{90}{$\in$}}&&\\
\eta^{\ord}&&\eta^{(1-\ord)}&&  \eta&&
\end{diagram}
\]
\end{landscape}
%\vspace{10cm}

\newcommand{\noopsort}[1]{\relax}
%\bibliographystyle{../../amsalphaurl}
%\bibliography{../../references}
\providecommand{\bysame}{\leavevmode\hbox to3em{\hrulefill}\thinspace}
\providecommand{\MR}[1]{%
 MR \href{http://www.ams.org/mathscinet-getitem?mr=#1}{#1}.
}
\providecommand{\href}[2]{#2}
\newcommand{\articlehref}[2]{\href{#1}{#2}}

\end{document}